\documentclass[11pt]{article}

\usepackage{amsfonts}
\usepackage{amsmath,amsthm}
\usepackage{graphicx,epstopdf,epsfig,multirow,epic,bm}
\usepackage{color}
\usepackage{multicol}
\usepackage{algorithm}
\usepackage{algorithmic}
\usepackage{paralist}

\theoremstyle{definition}
\newtheorem{thm}{Theorem}

\theoremstyle{definition}
\newtheorem{defn}{Definition}
\theoremstyle{definition}
\newtheorem{rem}{Remark}

\theoremstyle{definition}
\newtheorem{problem}{Problem}

\theoremstyle{definition}

\theoremstyle{definition}
\newtheorem{example}{Example}
\theoremstyle{definition}

\DeclareGraphicsExtensions{.eps,.eps.gz}
\normalsize \rm
\usepackage{titlesec}

\newcommand{\qqed}{\hfill $\square$}
\newcommand{\paralled}{\hfill $\parallel$}

\usepackage[nottoc]{tocbibind} 

\begin{document}
\thispagestyle{empty}

\begin{center}
{\huge \textbf{New Techniques Based On Odd-Edge Total\\[8pt]
Colorings In Topological Cryptosystem}}\\[14pt]
{\large \textbf{Bing YAO},~ \textbf{Mingjun ZHANG}, ~ \textbf{Sihua YANG},~ \textbf{Guoxing WANG}}\\[8pt]
(\today)
\end{center}

\vskip 1cm

\pagenumbering{roman}
\tableofcontents

\newpage

\setcounter{page}{1}
\pagenumbering{arabic}

\thispagestyle{empty}

\begin{center}
{\huge \textbf{New Techniques Based On Odd-Edge Total\\[8pt]
Colorings In Topological Cryptosystem}}\\[14pt]
{\large \textbf{Bing Yao} $^{1,\dagger}$, \textbf{Mingjun Zhang} $^{2,3,4,\ddagger}$, \textbf{Sihua Yang} $^{3,\ast}$, \textbf{Guoxing Wang} $^{2,3,4,\diamond}$}\\[12pt]
{\small 1. College of Mathematics and Statistics, Northwest Normal University, Lanzhou, 730070 CHINA\\
2. China Northwest Center of Financial Research, Lanzhou University of Finance and Economics, Lanzhou 730020, CHINA\\
3. School of Information Engineering, Lanzhou University of Finance and Economics, Lanzhou 730020, CHINA\\
4. Key Laboratory of E-Business Technology and Application, Gansu Province, Lanzhou 730020, CHINA\\[6pt]
$^{\dagger}$ yybb918@163.com; $^{\ddagger}$ shuxue1998@163.com; $^{\ast}$ 731914872@qq.com; $^{\diamond}$ wanggx@lzufe.edu.cn\\[6pt]
}
\end{center}

\vskip 1cm

\begin{quote}
\textbf{Abstract:} For building up twin-graphic lattices towards topological cryptograph, we define four kinds of new odd-magic-type colorings: odd-edge graceful-difference total coloring, odd-edge edge-difference total coloring, odd-edge edge-magic total coloring, and odd-edge felicitous-difference total coloring in this article. Our RANDOMLY-LEAF-ADDING algorithms are based on adding randomly leaves to graphs for producing continuously graphs admitting our new odd-magic-type colorings. We use complex graphs to make caterpillar-graphic lattices and complementary graphic lattices, such that each graph in these new graphic lattices admits a uniformly $W$-magic total coloring. On the other hands, finding some connections between graphic lattices and integer lattices is an interesting research, also, is important for application in the age of quantum computer. We set up twin-type $W$-magic graphic lattices (as public graphic lattices vs private graphic lattices) and $W$-magic graphic-lattice homomorphism for producing more complex topological number-based strings.
\end{quote}

\vskip 1cm

\section{Introduction and preliminary}

\subsection{Researching background}

Lattice-based cryptography, as a new cryptosystem, has attracted much attention because of its great potential application value, since it is the use of conjectured hard problems on point lattices in $\mathbb{R}^n$ as the foundation for secure cryptographic systems, including apparent resistance quantum attacks, high asymptotic efficiency and parallelism, security under ``worst-case'' intractability assumptions, and solutions to long-standing open problems in cryptography said by Chris Peikert in \cite{Chris-Peikert-decade}. And moreover, Chris Peikert summarized that lattice-based cryptography possesses \emph{conjectured security} against quantum attacks, \emph{algorithmic simplicity, efficiency, and parallelism}, and \emph{strong security} guarantees from worst-case hardness.

In \cite{Peikert-Lattice-Cryptography-Internet2014} the author pointed: Lattice-based cryptography has been recognized for its many attractive properties, such as strong provable security guarantees and apparent resistance to quantum attacks, flexibility for realizing powerful tools like fully homomorphic encryption, and high asymptotic efficiency. He has given efficient and practical lattice-based protocols for key transport, encryption, and authenticated key exchange that are suitable as ``drop-in'' components for proposed Internet standards and other open protocols. The security of all proposals is provably based on the well-studied ``learning with errors over rings'' problem, and hence on the conjectured worst-case hardness of problems on ideal lattices (against quantum algorithms).

The authors in \cite{Tim-Vadim-Thomas-2012-978-3-642-33027} have presented such an alternative - a signature scheme whose security is derived from the hardness of lattice problems; it is based on recent theoretical advances in lattice-based cryptography and is highly optimized for practicability and use in embedded systems. The public and secret keys are roughly 12000 and 2000 bits long, while the signature size is approximately 9000 bits for a security level of around 100 bits.

In the articles \cite{Wang-Xu-Yao-2016, Wang-Xu-Yao-2017-Twin, Wang-Xu-Yao-2017-Matching} the authors have discussed some topics of topological coding, such as topological graph password and twin odd-graceful graphs for matching topological public-keys and topological private-keys in asymmetric cryptography. The authors, in \cite{Tian-Li-Peng-Yang-2021-102212}, use the topological graph to generate the honeywords, which is the first application of graphic labeling of topological coding in the honeywords generation. They propose a method to protect the \emph{hashed passwords} by using \emph{topological graphic sequences}.

\subsection{Examples from topological coding}

For introducing \emph{topological number-based strings} since they are made by Topcode-gpws (the abbreviation ``graphic passwords in topological coding''), we show an example as:

\begin{example}\label{exa:first-example-11}
A colored graph (also, Topcode-gpw) $B_1$ shown in Fig.\ref{fig:111-example} (a) corresponds a matrix $T_{code}(B_1)$ shown in Eq.(\ref{eqa:Topcode-gpw-B-1}), called a \emph{Topcode-matrix}, and this matrix distributes us $(30)!$ topological number-based strings like $s_1$ and $s_2$ shown in Eq.(\ref{eqa:B-1-number-based-string11}). As each topological number-based string $s_i$ ($i=1,2$) is a public-key, so it has its own private-key $s\,'_i$ shown in Eq.(\ref{eqa:B-1-number-based-string22}), and two topological number-based strings $s\,'_1$ and $s\,'_2$ are induced from the Topcode-matrix $T_{code}(B_{10})$ shown in Eq.(\ref{eqa:Topcode-gpw-B-1-22}), where the colored graph $B_{10}$ is shown in Fig.\ref{fig:111-example} (j).\qqed
\end{example}

\begin{equation}\label{eqa:B-1-number-based-string11}
{
\begin{split}
s_1&=7191136654117611915211241315191515017191020\\
s_2&=0219191701515191513421121591167114566311917\\
\end{split}}
\end{equation}

\begin{figure}[h]
\centering
\includegraphics[width=16.4cm]{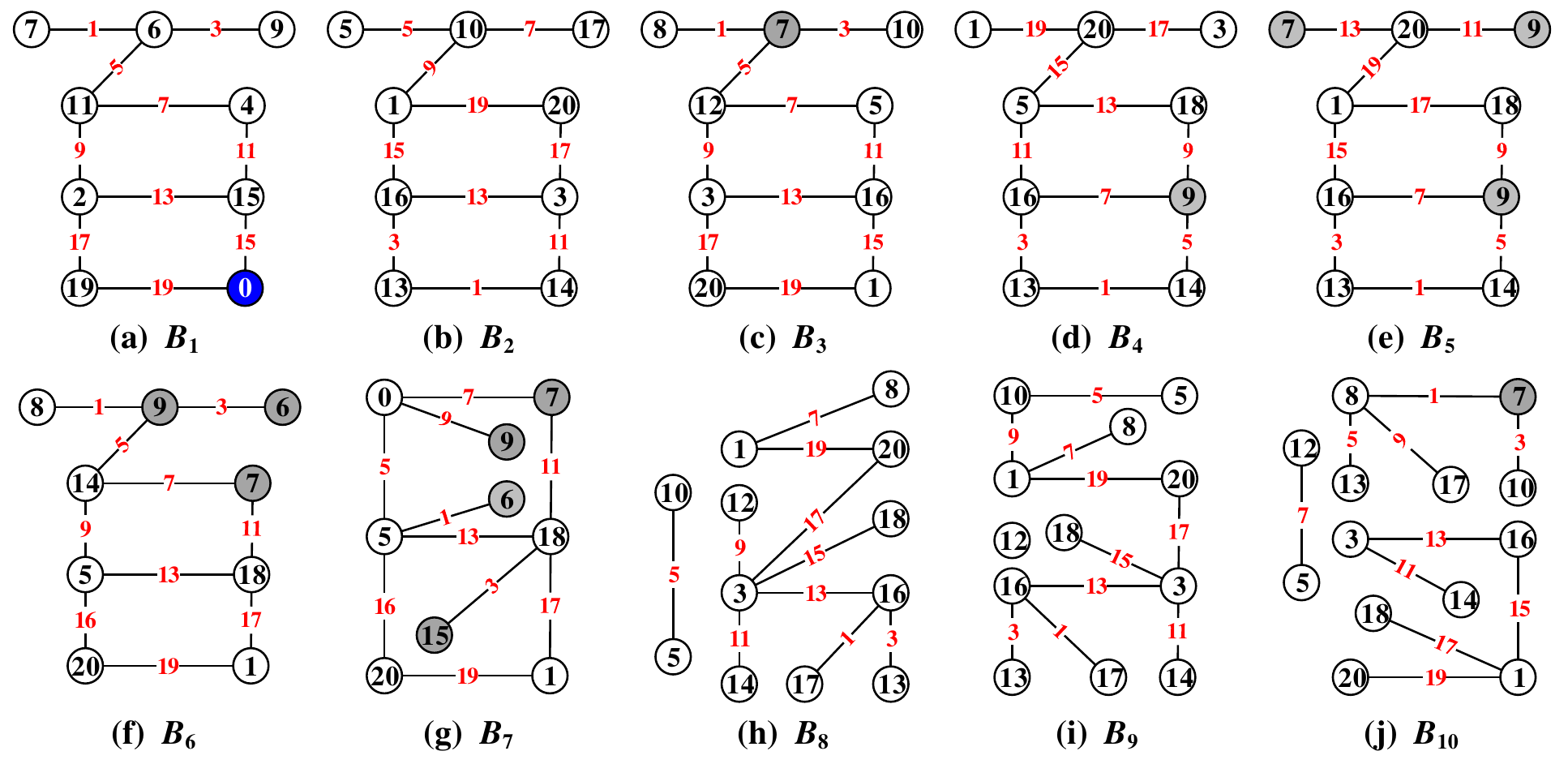}\\
\caption{\label{fig:111-example} {\small Ten Topcode-gpws.}}
\end{figure}

\begin{equation}\label{eqa:Topcode-gpw-B-1}
\centering
{
\begin{split}
T_{code}(B_1)= \left(
\begin{array}{cccccccccc}
6 & 6 & 6 & 11 & 2 & 4 & 15 & 0 & 2 & 0\\
1 & 3 & 5 & 7 & 9 & 11 & 13 & 15 & 17 & 19\\
7 & 9 & 11 & 4 & 11 & 15 & 2 & 15 & 19 & 19
\end{array}
\right)_{3\times 10}
\end{split}}
\end{equation}

\begin{equation}\label{eqa:B-1-number-based-string22}
{
\begin{split}
s\,'_1&=71108287551771312914311831316181516117201911\\
s\,'_2&=11911172018151631316311812914177137558381107\\
\end{split}}
\end{equation}
\begin{equation}\label{eqa:Topcode-gpw-B-1-22}
\centering
{
\begin{split}
T_{code}(B_{10})= \left(
\begin{array}{cccccccccc}
8 & 7 & 13 & 12 & 8 & 3 & 16 & 1 & 1 & 1\\
1 & 3 & 5 & 7 & 9 & 11 & 13 & 15 & 17 & 19\\
7 & 10 & 8 & 5 & 17 & 14 & 3 & 16 & 18 & 20
\end{array}
\right)_{3\times 10}
\end{split}}
\end{equation}

A \emph{topological authentication} consists of a topological structure matching $\langle B_1, B_{10}\rangle$ and a Topcode-matrix matching $\langle T_{code}(B_{1}), T_{code}(B_{10})\rangle$ and a group of topological number-based string matchings $\langle s_i,s\,'_i\rangle$ for $i=1,2,\dots ,n$, where each topological number-based string matching $\langle s_i,s\,'_i\rangle$ is obtained by a fixed rile $a_i$ implementing to two Topcode-matrices $T_{code}(B_{1})$ and $T_{code}(B_{10})$, respectively.
The procure
\begin{equation}\label{eqa:procure11-number-based-string}
B_1\rightarrow T_{code}(B_{1})\rightarrow_{a_i} s_i
\end{equation} is for encrypting a digital file $D_{ocu}$ by a topological number-based string $s_i$, get a encrypted file $s_i(D_{ocu})$. Next, one find another topological number-based string $s\,'_i$ from anther procure
\begin{equation}\label{eqa:procure22-number-based-string}
B_{10}\rightarrow T_{code}(B_{10})\rightarrow_{a_i} s\,'_i
\end{equation} and go to the topological authentication
\begin{equation}\label{eqa:a-topological-authentication}
\textbf{\textrm{T}}_{authen}=(\langle B_1, B_{10}\rangle, \langle T_{code}(B_{1}), T_{code}(B_{10})\rangle, \langle s_i,s\,'_i\rangle)
\end{equation} After finishing the above topological authentication (\ref{eqa:a-topological-authentication}), one can use $s\,'_i$ to decrypt the encrypted file $s_i(D_{ocu})$, such that $s\,'_i(s_i(D_{ocu}))=D_{ocu}$.

In the above Example \ref{exa:first-example-11}, there are the following problems:
\begin{asparaenum}[\textrm{\textbf{Pro}}-1. ]
\item Notice that two topological structures $B_1$ and $B_{10}$ are not isomorphic from each other, that is, $B_1\not \cong B_{10}$; two Topcode-matrices are not equivalent from each other, also $T_{code}(B_{1})\neq T_{code}(B_{10})$, and $s_i\neq s\,'_i$ for $i=1,2,\dots ,n$. The vertices of the colored graph $B_1$ are colored with the numbers of a set $f(B_1)=\{0,2,4,6,7,9,11,15,19\}$ and the vertices of the colored graph $B_{10}$ are colored with the numbers of another set $g(B_{10})=\{1,3,5,8,10,12,13,14,16,17,18,20\}$, such that $f(B_1)\cup g(B_{10})=\{0,1,\dots ,20\}$, which is in the Topcode-matrix matching $\langle T_{code}(B_{1}), T_{code}(B_{10})\rangle$. This case was introduced in \cite{Wang-Xu-Yao-2017-Twin}, two colored graphs $B_1$ and $B_{10}$ admits a so-called \emph{twin odd-graceful labeling}.
\item Each colored graph $B_{k}$ with $2\leq k\leq 9$ shown in Fig.\ref{fig:111-example} (b)-(i) can be used as a private-key corresponding to the public-key $B_1$. Thereby, one public-key corresponds two or more public-keys.
\item Finding the public-key $B_1$ is impossible from the public-key string $s_i$, that is, no way for $s_i\rightarrow T_{code}(B_{1})\rightarrow B_1$. Also, no way for $s\,'_i\rightarrow T_{code}(B_{10})\rightarrow B_{10}$.
\item In the topological structure matching $\langle B_1, B_{10}\rangle$, finding a private-key $B_{10}$ will meet the Graph isomorphic Problem, a NP-problem as known, since there are two or more private-keys like $B_{10}$ to match with a public-key like $B_1$.
\item Finding the coloring for a private-key will be facing thousands of colorings of topological coding. And no algorithm is for finding out all colorings for a graph having huge numbers of vertices and edges.
\end{asparaenum}

The above problems tell us that topological number-based strings will have applications in the age of quantum computers.

\begin{example}\label{exa:8888888888}
Let $S_{tring}(B_1)=\{s_1(43),s_2(43),\dots ,s_M(43)\}$ with $M=(30)!$ be the set of topological number-based strings generated from the Topcode-matrix $T_{code}(B_1)$ shown in Eq.(\ref{eqa:Topcode-gpw-B-1}), where each topological number-based string $s_i$ has $43$ bytes.

Then we have \emph{compound number-based strings} $L_r=s_{r,1}(43)s_{r,2}(43)\cdots s_{r,M}(43)$ for $r\in \{1,2,\dots $, $M!\}$ with longer bytes, such that each string $L_r$ has $43\cdot M=43\cdot (30)!$ bytes, in total. It is a proof for our techniques having broad application potential.\qqed
\end{example}

\subsection{Main works}

Motivated from lattice-based cryptography, the authors in \cite{Yao-Wang-Su-Sun-ITOEC2020, yao-sun-su-wang-matching-groups-zhao-2020, Yao-Wang-Liu-ice-flower-2020arXiv, Bing-Yao-2020arXiv, Bing-Yao-Hongyu-Wang-arXiv-2020-homomorphisms, Yao-Zhang-Sun-Mu-Sun-Wang-Wang-Ma-Su-Yang-Yang-Zhang-2018arXiv, Zhang-Yang-Yao-Frontiers-Computer-2021} have proposed \emph{graphic lattices} and shown many researching results on graphic lattices in topological coding.

In this article, we will make new RLA-algorithms for new colorings: odd-edge graceful-difference total coloring, odd-edge edge-difference total coloring, odd-edge edge-magic total coloring, and odd-edge felicitous-difference total coloring. Moreover, we will design RLA-algorithms for adding randomly leaves to graphs continuously. We will build up caterpillar-graphic lattices and complementary graphic lattices made by uniformly $W$-magic total colorings, and show some connections between graphic lattices and integer lattices, and analyze the complexity of graph lattices introduced here.

\subsection{Basic notations and definitions}

For simplicity and accuracy, we will apply the standard terminology and notation in \cite{Bondy-2008} and \cite{Gallian2021} in this article. All graphs menetioned here are simple, also, they have no loop and multiple-edge. Others are as follows:
\begin{asparaenum}[$\bullet$ ]
\item The notation $[a,b]$ indicates an integer set $\{a, a+1, \dots, b\}$ with integers $a, b$ holding $0\leq a<b$, and $[a, b]^o$ denotes an odd-set $\{m, m+2, \dots, n\}$ with odd integers $m, n$ with respect to $1\leq m< n$.
\item The number of elements of a set $X$ is written as $|X|$.
\item $N(u)$ is the set of vertices adjacent with a vertex $u$, and the number $\deg_G(u)=|N(u)|$ is called the \emph{degree} of the vertex $u$. The maximum degree $\Delta(G)=\max \{\deg_G(u):u\in V(G)\}$, and the minimum degree $\delta(G)=\min \{\deg_G(u):u\in V(G)\}$.
\item A \emph{leaf} is a vertex $x$ having its degree $\textrm{deg}(x)=1=|N(x)|$.
\item A $(p,q)$-graph having $p$ vertices and $q$ edges.
\item The sentence ``\emph{adding a leaf $w$ to a graph $G$}'' is an graph operation defined by adding a new vertex $w$ to $G$, and join $w$ with a vertex $x$ of $G$ by an edge $xw$, the resultant graph is denoted as $G+xw$, called leaf-added graph, such that $w$ is a leaf of $G+xw$.
\item Let $G$ and $H$ be two graphs. If $H=G+xy-uv$ for edge $uv\in E(G)$ and edge $xy\not\in E(G)$, then we call $H$ \emph{$e^-_{+}$-dual} of $G$, also, \emph{added-edge-removed graph}.
\end{asparaenum}

\begin{defn} \label{defn:traditional-lattice00}
A \emph{lattice} $\textrm{\textbf{L}}(\textbf{B})$ defined as
\begin{equation}\label{equ:traditional-lattice00}
\textrm{\textbf{L}}(\textbf{B}) =\{x_1\textbf{\textrm{b}}_1+x_2\textbf{\textrm{b}}_2+\cdots +x_n\textbf{\textrm{b}}_n : x_i \in Z\}
\end{equation} is a set of all integer combinations of $n$ linearly independent vectors of a \emph{base} $\textbf{B}=(\textbf{\textrm{b}}_1,\textbf{\textrm{b}}_2,\dots $, $ \textbf{\textrm{b}}_n)$ in $R^m$ with $n\leq m$, where $Z$ is the integer set, $m$ is the \emph{dimension} and $n$ is the \emph{rank} of the lattice, and $\textbf{B}$ is called \emph{lattice base}. Particularly, if each component $b_{k,j}$ of each vector $\textbf{\textrm{b}}_k=(b_{k,1},b_{k,2},\dots,b_{k,m})$ of the lattice base $\textrm{\textbf{L}}(\textbf{B})$ is an integer, we get an \emph{integer lattice}, denoted as $\textbf{\textrm{L}}(Z\textbf{\textrm{B}})$.\qqed
\end{defn}

In the view of geometry, a lattice is a set of discrete points with periodic structure in $R^m$. For no confusion, we call $\textrm{\textbf{L}}(\textbf{B})$ defined in Definition \ref{defn:traditional-lattice00} \emph{traditional lattice} in the following discussion.

\begin{defn}\label{defn:topcode-matrix-definition}
\cite{Yao-Sun-Zhao-Li-Yan-ITNEC-2017, Yao-Zhao-Zhang-Mu-Sun-Zhang-Yang-Ma-Su-Wang-Wang-Sun-arXiv2019} A \emph{Topcode-matrix} (or \emph{topological coding matrix}) is defined as
\begin{equation}\label{eqa:Topcode-matrix}
\centering
{
\begin{split}
T_{code}= \left(
\begin{array}{ccccc}
x_{1} & x_{2} & \cdots & x_{q}\\
e_{1} & e_{2} & \cdots & e_{q}\\
y_{1} & y_{2} & \cdots & y_{q}
\end{array}
\right)_{3\times q}=
\left(\begin{array}{c}
X\\
E\\
Y
\end{array} \right)=(X,~E,~Y)^{T}
\end{split}}
\end{equation} where \emph{v-vector} $X=(x_1, x_2, \dots, x_q)$, \emph{e-vector} $E=(e_1$, $e_2 $, $ \dots $, $e_q)$, and \emph{v-vector} $Y=(y_1, y_2, \dots, y_q)$ consist of non-negative integers $e_i$, $x_i$ and $y_i$ for $i\in [1,q]$. We say $T_{code}$ to be \emph{evaluated} if there exists a function $\theta$ such that $e_i=\theta(x_i,y_i)$ for $i\in [1,q]$, and call $x_i$ and $y_i$ to be the \emph{ends} of $e_i$, and $q$ the \emph{size} of $T_{code}$.\qqed
\end{defn}

\subsection{Particular trees and complex graphs}

Recall, a tree $T$ has that any pair of two vertices can be connected by a unique path, each vertex of degree one is a called a \emph{leaf} in $T$. If removing all leaves of a tree produces a path $P=u_1u_2\cdots u_n$ of $n$ vertices for $n\geq 1$, we call this tree \emph{caterpillar}, and call the path $P$ \emph{spine path} of the caterpillar. If removing some leaves of a tree produces a caterpillar, then we call this tree \emph{lobster}.

\vskip 0.4cm

\begin{defn} \label{defn:four-caterpillars-relations}
$^*$ There are four caterpillars $H$, $T$, $T^*$ and $G$ in the following paragraphs:

$\bullet$ Let $L_{eaf}(H)$ be the set of all leaves of the caterpillar $H$. The delation of all leaves of $H$ makes a graph, denoted by $H-L_{eaf}(H)$. By the definition of a caterpillar, we have $H-L_{eaf}(H)=P=u_1u_2\cdots u_n$, where $P$ is the spine path of $H$. Let the set of leaves adjacent with a vertex $u_i\in V(P)$ be denoted as $L_{eaf}(u_i)=\{v_{i, j}:j\in [1, a_i]\}$ for $i\in [1,n]$, where integer $a_i\geq 0$. So the leaf set $L_{eaf}(H)=\bigcup ^n_{i=1}L_{eaf}(u_i)$, such that $V(H)=V(P)\cup L_{eaf}(H)$, the caterpillar $H$ has $m$ leaves, where $m=\sum^n_{i=1}|L_{eaf}(u_i)|=\sum^n_{i=1}a_i$.

\vskip 0.2cm

$\bullet$ The caterpillar $T$ has the spine path $P\,'=x_1x_2\cdots x_n$ for $n\geq 1$, and each vertex $x_i$ of the spine path $P\,'$ has the leaf set $L_{eaf}(x_i)=\{y_{i, j}:j\in [1, b_i]\}$ with $i\in [1,n]$. If the leaf sets of two caterpillars $H$ and $T$ hold $|L_{eaf}(u_i)|+|L_{eaf}(x_i)|=M$ for $i\in [1,n]$, we say two caterpillars $H$ and $T$ to be \emph{uniform $M$-leaf complementary trees}.

\vskip 0.2cm

$\bullet$ Let $P^*=x\,'_1x\,'_2\cdots x\,'_n$~$(n\geq 1)$ be the spine path of the caterpillar $T^*$, and each vertex $x\,'_i$ of the spine path $P^*$ has the leaf set $L_{eaf}(x\,'_i)=\{y\,'_{i, j}:j\in [1, b\,'_i]\}$ with $i\in [1,n]$. Suppose $x\,'_{j_1},x\,'_{j_2},\cdots ,x\,'_{j_n}$ is a permutation of vertices $x\,'_1,x\,'_2,\cdots ,x\,'_n$, if the leaf sets of two caterpillars $H$ and $T^*$ hold $|L_{eaf}(u_i)|+|L_{eaf}(x\,'_{j_i})|=M$ for $i\in [1,n]$, we call two caterpillars $H$ and $T^*$ to be \emph{$M$-leaf complementary trees}.

\vskip 0.2cm

$\bullet$ Assume that $P\,''=s_1s_2\cdots s_n$ is the spine path of the caterpillar $G$, and each vertex $s_i$ of the spine path of $G$ has its own leaf set $L_{eaf}(s_i)=\{t_{i, j}:j\in [1, c_i]\}$ for $i\in [1,n]$. If the leaf sets of three caterpillars $H$, $T$ and $G$ satisfy $|L_{eaf}(u_i)|+|L_{eaf}(x_i)|=|L_{eaf}(s_i)|$ for $i\in [1,n]$, then the caterpillar $G$ is called \emph{universal graph} of each of two caterpillars $H$ and $T$, and two caterpillars $H$ and $T$ are $G$-complementary trees about the caterpillar $G$. By the graph operation of view, we coincide two spine paths of two caterpillars $H$ and $T$ into one, and then get the caterpillar $G$.\qqed
\end{defn}

\vskip 0.4cm

Computing the number of caterpillars obtained by adding $m$ leaves will meet the Integer Partition Problem, this is not an easy work.

\begin{defn}\label{defnc:complex-graph-definition}
\cite{Wang-Yao-Su-Wanjia-Zhang-2021-IMCEC} A \emph{complex graph} $G$ has its own vertex set $V(G)=X^{\circ}\cup Y^{\square}$ with $X^{\circ}\cap Y^{\square}=\emptyset$, $X^{\circ}\neq \emptyset$ and $Y^{\square}\neq \emptyset$, such that the \emph{degree} $\textrm{deg}_G(u)\geq 0$ for each vertex $u\in X^{\circ}$, and the \emph{image-degree} $\textrm{deg}_G(v)< 0$ for each vertex $v\in Y^{\square}$. Moreover, a vertex $x$ of the complex graph $G$ is adjacent with $m$ leaves, then we define the \emph{leaf-degree} $l_{eaf}(x)=m$ if $x\in X^{\circ}$, and the \emph{leaf-image-degree} $l_{eaf}(x)=-m=mi^2$ if $x \in Y^{\square}$, where $i^2=-1$.\qqed
\end{defn}

In Fig.\ref{fig:maomao-degree-sequence}, each vertex in the spine path of the caterpillar $A_1$ has its own leaf-degree $l_{eaf}(u_1)=7$, $l_{eaf}(u_2)=3$, $l_{eaf}(u_3)=0$, $l_{eaf}(u_4)=3$, $l_{eaf}(u_5)=0$ and $l_{eaf}(u_6)=6$, so the caterpillar $A_1$ has its own leaf-degree sequence $d_{\textrm{leaf}}(A_1)=(7,3,0,3,0,6)$. Each vertex in the spine path of the caterpillar $A_2$ has its own leaf-degree $l_{eaf}(x_1)=7$, $l_{eaf}(x_2)=-3$, $l_{eaf}(x_3)=0$, $l_{eaf}(x_4)=-3$, $l_{eaf}(x_5)=0$ and $l_{eaf}(x_6)=6$, then the caterpillar $A_2$ has its own leaf-degree sequence $d_{\textrm{leaf}}(A_2)=(7,-3,0,-3,0,6)$.

\begin{figure}[h]
\centering
\includegraphics[width=16.4cm]{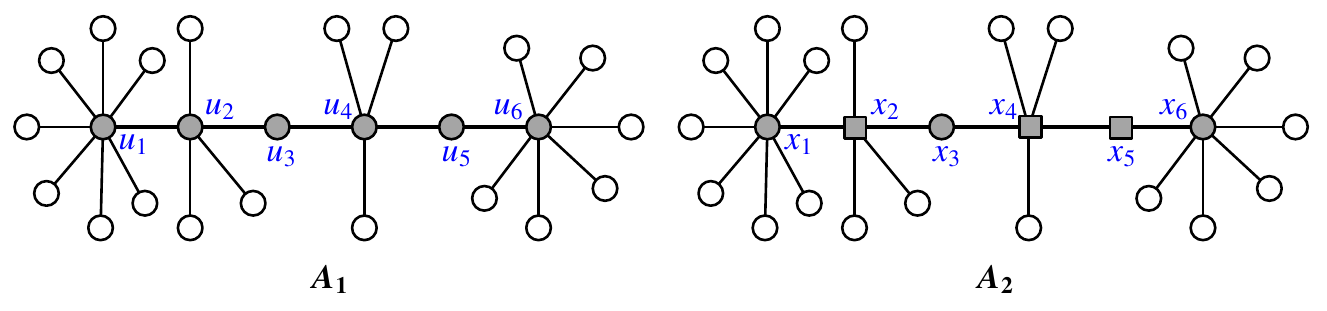}\\
\caption{\label{fig:maomao-degree-sequence} {\small Two caterpillars $A_1$ and $A_2$ have the same topological structure, that is, $A_1\cong A_2$.}}
\end{figure}

In Fig.\ref{fig:12-maomao-degree-sequence}, each vertex in the spine path of the caterpillar $B_1$ has its own leaf-degree $l_{eaf}(v_1)=1$, $l_{eaf}(v_2)=5$, $l_{eaf}(v_3)=8$, $l_{eaf}(v_4)=5$, $l_{eaf}(v_5)=8$ and $l_{eaf}(v_6)=2$, so the caterpillar $B_1$ has its own leaf-degree sequence $d_{\textrm{leaf}}(B_1)=(1,5,8,5,8,2)$. Each vertex in the spine path of the caterpillar $B_2$ has its own leaf-degree $l_{eaf}(y_1)=-2$, $l_{eaf}(y_2)=8$, $l_{eaf}(y_3)=5$, $l_{eaf}(y_4)=8$, $l_{eaf}(y_5)=5$ and $l_{eaf}(y_6)=-1$, then the caterpillar $B_2$ has its own leaf-degree sequence $d_{\textrm{leaf}}(B_2)=(-2,8,5,8,5,-1)$.

It is noticeable, Fig.\ref{fig:maomao-degree-sequence} and Fig.\ref{fig:12-maomao-degree-sequence} show us two groups of isomorphic caterpillars, that is, $A_1\cong A_2$ and $B_1\cong B_2$; the caterpillar $A_1$ and the caterpillar $B_1$ are the uniformly $8$-leaf complement trees; and the caterpillar $A_2$ and the caterpillar $B_2$ are the $8$-leaf complement trees, since there are matchings $y_1\leftrightarrow x_6$, $y_2\leftrightarrow x_3$, $y_3\leftrightarrow x_2$, $y_4\leftrightarrow x_5$, $y_5\leftrightarrow x_4$ and $y_5\leftrightarrow x_1$.

\begin{figure}[h]
\centering
\includegraphics[width=16.4cm]{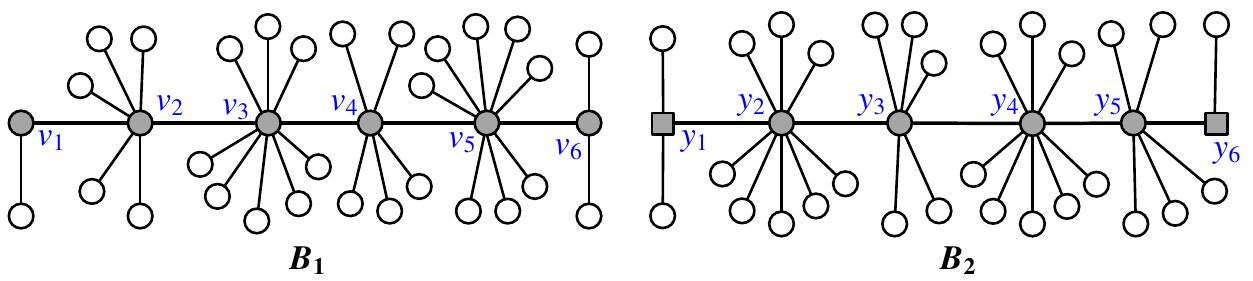}\\
\caption{\label{fig:12-maomao-degree-sequence} {\small Two caterpillars $A_1$ and $B_1$ are the uniformly $8$-leaf complement trees, two caterpillars $A_2$ and $B_2$ are the $8$-leaf complement trees, and $B_1\cong B_2$, where two caterpillars $A_1$ and $A_2$ are shown in Fig.\ref{fig:maomao-degree-sequence}.}}
\end{figure}

\begin{problem}\label{problem:99999}
Suppose that a caterpillar $H$ (as a public-key) and another caterpillar $T$ (as a private-key) have the spine paths of the same length, \textbf{find} the conditions (as a topological authentication) if these two caterpillars are $M$-leaf complement trees defined in Definition \ref{defn:four-caterpillars-relations}.
\end{problem}

\begin{problem}\label{problem:complement-complex-graphic-matching}
Suppose that a complex graph $G$ (as a public-key) and another complex graph $J$ (as a private-key) have the same number of vertices. The complex graph $G$ has its own degree sequence $\textrm{deg}(G)=(\textrm{deg}_G(u_1),\textrm{deg}_G(u_2),\dots ,\textrm{deg}_G(u_n))$, and the complex graph $J$ has its own degree sequence $\textrm{deg}(J)=(\textrm{deg}_J(v_1),\textrm{deg}_J(v_2)$, $\dots $, $\textrm{deg}_J(v_n))$. If there is a constant $M$, such that $\textrm{deg}_G(u_i)+\textrm{deg}_J(v_i)=M$ for $i\in [1,n]$, we call two complex graphs $G$ and $J$ \emph{uniformly $M$-complement complex graphic matching}, denoted as $M_{comp}\langle G,J \rangle$. \textbf{Characterize} each uniformly $M$-complement complex graphic matching $M_{comp}\langle G,J \rangle$.
\end{problem}

\section{New colorings and dual-type colorings}

\subsection{Basic labelings and colorings}

\begin{defn} \label{defn:basic-W-type-labelings}
\cite{Gallian2021, Yao-Sun-Zhang-Mu-Sun-Wang-Su-Zhang-Yang-Yang-2018arXiv, Bing-Yao-Cheng-Yao-Zhao2009, Zhou-Yao-Chen-Tao2012} Suppose that a connected $(p,q)$-graph $G$ admits a mapping $\theta:V(G)\rightarrow \{0,1,2,\dots \}$. For each edge $xy\in E(G)$, the induced edge color is defined as $\theta(xy)=|\theta(x)-\theta(y)|$. Write vertex color set by $\theta(V(G))=\{\theta(u):u\in V(G)\}$, and edge color set by
$\theta(E(G))=\{\theta(xy):xy\in E(G)\}$. There are the following constraint conditions:

\textbf{C-1.} $|\theta(V(G))|=p$;

\textbf{C-2.} $\theta(V(G))\subseteq [0,q]$, $\min \theta(V(G))=0$;

\textbf{C-3.} $\theta(V(G))\subset [0,2q-1]$, $\min \theta(V(G))=0$;

\textbf{C-4.} $\theta(E(G))=\{\theta(xy):xy\in E(G)\}=[1,q]$;

\textbf{C-5.} $\theta(E(G))=\{\theta(xy):xy\in E(G)\}=[1,2q-1]^o$;

\textbf{C-6.} $G$ is a bipartite graph with vertex bipartition $(X,Y)$ such that $\max\{\theta(x):x\in X\}< \min\{\theta(y):y\in Y\}$ ($\max \theta(X)<\min \theta(Y)$ for short);

\textbf{C-7.} $G$ is a tree having a perfect matching $M$ holding $\theta(x)+\theta(y)=q$ for each matching edge $xy\in M$; and

\textbf{C-8.} $G$ is a tree having a perfect matching $M$ holding $\theta(x)+\theta(y)=2q-1$ for each matching edge $xy\in M$.

\noindent \textbf{Then}:
\begin{asparaenum}[\textbf{\textrm{Lab}}-1.]
\item A \emph{graceful labeling} $\theta$ satisfies C-1, C-2 and C-4 at the same time.
\item A \emph{set-ordered graceful labeling} $\theta$ holds C-1, C-2, C-4 and C-6 true.
\item A \emph{strongly graceful labeling} $\theta$ holds C-1, C-2, C-4 and
C-7 true.
\item A \emph{set-ordered strongly graceful labeling} $\theta$ holds C-1, C-2, C-4, C-6 and C-7 true.
\item An \emph{odd-graceful labeling} $\theta$ holds C-1, C-3 and C-5 true.
\item A \emph{set-ordered odd-graceful labeling} $\theta$ abides C-1, C-3, C-5 and C-6.
\item A \emph{strongly odd-graceful labeling} $\theta$ holds C-1, C-3, C-5 and C-8, simultaneously.
\item A \emph{set-ordered strongly odd-graceful labeling} $\theta$ holds C-1, C-3, C-5, C-6 and C-8 true.\qqed
\end{asparaenum}
\end{defn}

\begin{defn} \label{defn:basic-W-type-labelings11}
$^*$ In Definition \ref{defn:basic-W-type-labelings}, if $|\theta(V(G))|<p$ holds true, we get

(i) A \emph{graceful coloring} $\theta$ satisfies C-2 and C-4 defined in Definition \ref{defn:basic-W-type-labelings}.

(ii) A \emph{set-ordered graceful coloring} $\theta$ satisfies C-2, C-4 and C-6 defined in Definition \ref{defn:basic-W-type-labelings}.

(iii) An \emph{odd-graceful coloring} $\theta$ satisfies C-3 and C-5 defined in Definition \ref{defn:basic-W-type-labelings}.

(ii) A \emph{set-ordered odd-graceful coloring} $\theta$ satisfies C-3, C-5 and C-6 defined in Definition \ref{defn:basic-W-type-labelings}.\qqed
\end{defn}

\begin{defn} \label{defn:new-graceful-strongly-colorings}
\cite{Bing-Yao-2020arXiv} Suppose that a connected $(p,q)$-graph $G$ ($\neq K_p$) admits a total coloring $f:V(G)\cup E(G)\rightarrow [1,M]$, and there are $f(x)=f(y)$ for some pairs of vertices $x,y\in V(G)$. Write $f(S)=\{f(w):w\in S\}$ for a non-empty set $S\subseteq V(G)\cup E(G)$ and let $k$ be a fixed positive integer. There are the following constraint conditions:
\begin{asparaenum}[(1$^\ast$)]
\item \label{27TProper01} $|f(V(G))|< p$;
\item \label{27TProper02} $|f(E(G))|=q$;
\item \label{27TGraceful-001} $f(V(G))\subseteq [1,M]$, $\min f(V(G))=1$;
\item \label{27TOdd-graceful-001} $f(V(G))\subset [1,2q+1]$, $\min f(V(G))=1$;
\item \label{27TGraceful-002} $f(E(G))=[1,q]$;
\item \label{27Tmodulo-01} $f(E(G))=[0,q-1]$;
\item \label{27TOdd-graceful-002} $f(E(G))=[1,2q-1]^o$;
\item \label{27Teven-edge-set-11} $f(E(G))=[2, 2q]^e$;
\item \label{27Tsequential-edge-set} $f(E(G))=[c,c+q-1]$;
\item \label{27Tgraceful-002} $f(uv)=|f(u)-f(v)|$ for each edge $uv\in E(G)$;
\item \label{27Tfelicitous-002} $f(uv)=f(u)+f(v)$ for each edge $uv\in E(G)$;
\item \label{27Tedge-labels-even-odd} For each edge $uv\in E(G)$, $f(uv)=f(u)+f(v)$ when $f(u)+f(v)$ is even, and $f(uv)=f(u)+f(v)+1$ when $f(u)+f(v)$ is odd;
\item \label{27Tmodulo-00} $f(uv)=f(u)+f(v)~(\textrm{mod}~q)$ for each edge $uv\in E(G)$;
\item \label{27Tmodulo-11} $f(uv)=f(u)+f(v)~(\textrm{mod}~2q)$ for each edge $uv\in E(G)$;
\item \label{27Tedge-difference} $f(uv)+|f(u)-f(v)|=k$ for each edge $uv\in E(G)$;
\item \label{27Tgraceful-difference} $\big |f(uv)-|f(u)-f(v)|\big |=k$ for each edge $uv\in E(G)$;
\item \label{27Tfelicitous-differences} $|f(u)+f(v)-f(uv)|=k$ for each edge $uv\in E(G)$;
\item \label{27Tedge-magic} $f(u)+f(uv)+f(v)=k$ for each edge $uv\in E(G)$;
\item \label{27Tmodulo-ordered} There exists an integer $k$ so that $\min \{f(u),f(v)\}\leq k <\max\{f(u),f(v)\}$ for each edge $uv\in E(G)$; and
\item \label{27TSet-ordered} $(X,Y)$ is the bipartition of a bipartite graph $G$ such that $\max f(X)< \min f(Y)$.
\end{asparaenum}
\textbf{A \emph{$W$-type} coloring $f$ is one of the following colorings}:
\begin{asparaenum}[\textrm{TCL}-1. ]
\item An \emph{edge-gracefully total coloring} if (\ref{27TProper01}$^\ast$), (\ref{27TProper02}$^\ast$) and (\ref{27TGraceful-002}$^\ast$) hold true.
\item A \emph{set-ordered edge-gracefully total coloring} if (\ref{27TProper01}$^\ast$), (\ref{27TProper02}$^\ast$), (\ref{27TGraceful-002}$^\ast$) and (\ref{27TSet-ordered}$^\ast$) hold true.

\item An \emph{edge-odd-gracefully total coloring} if (\ref{27TProper01}$^\ast$), (\ref{27TOdd-graceful-001}$^\ast$) and (\ref{27TOdd-graceful-002}$^\ast$) hold true.
\item A \emph{set-ordered edge-odd-gracefully total coloring} if (\ref{27TProper01}$^\ast$), (\ref{27TOdd-graceful-001}$^\ast$), (\ref{27TOdd-graceful-002}$^\ast$) and (\ref{27TSet-ordered}$^\ast$) hold true.

\item A \emph{gracefully total coloring} if (\ref{27TProper01}$^\ast$), (\ref{27TGraceful-001}$^\ast$), (\ref{27TGraceful-002}$^\ast$) and (\ref{27Tgraceful-002}$^\ast$) hold true.
\item A \emph{set-ordered gracefully total coloring} if (\ref{27TProper01}$^\ast$), (\ref{27TGraceful-001}$^\ast$), (\ref{27TGraceful-002}$^\ast$), (\ref{27Tgraceful-002}$^\ast$) and (\ref{27TSet-ordered}$^\ast$) hold true.
\item An \emph{odd-gracefully total coloring} if (\ref{27TProper01}$^\ast$), (\ref{27TOdd-graceful-001}$^\ast$), (\ref{27TOdd-graceful-002}$^\ast$) and (\ref{27Tgraceful-002}$^\ast$) hold true.
\item A \emph{set-ordered odd-gracefully total coloring} if (\ref{27TProper01}$^\ast$), (\ref{27TOdd-graceful-001}$^\ast$), (\ref{27TOdd-graceful-002}$^\ast$), (\ref{27Tgraceful-002}$^\ast$) and (\ref{27TSet-ordered}$^\ast$) hold true.
\item A \emph{felicitous total coloring} if (\ref{27TGraceful-001}$^\ast$), (\ref{27Tmodulo-00}$^\ast$) and (\ref{27Tmodulo-01}$^\ast$) hold true.
\item A \emph{set-ordered felicitous total coloring} if (\ref{27TGraceful-001}$^\ast$),(\ref{27Tmodulo-00}$^\ast$), (\ref{27Tmodulo-01}$^\ast$) and (\ref{27TSet-ordered}$^\ast$) hold true.
\item An \emph{odd-elegant total coloring} if (\ref{27TOdd-graceful-001}$^\ast$), (\ref{27Tmodulo-11}$^\ast$) and (\ref{27TOdd-graceful-002}$^\ast$) hold true.
\item A \emph{set-ordered odd-elegant total coloring} if (\ref{27TOdd-graceful-001}$^\ast$), (\ref{27Tmodulo-11}$^\ast$), (\ref{27TOdd-graceful-002}$^\ast$) and (\ref{27TSet-ordered}$^\ast$) hold true.
\item A \emph{harmonious total coloring} if (\ref{27TGraceful-001}$^\ast$), (\ref{27Tmodulo-00}$^\ast$) and (\ref{27Tmodulo-01}$^\ast$) hold true, and when $G$ is a tree, exactly one edge color may be used on two vertices.
\item A \emph{set-ordered harmonious total coloring} if (\ref{27TGraceful-001}$^\ast$), (\ref{27Tmodulo-00}$^\ast$), (\ref{27Tmodulo-01}$^\ast$) and (\ref{27TSet-ordered}$^\ast$) hold true.
\item A \emph{strongly harmonious total coloring} if (\ref{27TGraceful-001}$^\ast$), (\ref{27Tmodulo-00}$^\ast$), (\ref{27Tmodulo-01}$^\ast$) and (\ref{27Tmodulo-ordered}$^\ast$) hold true.
\item A \emph{properly even harmonious total coloring} if (\ref{27TOdd-graceful-001}$^\ast$), (\ref{27Teven-edge-set-11}$^\ast$) and (\ref{27Tmodulo-11}$^\ast$) hold true.
\item A \emph{$c$-harmonious total coloring} if (\ref{27TGraceful-001}$^\ast$), (\ref{27Tfelicitous-002}$^\ast$) and (\ref{27Tsequential-edge-set}$^\ast$) hold true.
\item An \emph{even sequential harmonious total coloring} if (\ref{27TOdd-graceful-001}$^\ast$), (\ref{27Tedge-labels-even-odd}$^\ast$) and (\ref{27Teven-edge-set-11}$^\ast$) hold true.
\item A \emph{pan-harmonious total coloring} if (\ref{27TProper02}$^\ast$) and (\ref{27Tfelicitous-002}$^\ast$) hold true.
\item An \emph{edge-magic total coloring} if (\ref{27Tedge-magic}$^\ast$) holds true.
\item A \emph{set-ordered edge-magic total coloring} if (\ref{27Tedge-magic}$^\ast$) and (\ref{27TSet-ordered}$^\ast$) hold true.
\item A \emph{gracefully edge-magic total coloring} if (\ref{27TGraceful-002}$^\ast$) and (\ref{27Tedge-magic}$^\ast$) hold true.
\item A \emph{set-ordered graceful edge-magic total coloring} if (\ref{27TGraceful-002}$^\ast$), (\ref{27Tedge-magic}$^\ast$) and (\ref{27TSet-ordered}$^\ast$) hold true.
\item An \emph{edge-difference total coloring} if (\ref{27Tedge-difference}$^\ast$) holds true.
\item A \emph{set-ordered edge-difference total coloring} if (\ref{27Tedge-difference}$^\ast$) and (\ref{27TSet-ordered}$^\ast$) hold true.
\item A \emph{graceful edge-difference total coloring} if (\ref{27TGraceful-002}$^\ast$) and (\ref{27Tedge-difference}$^\ast$) hold true.
\item A \emph{set-ordered graceful edge-difference total coloring} if (\ref{27TGraceful-002}$^\ast$), (\ref{27Tedge-difference}$^\ast$) and (\ref{27TSet-ordered}$^\ast$) hold true.
\item A \emph{felicitous-difference total coloring} if (\ref{27TProper01}$^\ast$), (\ref{27TProper02}$^\ast$) and (\ref{27Tfelicitous-differences}$^\ast$) hold true.
\item A \emph{set-ordered felicitous-difference total coloring} if (\ref{27TProper01}$^\ast$), (\ref{27TProper02}$^\ast$), (\ref{27Tfelicitous-differences}$^\ast$) and (\ref{27TSet-ordered}$^\ast$) hold true.
\item An \emph{graceful-difference total coloring} if (\ref{27Tgraceful-difference}$^\ast$) holds true.
\item A \emph{set-ordered graceful-difference total coloring} if (\ref{27Tgraceful-difference}$^\ast$) and (\ref{27TSet-ordered}$^\ast$) hold true.
\item An \emph{edge-graceful graceful-difference total coloring} if (\ref{27TGraceful-002}$^\ast$) and (\ref{27Tgraceful-difference}$^\ast$) hold true.
\item A \emph{set-ordered edge-graceful graceful-difference total coloring} if (\ref{27TGraceful-002}$^\ast$), (\ref{27Tgraceful-difference}$^\ast$) and (\ref{27TSet-ordered}$^\ast$) hold true.\qqed
\end{asparaenum}
\end{defn}

\subsection{New labelings and colorings}

For the convenience of statement, the word ``magic-type'' is as the same as the word ``$W$-magic'' in the following discussion.

\begin{defn}\label{defn:group-total-labelings-definition}
$^*$ Let $G$ be a bipartite $(p,q)$-graph with vertex bipartition $(X,Y)$, then $V(G)=X\cup Y$ with $X\cap Y=\emptyset$. There are four labelings defined as follows:

(i) A \emph{set-ordered odd-edge edge-magic total labeling} is a mapping $f:V(G)\cup E(G)\rightarrow [0,2q-1]$, such that $f(x)\neq f(y)$ for $x,y\in V(G)$, and the set-ordered restriction $\max f(X)<\min f(Y)$ holds true, and the edge color set $f(E(G))=[1,2q-1]^o$, as well as each edge $uv\in E(G)$ holds $f(u)+f(uv)+f(v)=k_1$, where $k_1$ is a positive integer.

(ii) A \emph{set-ordered odd-edge edge-difference total labeling} is a mapping $g:V(G)\cup E(G)\rightarrow [0,2q-1]$, such that $g(x)\neq g(y)$ for $x,y\in V(G)$, the set-ordered restriction $\max g(X)<\min g(Y)$ holds true, and the edge color set $g(E(G))=[1,2q-1]^o$, as well as each edge $uv\in E(G)$ satisfies $g(uv)+\big |g(u)-g(v)\big |=k_2$, where $k_2$ is a positive integer.

(iii) A \emph{set-ordered odd-edge felicitous-difference total labeling} is a mapping $h:V(G)\cup E(G)\rightarrow [0,2q-1]$, such that $h(x)\neq h(y)$ for $x,y\in V(G)$, the set-ordered restriction $\max h(X)<\min h(Y)$ holds true, and the edge color set $h(E(G))=[1,2q-1]^o$, as well as each edge $uv\in E(G)$ satisfies $\big |h(u)+h(v)-h(uv)\big |=k_3$, where $k_3$ is a non-negative integer.

(vi) A \emph{set-ordered odd-edge graceful-difference total labeling} is a mapping $\alpha:V(G)\cup E(G)\rightarrow [0,2q-1]$, such that $\alpha(x)\neq \alpha(y)$ for $x,y\in V(G)$, the set-ordered restriction $\max \alpha(X)<\min \alpha(Y)$ holds true, and the edge color set $\alpha(E(G))=[1,2q-1]^o$, as well as each edge $uv\in E(G)$ satisfies $\big ||\alpha(u)-\alpha(v)|-\alpha(uv)\big |=k_4$, where $k_4$ is a non-negative integer.\qqed
\end{defn}

\begin{defn} \label{defn:group-total-colorings-definition}
$^*$ Let ``$W$-magic'' be one of edge-magic, edge-difference, felicitous-difference, graceful-difference. We will obtain four \emph{odd-edge $W$-magic total labelings} if we remove the restriction ``set-ordered'' from Definition \ref{defn:group-total-labelings-definition}. If we allow that there is at least a pair of vertices colored with the same color in Definition \ref{defn:group-total-labelings-definition}, we will obtain four \emph{odd-edge $W$-magic total colorings} (see Fig.\ref{fig:4magice-total-colorings}). \qqed
\end{defn}

\begin{example}\label{exa:8888888888}
Fig.\ref{fig:4magice-total-colorings} is for illustrating Definition \ref{defn:group-total-labelings-definition} and Definition \ref{defn:group-total-colorings-definition}, we can see:

(a) The graph $M_0$ admits a set-ordered odd-graceful coloring $f_0$, since there are two vertices colored with 25. And $f_0(E(M_0))=\{f_0(xy)=f_0(y)-f_0(x):xy\in E(M_0)\}=[1,31]^o$.

(b) The graph $M_1$ admits a set-ordered odd-edge edge-magic total coloring $f_1$, since there are two vertices colored with 25. Each edge $xy\in E(M_1)$ holds $f_1(x)+f_1(xy)+f_1(y)=42$ true.

(c) The graph $M_2$ admits a set-ordered odd-edge edge-difference total coloring $f_2$, since there are two vertices colored with 6. Each edge $xy\in E(M_2)$ holds $f_2(xy)+|f_2(x)-f_2(y)|=f_2(xy)+f_2(y)-f_2(x)=32$ true.

(d) The graph $M_3$ admits a set-ordered odd-edge felicitous-difference total coloring $f_3$, since there are two vertices colored with 25. Each edge $xy\in E(M_3)$ holds $|f_3(x)+f_3(y)-f_3(xy)|=10$ true.

(e) The graph $M_4$ admits a set-ordered odd-edge graceful-difference total coloring $f_4$, since there are two vertices colored with 17. Each edge $xy\in E(M_4)$ holds $\big ||f_4(x)-f_4(y)|-f_4(xy)\big |=[f_4(y)-f_4(x)]-f_4(xy)=0$ true.\qqed
\end{example}

\begin{figure}[h]
\centering
\includegraphics[width=16.4cm]{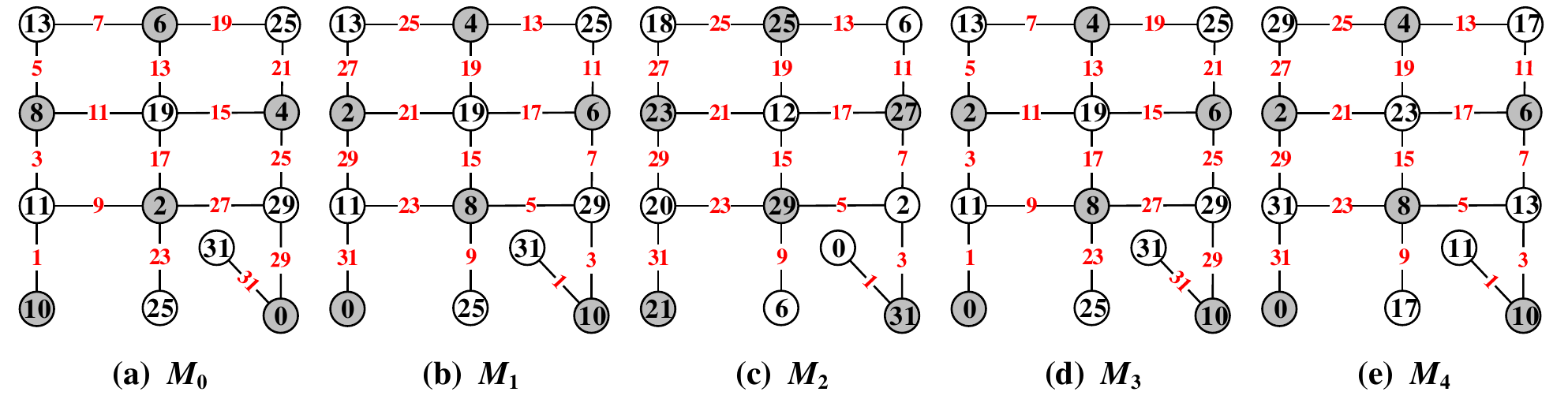}\\
\caption{\label{fig:4magice-total-colorings}{\small A scheme for illustrating Definition \ref{defn:group-total-labelings-definition} and Definition \ref{defn:group-total-colorings-definition}.}}
\end{figure}

\begin{example}\label{exa:8888888888}
Fig.\ref{fig:basic-definition-group} is for illustrating Definition \ref{defn:basic-W-type-labelings}, Definition \ref{defn:group-total-labelings-definition} and Definition \ref{defn:group-total-colorings-definition}, there are:

(a) The graph $R_1$ admits a set-ordered graceful labeling $g_1$, such that the edge color set $g_1(E(R_1))=\{g_1(xy)=g_1(y)-g_1(x):xy\in E(R_1)\}=[1,13]$.

(a-1) The graph $O_1$ admits a set-ordered odd-graceful labeling $g\,'_1$, such that the edge color set $g\,'_1(E(O_1))=\{g\,'_1(xy)=g\,'_1(y)-g\,'_1(x):xy\in E(O_1)\}=[1,25]^o$.

(b) The graph $R_2$ admits a set-ordered felicitous-difference total labeling $g_2$, such that each edge $xy\in E(R_2)$ holds $|g_2(x)+g_2(y)-g_2(xy)|=7$ true, and the edge color set $g_2(E(O_2))=[1,13]$.

(b-2) The graph $O_2$ admits a set-ordered odd-edge felicitous-difference total labeling $g\,'_2$, such that each edge $xy\in E(O_2)$ holds $|g\,'_2(x)+g\,'_2(y)-g\,'_2(xy)|=7$ true, and the edge color set $g\,'_2(E(O_2))=[1,25]^o$.

(c) The graph $R_3$ admits a set-ordered edge-magic total labeling $g_3$, such that each edge $xy\in E(R_3)$ holds $g_3(x)+g_3(xy)+g_3(y)=21$ true, and the edge color set$g_3(E(O_3))=[1,13]$.

(c-1) The graph $O_3$ admits a set-ordered odd-edge edge-magic total labeling $g\,'_3$, such that each edge $xy\in E(O_3)$ holds $g\,'_3(x)+g\,'_3(xy)+g\,'_3(y)=40$ true, and the edge color set $g\,'_3(E(O_3))=[1,25]^o$.

(d) The graph $R_4$ admits a set-ordered odd-edge graceful-difference total labeling $g_4$, such that each edge $xy\in E(R_4)$ holds $\big ||g_4(x)-g_4(y)|-g_4(xy)\big |=[g_4(y)-g_4(x)]-g_4(xy)=0$ true, and the edge color set $g_4(E(O_4))=[1,13]$.

(d-1) The graph $O_4$ admits a set-ordered odd-edge graceful-difference total labeling $g\,'_4$, such that each edge $xy\in E(O_4)$ holds $\big ||g\,'_4(x)-g\,'_4(y)|-g\,'_4(xy)\big |=[g\,'_4(y)-g\,'_4(x)]-g\,'_4(xy)=0$ true, and the edge color set $g\,'_4(E(O_4))=[1,25]^o$.

(e) The graph $R_5$ admits a set-ordered edge-difference total labeling $g_5$, such that each edge $xy\in E(R_5)$ holds $g_5(xy)+|g_5(x)-g_5(y)|=g_5(xy)+g_5(y)-g_5(x)=14$ true, and the edge color set $g_5(E(O_5))=[1,13]$.

(e-1) The graph $O_5$ admits a set-ordered odd-edge edge-difference total labeling $g\,'_5$, such that each edge $xy\in E(O_5)$ holds $g\,'_5(xy)+|g\,'_5(x)-g\,'_5(y)|=g\,'_5(xy)+g\,'_5(y)-g\,'_5(x)=26$ true, and the edge color set $g\,'_5(E(O_5))=[1,25]^o$.\qqed
\end{example}

\begin{figure}[h]
\centering
\includegraphics[width=16.4cm]{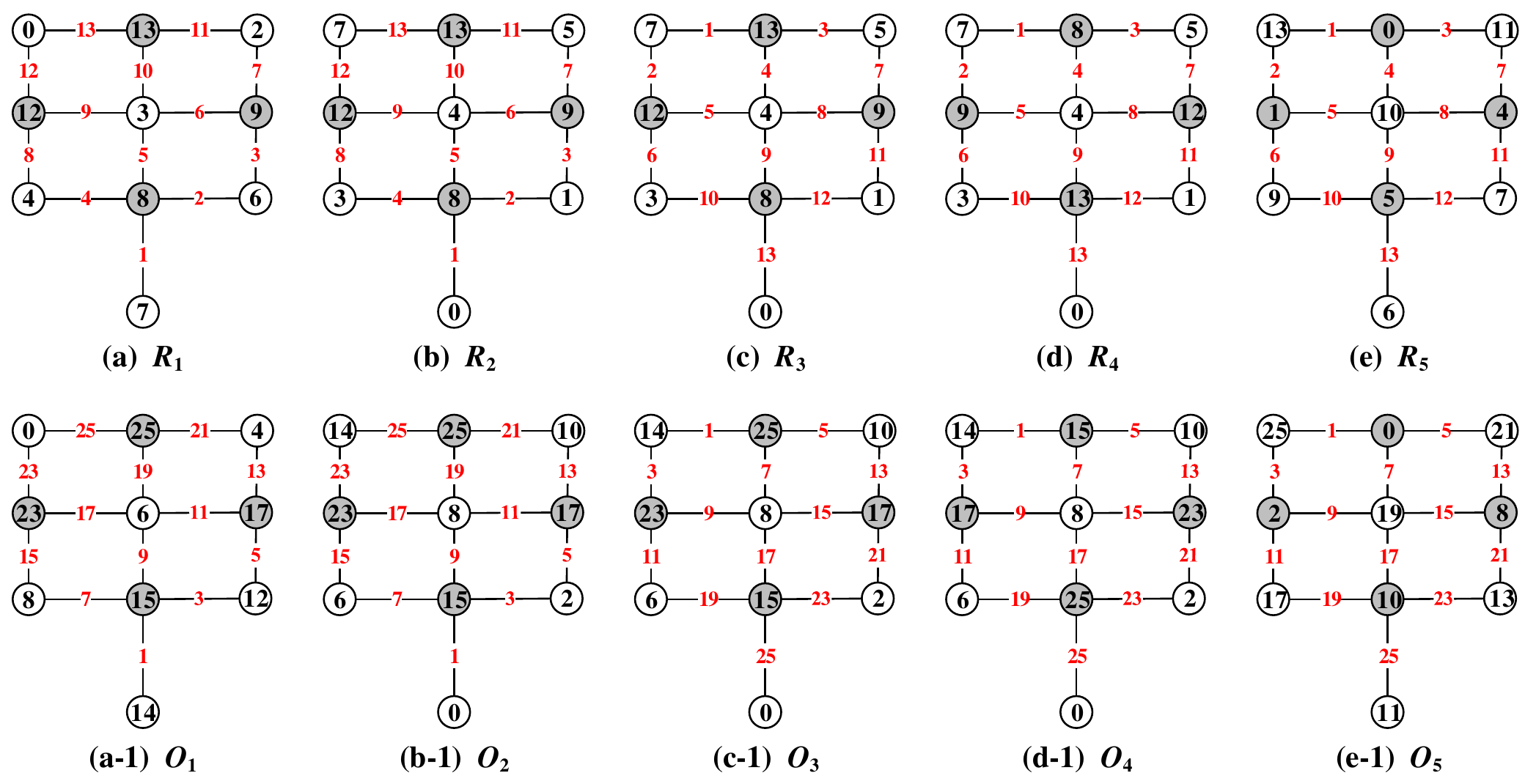}\\
\caption{\label{fig:basic-definition-group}{\small A scheme for illustrating Definition \ref{defn:basic-W-type-labelings} and Definition \ref{defn:basic-W-type-labelings11}.}}
\end{figure}

\begin{example}\label{exa:8888888888}
We present Fig.\ref{fig:basic-definition-perfecr-m} for illustrating Definition \ref{defn:basic-W-type-labelings}, Definition \ref{defn:group-total-labelings-definition} and Definition \ref{defn:group-total-colorings-definition}, there are:

(A) The tree $T_1$ with a perfect matching $M(T_1)$ admits a \emph{set-ordered strongly graceful labeling} $h_1$, such that each matching edge $uv\in M(T_1)$ holds $h_1(u)+h_1(v)=11$ true, and the edge color set $h_1(E(T_1))=\{h_1(xy)=h_1(y)-h_1(x):xy\in E(T_1)\}=[1,11]$.

(A-1) The tree $P_1$ with a perfect matching $M(P_1)$ admits a \emph{set-ordered strongly odd-graceful labeling} $h\,'_1$, such that each matching edge $uv\in M(P_1)$ holds $h\,'_1(u)+h\,'_1(v)=21$ true, and the edge color set $h\,'_1(E(P_1))=\{h\,'_1(xy)=h\,'_1(y)-h\,'_1(x):xy\in E(P_1)\}=[1,21]^o$.

(B) The tree $T_2$ with a perfect matching $M(T_2)$ admits a set-ordered felicitous-difference total labeling $h_2$: (i) each edge $xy\in E(T_2)$ holds $|h_2(x)+h_2(y)-h_2(xy)|=5$ true; (ii) the edge color set $h_2(E(T_2))=[1,11]$; and (iii) the matching edge set $\{h_2(u)+h_2(v):uv\in M(T_2)\}=\{$6, 8, 10, 12, 14, 16$\}$.

(B-1) The tree $P_2$ with a perfect matching $M(P_2)$ admits a set-ordered odd-edge felicitous-difference total labeling $h\,'_2$: (i) each edge $xy\in E(P_2)$ holds $|h\,'_2(x)+h\,'_2(y)-h\,'_2(xy)|=10$ true; (ii) the edge color set $h\,'_2(E(P_2))=[1,21]^o$; and (iii) the matching edge set $\{h\,'_2(u)+h\,'_2(v):uv\in M(P_2)\}=\{11,15,19,23,17,31\}$.

(C) The tree $T_3$ with a perfect matching $M(T_3)$ admits a set-ordered edge-magic total coloring $h_3$: (i) each edge $xy\in E(T_3)$ holds $h_3(x)+h_3(xy)+h_3(y)=17$ true; (ii) the edge color set $h_3(E(T_3))=[1,11]$; and (iii) the matching edge set $\{h_3(u)+h_3(v):uv\in M(T_3)\}=\{$6, 8, 10, 12, 14, 16$\}$.

(C-1) The tree $P_3$ with a perfect matching $M(P_3)$ admits a set-ordered odd-edge edge-magic total coloring $h\,'_3$: (i) each edge $xy\in E(P_3)$ holds $h\,'_3(x)+h\,'_3(xy)+h\,'_3(y)=32$ true; (ii) the edge color set $h\,'_3(E(P_3))=[1,21]^o$; and (iii) the matching edge set $\{h\,'_3(u)+h\,'_3(v):uv\in M(P_3)\}=\{$11, 15, 19, 23, 17, 31$\}$.

(D) The tree $T_4$ with a perfect matching $M(T_4)$ with a perfect matching $M(T_4)$ admits a \emph{set-ordered strongly graceful-difference total labeling} $h_4$: (i) each edge $xy\in E(T_4)$ holds $\big ||h_4(x)-h_4(y)|-h_4(xy)\big |=0$ true; (ii) the edge color set $h_4(E(T_4))=[1,11]$; and (iii) each matching edge $uv\in M(T_4)$ holds $h_4(u)+h_4(v)=11$ true.

(D-1) The tree $P_4$ with a perfect matching $M(P_4)$ with a perfect matching $M(P_4)$ admits a \emph{set-ordered strongly odd-edge graceful-difference total labeling} $h\,'_4$: (i) each edge $xy\in E(P_4)$ holds $\big ||h\,'_4(x)-h\,'_4(y)|-h\,'_4(xy)\big |=0$ true; (ii) the edge color set $h\,'_4(E(P_4))=[1,21]^o$; and (iii) each matching edge $uv\in M(P_4)$ holds $h\,'_4(u)+h\,'_4(v)=21$ true.

(E) The tree $T_5$ with a perfect matching $M(T_5)$ admits a \emph{set-ordered strongly edge-difference total labeling} $h_5$: (i) each edge $xy\in E(T_5)$ holds $h_5(xy)+|h_5(x)-h_5(y)|=12$ true; (ii) the edge color set $h_5(E(T_5))=[1,11]$; and (iii) each matching edge $uv\in M(T_5)$ holds $h_5(u)+h_5(v)=11$ true.

(E-1) The tree $P_5$ with a perfect matching $M(P_5)$ admits a \emph{set-ordered strongly odd-edge edge-difference total labeling} $h\,'_5$: (i) each edge $xy\in E(P_5)$ holds $h\,'_5(xy)+|h\,'_5(x)-h\,'_5(y)|=22$ true; (ii) the edge color set $h\,'_5(E(P_5))=[1,21]^o$; and (iii) each matching edge $uv\in M(P_5)$ holds $h\,'_5(u)+h\,'_5(v)=21$ true.\qqed
\end{example}

\begin{figure}[h]
\centering
\includegraphics[width=16.4cm]{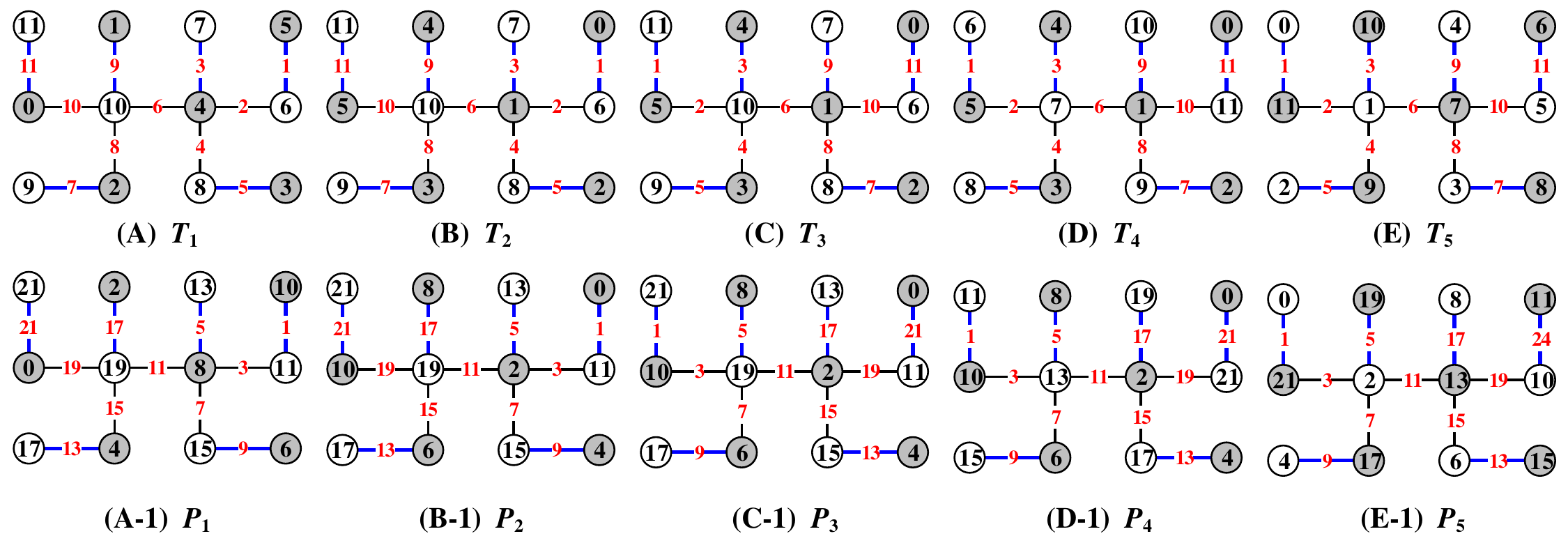}\\
\caption{\label{fig:basic-definition-perfecr-m}{\small For basic-definition-perfect-matching.}}
\end{figure}

\begin{defn}\label{defn:perfect-total-labelings}
$^*$ Let $G$ be a $(p,p-1)$-tree with a perfect matching $M(G)$ and the vertex bipartition $(X,Y)$ holding $V(G)=X\cup Y$ and $X\cap Y=\emptyset$ true. Suppose that $G$ admits a total labeling $f:V(G)\cup E(G)\rightarrow [0,M]$. Let $a,b,c$ be non-negative integers, there are the following restrictions:
\begin{asparaenum}[\textbf{\textrm{Sc}}-1.]
\item \label{perfect:labeling} \textbf{C-1.} $|f(V(G))|=p$;
\item \label{perfect:vertex} $f(V(G))\subseteq [0,p-1]$, $\min f(V(G))=0$;
\item \label{perfect:odd-v} $f(V(G))\subset [0,2p-3]$, $\min f(V(G))=0$;
\item \label{perfect:edge} $f(E(G))=\{f(xy):xy\in E(G)\}=[1,p-1]$;
\item \label{perfect:odd-edge} $f(E(G))=\{f(xy):xy\in E(G)\}=[1,2p-3]^o$;
\item \label{perfect:set-ordered} $\max f(X)<\min f(Y)$;
\item \label{perfect:graceful-difference} $\big ||f(x)-f(y)|-f(xy)\big |=a$ for each edge $xy\in E(G)$;
\item \label{perfect:edge-difference} $f(xy)+|f(x)-f(y)|=b$ for each edge $xy\in E(G)$;
\item \label{perfect:felicitous-difference} $\big |f(x)-f(y)-f(xy)\big |=c$ for each edge $xy\in E(G)$;
\item \label{perfect:edge-magic} $f(x)+f(y)+f(xy)=d$ for each edge $xy\in E(G)$;
\item \label{perfect:strongly} each matching edge $uv\in M(G)$ holds $f(u)+f(v)=e$ true.
\end{asparaenum}
We call $f$:
\begin{asparaenum}[\textbf{\textrm{Sp}}-1.]
\item A \emph{set-ordered strongly graceful-difference total labeling} if Sc-\ref{perfect:labeling}, Sc-\ref{perfect:vertex}, Sc-\ref{perfect:edge}, Sc-\ref{perfect:set-ordered}, Sc-\ref{perfect:graceful-difference} and Sc-\ref{perfect:strongly} hold true.
\item A \emph{set-ordered strongly odd-edge graceful-difference total labeling} if Sc-\ref{perfect:labeling}, Sc-\ref{perfect:odd-v}, Sc-\ref{perfect:odd-edge}, Sc-\ref{perfect:set-ordered}, Sc-\ref{perfect:graceful-difference} and Sc-\ref{perfect:strongly} hold true.
\item A \emph{set-ordered strongly edge-difference total labeling} if Sc-\ref{perfect:labeling}, Sc-\ref{perfect:vertex}, Sc-\ref{perfect:edge}, Sc-\ref{perfect:set-ordered}, Sc-\ref{perfect:edge-difference} and Sc-\ref{perfect:strongly} hold true.
\item A \emph{set-ordered strongly odd-edge edge-difference total labeling} if Sc-\ref{perfect:labeling}, Sc-\ref{perfect:odd-v}, Sc-\ref{perfect:odd-edge}, Sc-\ref{perfect:set-ordered}, Sc-\ref{perfect:edge-difference} and Sc-\ref{perfect:strongly} hold true.
\item A \emph{set-ordered strongly edge-magic total labeling} if Sc-\ref{perfect:labeling}, Sc-\ref{perfect:vertex}, Sc-\ref{perfect:edge}, Sc-\ref{perfect:set-ordered}, Sc-\ref{perfect:edge-magic} and Sc-\ref{perfect:strongly} hold true.
\item A \emph{set-ordered strongly odd-edge edge-magic total labeling} if Sc-\ref{perfect:labeling}, Sc-\ref{perfect:odd-v}, Sc-\ref{perfect:odd-edge}, Sc-\ref{perfect:set-ordered}, Sc-\ref{perfect:edge-magic} and Sc-\ref{perfect:strongly} hold true.

\item A \emph{set-ordered strongly felicitous-difference total labeling} if Sc-\ref{perfect:labeling}, Sc-\ref{perfect:vertex}, Sc-\ref{perfect:edge}, Sc-\ref{perfect:set-ordered}, Sc-\ref{perfect:felicitous-difference} and Sc-\ref{perfect:strongly} hold true.
\item A \emph{set-ordered strongly odd-edge felicitous-difference total labeling} if Sc-\ref{perfect:labeling}, Sc-\ref{perfect:odd-v}, Sc-\ref{perfect:odd-edge}, Sc-\ref{perfect:set-ordered}, Sc-\ref{perfect:felicitous-difference} and Sc-\ref{perfect:strongly} hold true.
  \qqed
\end{asparaenum}
\end{defn}

We present the \emph{twin odd-edge $W$-magic total labelings} as follows:

\begin{defn}\label{defn:group-definition-twin-total-labelingss}
$^*$ Let $G$ be a bipartite $(p,q)$-graph having its own vertex set $V(G)=X_G\cup Y_G$ with $X_G\cap Y_G=\emptyset$, and let $T$ be another bipartite $(p\,',q)$-graph having its own vertex set $V(T)=X_T\cup Y_T$ with $X_T\cap Y_T=\emptyset$. The bipartite $(p,q)$-graph $G$ admits a total labeling $F:V(G)\cup E(G)\rightarrow [0,2q-1]$, and the bipartite $(p\,',q)$-graph $T$ admits a total labeling $F^*:V(T)\cup E(T)\rightarrow [0,2q]$.

(i) \textbf{If}

(i-1) $F$ is a set-ordered odd-edge edge-magic total labeling of $G$;

(i-2) the set-ordered restriction $F^*_{\max}(X_T)<F^*_{\min}(Y_T)$ holds true;

(i-3) the edge color set $F^*(E(T))=[1,2q-1]^o$;

(i-4) there is a positive integer $c_1$, so that each edge $xy\in E(T)$ holds a magic-type restriction $F^*(x)+F^*(xy)+F^*(y)=c_1$; and

(i-5) $F(V(G))\cup F^*(V(T))\subseteq [0,2q]$,\\
then we call $\langle F,F^*\rangle $ a \emph{twin set-ordered odd-edge edge-magic total labeling} of two graphs $G$ and $T$. Especially, we call $\langle F,F^*\rangle $ a \emph{perfect twin set-ordered odd-edge edge-magic total labeling} of $G$ and $T$ if $F(V(G))\cup F^*(V(T))=[0,2q]$.

(ii) \textbf{If}

(ii-1) $F$ is a set-ordered odd-edge edge-difference total labeling of $G$;

(ii-2) the set-ordered restriction $F^*_{\max}(X_T)<F^*_{\min}(Y_T)$ holds true;

(ii-3) the edge color set $F^*(E(T))=[1,2q-1]^o$;

(ii-4) there is a positive integer $c_2$, so that each edge $xy\in E(T)$ holds a magic-type restriction $F^*(xy)+|F^*(y)-F^*(x)|=c_2$; and

(ii-5) $F(V(G))\cup F^*(V(T))\subseteq [0,2q]$,\\
then, $\langle F,F^*\rangle $ is called a \emph{twin set-ordered odd-edge edge-difference total labeling} of two graphs $G$ and $T$. Moreover, we call $\langle F,F^*\rangle $ a \emph{perfect twin set-ordered odd-edge edge-difference total labeling} of $G$ and $T$ if $F(V(G))\cup F^*(V(T))=[0,2q]$.

(iii) \textbf{If}

(iii-1) $F$ is a set-ordered odd-edge felicitous-difference total labeling of $G$;

(iii-2) the set-ordered restriction $F^*_{\max}(X_T)<F^*_{\min}(Y_T)$ holds true;

(iii-3) the edge color set $F^*(E(T))=[1,2q-1]^o$;

(iii-4) there is a non-negative integer $c_3$, so that each edge $xy\in E(T)$ holds a magic-type restriction $|F^*(y)+F^*(x)-F^*(xy)|=c_3$; and

(iii-5) $F(V(G))\cup F^*(V(T))\subseteq [0,2q]$,\\
we call $\langle F,F^*\rangle $ a \emph{twin set-ordered odd-edge felicitous-difference total labeling} of two graphs $G$ and $T$. And we call $\langle F,F^*\rangle $ a \emph{perfect twin set-ordered odd-edge felicitous-difference total labeling} of $G$ and $T$ if $F(V(G))\cup F^*(V(T))=[0,2q]$.

(vi) \textbf{If}

(vi-1) $F$ is a set-ordered odd-edge graceful-difference total labeling;

(vi-2) the set-ordered restriction $F^*_{\max}(X_T)<F^*_{\min}(Y_T)$ holds true;

(vi-3) the edge color set $F^*(E(T))=[1,2q-1]^o$;

(vi-4) there is a non-negative integer $c_4$, so that each edge $xy\in E(T)$ holds a magic-type restriction $\big ||F^*(y)-F^*(x)|-F^*(xy)\big |=c_4$; and

(vi-5) $F(V(G))\cup F^*(V(T))\subseteq [0,2q]$,\\
then, $\langle F,F^*\rangle $ is called a \emph{twin set-ordered odd-edge graceful-difference total labeling} of two graphs $G$ and $T$. Furthermore, we call $\langle F,F^*\rangle $ a \emph{perfect twin set-ordered odd-edge graceful-difference total labeling} of $G$ and $T$ if $F(V(G))\cup F^*(V(T))=[0,2q]$.\qqed
\end{defn}

\begin{example}\label{exa:twin-set-ordered-odd-edge-labelings}
For illustrating Definition \ref{defn:group-definition-twin-total-labelingss}, we see examples shown in Fig.\ref{fig:22-twin-matching} as follows: Notice that $A_0\cong B_0\cong A_i\cong B_i$ for $i\in [1,4]$. Let $q=|E(A_0)|=9$.

(a) A graph $A_0$ admits a set-ordered odd-graceful labeling $f_0$, and $f_0(E(A_0))=[1,17]^o$.

(a-1) A graph $B_0$ admits a set-ordered labeling $g_0$ with $g_0(E(B_0))=[1,17]^o$.

Since $f_0(V(A_0))\cup g_0(V(B_0))=[0,18]=[0,2q]$, $\langle f_0,g_0\rangle $ is a \emph{twin set-ordered odd-graceful labeling}.

(b) A graph $A_1$ admits a set-ordered odd-edge felicitous-difference total labeling $f_1$ holding $|f_1(x)+f_1(y)-f_1(xy)|=8$ for each edge $xy\in E(A_1)$ and $f_1(E(A_1))=[1,17]^o$.

(b-1) A graph $B_1$ admits a set-ordered odd-edge felicitous-difference total labeling $g_1$ with $|g_1(u)+g_1(v)-g_1(uv)|=8$ for each edge $uv\in E(B_1)$ and $g_1(E(B_1))=[1,17]^o$.

Since $f_1(V(A_1))\cup g_1(V(B_1))=[0,18]=[0,2q]$, $\langle f_1,g_1\rangle $ is a \emph{twin set-ordered felicitous-difference total labeling}.

(c) A graph $A_2$ admits a set-ordered odd-edge edge-magic total labeling $f_2$ holding $f_2(x)+f_2(xy)+f_2(y)=26$ for each edge $xy\in E(A_2)$ and $f_2(E(A_2))=[1,17]^o$.

(c-1) A graph $B_2$ admits a set-ordered odd-edge edge-magic total labeling $g_2$ with $g_2(u)+g_2(uv)+g_2(v)=28$ for each edge $uv\in E(B_2)$ and $g_2(E(B_2))=[1,17]^o$.

Since $f_2(V(A_2))\cup g_2(V(B_2))=[0,18]=[0,2q]$, $\langle f_2,g_2\rangle $ is a \emph{twin set-ordered edge-magic total labeling}.

(d) A graph $A_3$ admits a set-ordered odd-edge edge-difference total labeling $f_3$ holding $f_3(xy)+|f_3(y)-f_3(x)|=18$ for each edge $xy\in E(A_3)$ and $f_3(E(A_3))=[1,17]^o$.

(d-1) A graph $B_3$ admits a set-ordered odd-edge edge-difference total labeling $g_3$ with $g_3(uv)+|g_3(u)-g_3(v)|=18$ for each edge $uv\in E(B_3)$ and $g_3(E(B_3))=[1,17]^o$.

Since $f_3(V(A_3))\cup g_3(V(B_3))=[0,18]=[0,2q]$, $\langle f_3,g_3\rangle $ is a \emph{twin set-ordered edge-difference total labeling}.

(e) A graph $A_4$ admits a set-ordered odd-edge graceful-difference total labeling $f_4$ holding $\big ||f_4(x)-f_4(y)|-f_4(xy)\big |=0$ for each edge $xy\in E(A_4)$ and $f_4(E(A_4))=[1,17]^o$.

(e-1) A graph $B_4$ admits a set-ordered odd-edge graceful-difference total labeling $g_4$ with $\big ||g_4(u)-g_4(v)|-f_g(uv)\big |=0$ for each edge $uv\in E(B_4)$ and $g_4(E(B_4))=[1,17]^o$.

Since $f_4(V(A_4))\cup g_4(V(B_4))=[0,18]=[0,2q]$, $\langle f_4,g_4\rangle $ is a \emph{twin set-ordered graceful-difference total labeling}.\qqed
\end{example}

\begin{figure}[h]
\centering
\includegraphics[width=16.4cm]{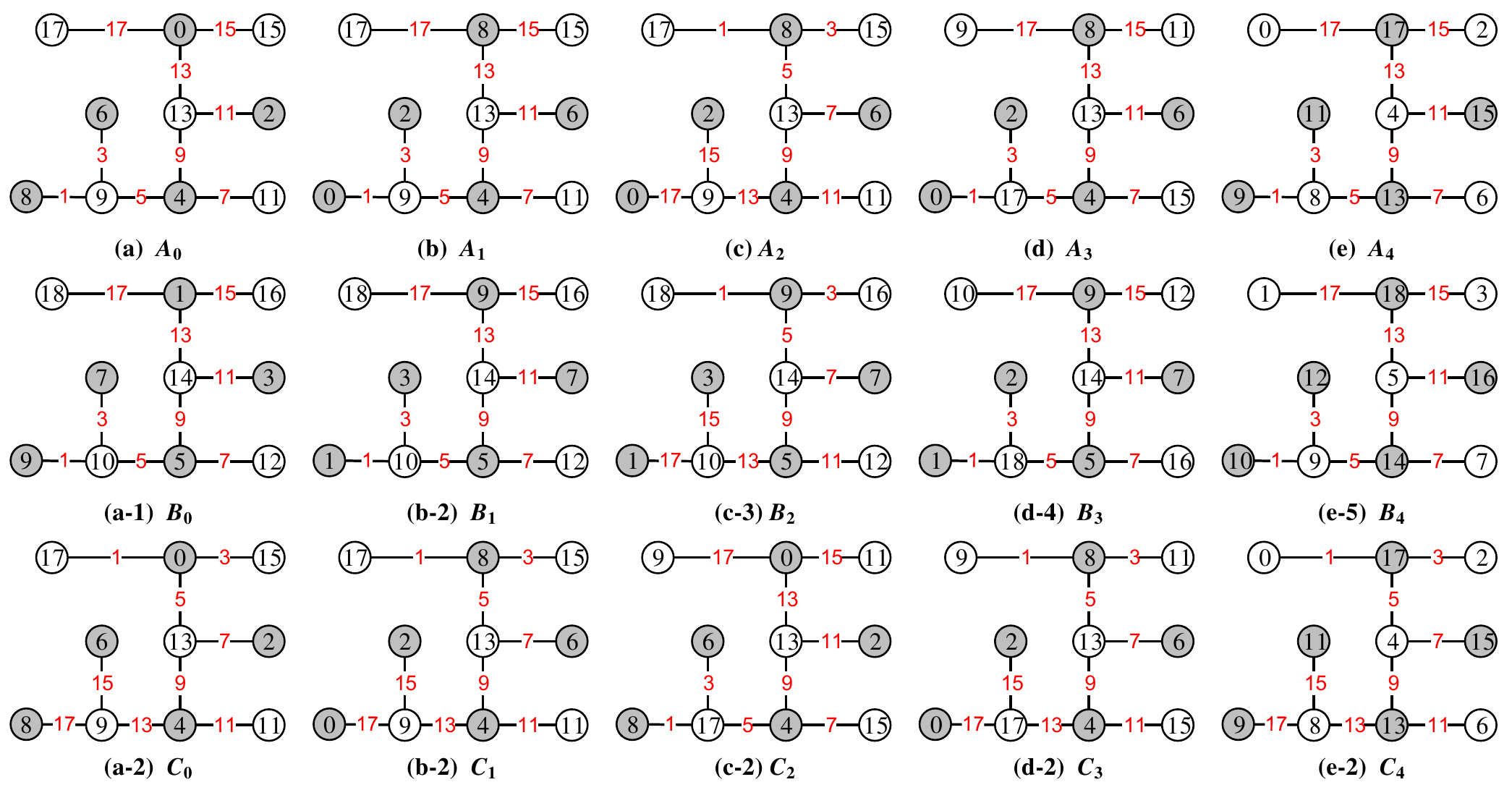}\\
\caption{\label{fig:22-twin-matching}{\small For illustrating Definition \ref{defn:group-definition-twin-total-labelingss}.}}
\end{figure}

\begin{example}\label{exa:33-edge-matching}
In Fig.\ref{fig:22-twin-matching}, each graph $A_i$ admits a labeling $f_i$ and each graph $C_i$ admits a labeling $h_i$, such that $f_i(xy)+h_i(xy)=18$ for each edge $xy\in E(A_i)=E(C_i)$ with $i\in [0,4]$. So, we call $\langle f_i,h_i\rangle $ an \emph{edge-matching $W$-magic total labeling} for $i\in [0,4]$, where ``$W$-magic'' is one of edge-magic, edge-difference, felicitous-difference, graceful-difference.\qqed
\end{example}

\begin{defn} \label{defn:111111}
$^*$ Let ``$W$-magic'' be one of edge-magic, edge-difference, felicitous-difference, graceful-difference. Removing the restriction ``set-ordered'' in Definition \ref{defn:group-definition-twin-total-labelingss} will produce four \emph{twin odd-edge $W$-magic total labelings}. If there is at least a pair of vertices colored with the same color in Definition \ref{defn:group-definition-twin-total-labelingss}, we obtain four \emph{twin odd-edge $W$-magic total colorings}.\qqed
\end{defn}

\begin{defn} \label{defn:111111}
$^*$ \textbf{$m$-tuple (set-ordered) odd-edge $W$-magic total labeling/coloring.} $F:V(G_i)\rightarrow [0,M]$ with $i\in [1,m]$, such that: $F$ is an (a set-ordered) odd-edge $W$-magic total labeling/coloring of $G_1$; $F(E(G_i))=[1,2q-1]^o$; $F(u_{i,s}v_{i,j})\in E(G_i)$ holds one of four magic-type restrictions defined in Definition \ref{defn:group-definition-twin-total-labelingss}; and $\bigcup^m_{i=1}F(V(G_i))\subseteq [0,M]$.\qqed
\end{defn}

\begin{defn} \label{defn:111111}
$^*$ There are \emph{$m$-tuple odd-edge Topcode-matrix team} $T_{code}^i=(X_i$, $E_i$, $Y_i)^T$ with v-vector $X_i=(x_{i,1},x_{i,2},\dots ,x_{i,n})$, e-vector $E_i=(e_{i,1},e_{i,2},\dots ,e_{i,n})$ and v-vector $Y_i=(y_{i,1},y_{i,2},\dots ,y_{i,n})$ for $i\in [1,m]$, such that $\bigcup ^m_{i=1}\bigcup^n_{j=1}(\{x_{i,j}\}\cup\{y_{i,j}\})\subseteq [0,M]$, and each $e_{i,j}$ is odd and holds one $W$-magic restriction of $x_{i,j}+e_{i,j}+y_{i,j}=c_1$, $e_{i,j}+|x_{i,j}-y_{i,j}|=c_2$, $|x_{i,j}+y_{i,j}-e_{i,j}|=c_3$ and $\big ||x_{i,j}-y_{i,j}|-e_{i,j}\big |=c_4$, as well as $\{e_{i,j}:e_{i,j}\in E_i\}=[1,2n-1]^o$.\qqed
\end{defn}

\begin{rem}\label{rem:333333}
A graphic group $\{F_m(G,f);\oplus\}$ admits a \emph{$n$-tuple (set-ordered) odd-edge $W$-magic total labeling/coloring} for $n\leq m$.\paralled
\end{rem}

\subsection{Dual-type labelings and colorings}

Part of the content in this subsection are cited from \cite{Yao-Su-Ma-Wang-Yang-arXiv-2202-03993v1}. Let $G$ be a connected bipartite $(p,q)$-graph admitting a \emph{set-ordered graceful labeling} $f$, and let $(X,Y)$ be the bipartition of vertex set $V(G)$, where $X=\{x_1,x_2,\dots,x_s\}$ and $Y=\{y_1,y_2,\dots,y_t\}$ with $s+t=p$. Without loss of generality, there are inequalities
\begin{equation}\label{equ:set-ordered-graceful-labeling}
0=f(x_1)<f(x_2)<\cdots <f(x_s)<f(y_1)<f(y_2)<\cdots <f(y_t)=q
\end{equation} also, $\max f(X)<\min f(Y)$, and $f(E(G))=[1,q]$. See a connected bipartite $(8,9)$-graph $G_0$ admitting a set-ordered graceful labeling shown in Fig.\ref{fig:hu-set-dual} (a).

We are ready to define the following set-dual type labelings:

\textrm{\textbf{Set-Dual-1.}} The \emph{total set-dual labeling} $f_{dual}$ of $f$ is defined as:
$$f_{dual}(w)=\max f(V(G))+\min f(V(G))-f(w)=q-f(w)
$$ for $w\in V(G)$, and the induced edge color of each edge $x_iy_j$ is
\begin{equation}\label{equ:2222222}
f_{dual}(x_iy_j)=|f_{dual}(x_i)-f_{dual}(y_j)|=|f(x_i)-f(y_j)|=f(y_j)-f(x_i)=f(x_iy_j)
\end{equation}
Then $f_{dual}(E(G))=f(E(G))=[1,q]$ and there are
\begin{equation}\label{equ:2222222}
{
\begin{split}
0=&f_{dual}(y_t)<f_{dual}(y_{t-1})<\cdots <f_{dual}(y_1)<f_{dual}(x_s)\\
&<f_{dual}(x_{s-1})<\cdots <f_{dual}(x_2)<\cdots <f_{dual}(x_1)=q
\end{split}}
\end{equation} also, the dual labeling $f_{dual}$ is a \emph{set-ordered graceful labeling} of $G$ too.

\begin{problem}\label{qeu:444444}
Suppose that a connected bipartite $(p,q)$-graph $G$ admitting a set-ordered graceful labeling $f$, and $f_{dual}$ is the dual labeling of $f$. \textbf{Dose} $f(V(G))\cup f_{dual}(V(G))=[0,q]$?
\end{problem}

Another total dual labeling $f^*_{dual}$ of $f$ is defined as
$$
f^*_{dual}(w)=\max f(V(G))+\min f(V(G))-f(w)=q-f(w)
$$ for $w\in V(G)$, and the induced edge color of each edge $x_iy_j$ is defined by
$$
f^*_{dual}(x_iy_j)=\max f(E(G))+\min f(E(G))-f(x_iy_j)=q+1-f(x_iy_j)
$$ for $x_iy_j\in E(G)$, then $f^*_{dual}(E(G))=f(E(G))=[1,q]$. Because of
$$
f^*_{dual}(x_iy_j)+|f^*_{dual}(y_j)-f^*_{dual}(x_i)|=q+1-f(x_iy_j)+|f(y_j)-f(x_i)|=q+1
$$ so $f^*_{dual}$ is a \emph{set-ordered edge-difference total labeling} of $G$.

\begin{thm}\label{thm:666666}
A connected bipartite graph $G$ admits a set-ordered graceful labeling $f$ if and only if the dual labeling $f_{dual}$ of the labeling $f$ is a set-ordered graceful labeling and another dual labeling $f^*_{dual}$ of the labeling $f$ is a set-ordered edge-difference total labeling.
\end{thm}

\textrm{\textbf{Set-Dual-2.}} The \emph{$XY$-set-dual labeling} $g_{setXY}$ of $f$ is defined as:
$$
g_{setXY}(x_i)=\max f(X)+\min f(X)-f(x_i),~x_i\in X
$$ and
$$
g_{setXY}(y_j)=\max f(Y)+\min f(Y)-f(y_j),~y_j\in Y
$$ and the induced edge color of each edge $x_iy_j$ is defined by
\begin{equation}\label{equ:2222222}
{
\begin{split}
g_{setXY}(x_iy_j)&=|g_{setXY}(x_i)-g_{setXY}(y_j)|\\
&=\big |[\max f(X)+\min f(X)-f(x_i)]-[\max f(Y)+\min f(Y)-f(y_j)]\big |\\
&=[\max f(Y)+\min f(Y)]-[\max f(X)+\min f(X)]-f(x_iy_j)\\
&=q+\min f(Y)-\max f(X)-f(x_iy_j)
\end{split}}
\end{equation} and the edge color set
\begin{equation}\label{equ:2222222}
{
\begin{split}
g_{setXY}(E(G))&=[q+\min f(Y)-\max f(X)-q,~q+\min f(Y)-\max f(X)-1]\\
&=[\min f(Y)-\max f(X),~\min f(Y)-\max f(X)+(q-1)]
\end{split}}
\end{equation} then the $XY$-set-dual labeling $g_{setXY}$, when as $\min f(Y)-\max f(X)=1$, is a \emph{set-ordered graceful labeling} of $G$.

And another $XY$-set-dual labeling $g^*_{setXY}$ is defined as $g^*_{setXY}(w)=g_{setXY}(w)$ for $w\in V(G)$, and each edge $x_iy_j\in E(G)$ is colored with
$$
g^*_{setXY}(x_iy_j)=\max f(E(G))+\min f(E(G))-f(x_iy_j)=q+1-f(x_iy_j)
$$ which induces edge color set $g^*_{setXY}(E(G))=f(E(G))=[1,q]$. Since $g^*_{setXY}(u)\neq g^*_{setXY}(v)$ for distinct vertices $u,v\in V(G)$, and
$$\label{eqa:555555}
{
\begin{split}
&\quad \big | |g^*_{setXY}(y_j)-g^*_{setXY}(x_i)|-g^*_{setXY}(x_iy_j)\big |\\
&=\big | [q+\min f(Y)-\max f(X)-f(x_iy_j)]-[q+1-f(x_iy_j)]\big |\\
&=\min f(Y)-\max f(X)-1
\end{split}}
$$ for each edge $x_iy_j\in E(G)$, so $g^*_{setXY}$ is a \emph{set-ordered graceful-difference total labeling} of $G$.

Here, $g^*_{setXY}$ has its own dual labeling $\alpha _{set}$ defined by
$$
\alpha _{set}(w)=\max g^*_{setXY}(V(G))+\min g^*_{setXY}(V(G))- g^*_{setXY}(w)
$$ for $w\in V(G)$, and the edge color of each edge $x_iy_j$ is $\alpha _{set}(x_iy_j)= g^*_{setXY}(x_iy_j)$, so it is not hard to show that $\alpha _{set}$ is a \emph{set-ordered graceful labeling} of $G$.

\vskip 0.4cm

See Fig.\ref{fig:hu-set-dual} for understanding the labelings introduced in \textbf{Set-Dual-1} and \textbf{Set-Dual-2}.

\begin{figure}[h]
\centering
\includegraphics[width=16.4cm]{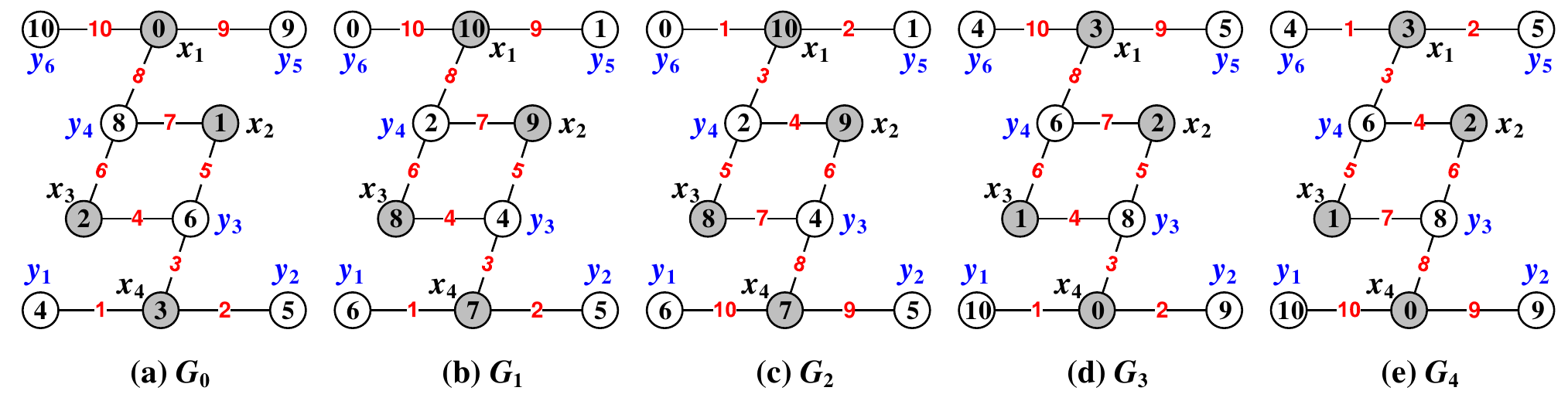}\\
\caption{\label{fig:hu-set-dual} {\small (a) $G_0$ admits a set-ordered graceful labeling $f$; (b) $G_1$ admits a set-ordered graceful labeling $f_{dual}$, which is the total set-dual labeling of $f$; (c) $G_2$ admits a set-ordered edge-difference total labeling $f^*_{dual}$; (d) $G_3$ admits a $XY$-set-dual labeling $g_{setXY}$ of $f$; (e) $G_4$ admits a set-ordered graceful-difference total labeling $g^*_{setXY}$.}}
\end{figure}

\begin{thm}\label{thm:666666}
A connected bipartite graph $G$ admits a set-ordered graceful labeling $f$ if and only if the set-dual labeling $g_{setXY}$ of the labeling $f$ is a set-ordered graceful labeling, and $g^*_{setXY}$ is a set-ordered graceful-difference total labeling of $G$.
\end{thm}

\textrm{\textbf{Set-Dual-3.}} The \emph{$X$-set-dual labeling} $h_{setX}$ of $f$ is defined as:
$$h_{setX}(x_i)=\max f(X)+\min f(X)-f(x_i)=\max f(X)-f(x_i),~x_i\in X
$$ and $h_{setX}(y_j)=f(y_j)$ for $y_j\in Y$, and the edge color of each edge $x_iy_j$ is $h_{setX}(x_iy_j)=f(x_iy_j)$ for $x_iy_j\in E(G)$, so $h_{setX}(E(G))=f(E(G))=[1,q]$. Furthermore, we have
$${
\begin{split}
h_{setX}(x_i)+h_{setX}(y_j)-h_{setX}(x_iy_j)&=\max f(X)+\min f(X)-f(x_i)+f(y_j)-f(x_iy_j)\\
&=\max f(X)
\end{split}}
$$ so $h_{setX}$ is a \emph{set-ordered felicitous-difference total labeling} of $G$.

Moreover, we define another $X$-set-dual labeling $h^*_{setX}$ by $h^*_{setX}(w)=h_{setX}(w)$ for $w\in V(G)$, and
$$
h^*_{setX}(x_iy_j)=\max f(E(G))+\min f(E(G))-f(x_iy_j)=q+1-f(x_iy_j)
$$ for each edge $x_iy_j\in E(G)$, then $h^*_{setX}(E(G))=f(E(G))=[1,q]$. Since
\begin{equation}\label{eqa:555555}
{
\begin{split}
h^*_{setX}(x_i)+h^*_{setX}(x_iy_j)+h^*_{setX}(y_j)&=h_{setX}(x_i)+q+1-f(x_iy_j)+h_{setX}(y_j)\\
&=\max f(X)+h_{setX}(x_iy_j)+q+1-f(x_iy_j)\\
&=\max f(X)+f(x_iy_j)+q+1-f(x_iy_j)\\
&=q+1+\max f(X)
\end{split}}
\end{equation} which shows that $h^*_{setX}$ is a \emph{set-ordered edge-magic total labeling} of $G$.

\begin{thm}\label{thm:666666}
A connected bipartite graph $G$ admits a set-ordered graceful labeling $f$ if and only if the set-dual labeling $h_{setX}$ of the labeling $f$ is a set-ordered felicitous-difference labeling, and $h^*_{setX}$ is a set-ordered edge-magic total labeling of $G$.
\end{thm}

\textrm{\textbf{Set-Dual-4.}} The \emph{$Y$-set-dual labeling} $h_{setY}$ of $f$ is defined as: $h_{setY}(x_i)=f(x_i)$ for $x_i\in X$,
$$
h_{setY}(y_j)=\max f(Y)+\min f(Y)-f(y_j)=q+\min f(Y)-f(y_j),~y_j\in Y
$$ and the edge color of each edge $x_iy_j$ is $h_{setY}(x_iy_j)=f(x_iy_j)$ for $x_iy_j\in E(G)$, immediately, $h_{setY}(E(G))=f(E(G))=[1,q]$. Moreover, we confirm that $h_{setY}$ is an \emph{edge-magic total labeling} of $G$, since
\begin{equation}\label{eqa:555555}
{
\begin{split}
h_{setY}(x_i)+h_{setY}(x_iy_j)+h_{setY}(y_j)&=f(x_i)+f(x_iy_j)+q+\min f(Y)-f(y_j)\\
&=q+\min f(Y)
\end{split}}
\end{equation} for each edge $x_iy_j\in E(G)$.

And another case, we define another $Y$-set-dual labeling $h^*_{setY}$ by $h^*_{setY}(w)=h_{setY}(w)$ for $w\in V(G)$, and
$$
h^*_{setY}(x_iy_j)=\max f(E(G))+\min f(E(G))-f(x_iy_j)=q+1-f(x_iy_j)
$$ for $x_iy_j\in E(G)$, then $h^*_{setY}(E(G))=f(E(G))=[1,q]$. We omit the proof for $h^*_{setY}$ being a \emph{set-ordered felicitous-difference total labeling} of $G$.

\begin{thm}\label{thm:666666}
A connected bipartite graph $G$ admits a set-ordered graceful labeling $f$ if and only if the set-dual labeling $h_{setY}$ of the labeling $f$ is a set-ordered edge-magic labeling, and $h^*_{setY}$ is a set-ordered felicitous-difference total labeling of $G$.
\end{thm}

\vskip 0.4cm

See Fig.\ref{fig:hu-set-dual-11} for understanding the labelings introduced in \textbf{Set-Dual-3} and \textbf{Set-Dual-4}, although the examples admits set-dual colorings (refer to Definition \ref{defn:group-total-colorings-definition}).
\begin{figure}[h]
\centering
\includegraphics[width=16.4cm]{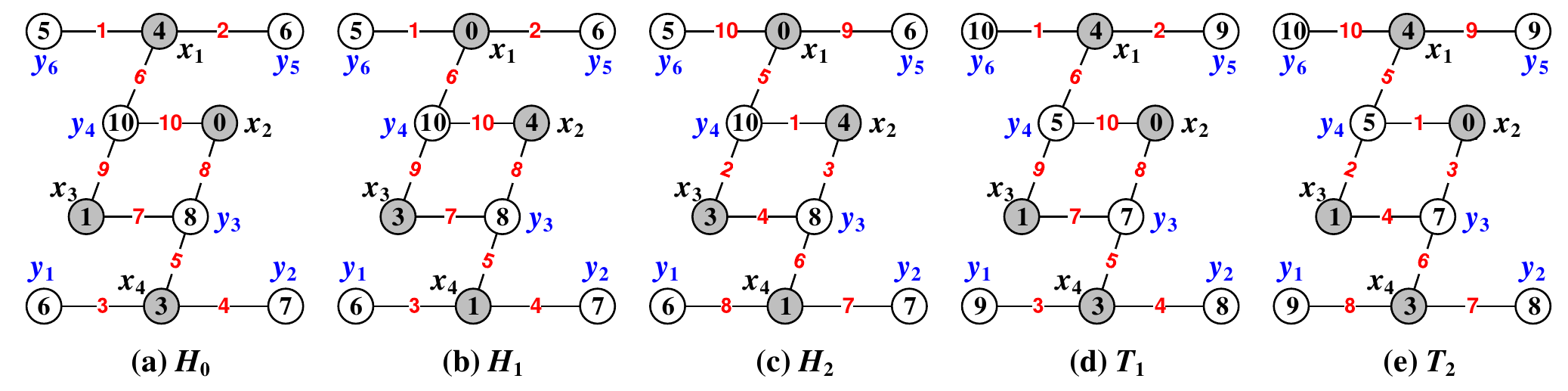}\\
\caption{\label{fig:hu-set-dual-11} {\small (a) $H_0$ admits a set-ordered graceful total coloring $h$, there are two vertices colored with the same color; (b) $H_1$ admits a set-ordered felicitous-difference total coloring $h_{setX}$, which is a $X$-set-dual coloring of the set-ordered graceful coloring $h$; (c) $H_2$ admits a set-ordered edge-magic total coloring $h^*_{setX}$; (d) $T_1$ admits a set-ordered edge-magic total coloring $h_{setY}$; (e) $T_2$ admits a set-ordered felicitous-difference total coloring $h^*_{setY}$.}}
\end{figure}

\vskip 0.4cm

The above set-dual type labelings from \textbf{Set-Dual-1} to \textbf{Set-Dual-4} produce the following coloring matchings:

\begin{asparaenum}[\textrm{\textbf{Matching}}-1. ]
\item The set-ordered graceful matching $\langle f, f^*_{dual}\rangle $ holds $f(w)+f^*_{dual}(w)=q$ for $w\in V(G)$ and $f(x_iy_j)+f^*_{dual}(x_iy_j)=q+1$ for $x_iy_j\in E(G)$.

\item $\langle g_{setXY}, g^*_{setXY}\rangle $ is a \emph{matching of a set-ordered graceful labeling and a graceful-difference total labeling}, such that $|g_{setXY}(x_iy_j)-g^*_{setXY}(x_iy_j)|$ is equal to a constant for $x_iy_j\in E(G)$.

\item $\langle h_{setX}, h^*_{setY}\rangle $ is a \emph{matching of two set-ordered edge-magic total labelings}, such that $h_{setY}(x_iy_j)+h^*_{setY}(x_iy_j)=q+1$ for $x_iy_j\in E(G)$.

\item $\langle h^*_{setX}, h_{setY}\rangle $ is a \emph{matching of two set-ordered felicitous-difference total labelings}, such that $h_{setY}(x_iy_j)+h^*_{setY}(x_iy_j)=q+1$ for $x_iy_j\in E(G)$.
\end{asparaenum}

\begin{defn} \label{defn:111111}
$^*$ If there is at least a pair of vertices colored with the same color in Set-Dual-$k$ for $k\in [1,4]$ above, we get: the dual total coloring, the $XY$-set-dual total coloring, the $X$-set-dual total coloring and the $Y$-set-dual total coloring, as well as four (set-ordered) $W$-magic total colorings, where ``$W$-magic'' is one of edge-magic, edge-difference, felicitous-difference, graceful-difference (see examples shown in Fig.\ref{fig:hu-set-dual-11}).\qqed
\end{defn}

\section{Algorithms of adding leaves randomly}

In this subsection, the sentence ``RANDOMLY-LEAF-adding algorithm'' is abbreviated as ``RLA-algorithm''. For constructing multiple-operation graphic lattices, we introduce the following RLA-algorithms.

\subsection{RLA-algorithm-A of the odd-edge graceful-difference total coloring}

\noindent \textbf{RLA-algorithm-A of the odd-edge graceful-difference total coloring.}

\textbf{Input:} A connected bipartite $(p,q)$-graph $G$ admitting a set-ordered odd-edge graceful-difference total labeling $f_{grd}$.

\textbf{Output:} A connected bipartite $(p+m,q+m)$-graph $G_A$ admitting an odd-edge graceful-difference coloring $f_{grd}^*$, where $G_A$, called \emph{leaf-added graph}, is the result of adding randomly $m$ leaves to $G$.

\textbf{Initialization.} A connected bipartite $(p,q)$-graph $G$ has its own vertex set $V(G)=X\cup Y$ with $X\cap Y=\emptyset$, where $X=\{x_1,x_2,\dots ,x_s\}$ and $Y=\{y_1,y_2,\dots ,y_t\}$ with $s+t=p=|V(G)|$. By the definition of a set-ordered odd-edge graceful-difference total labeling, so we have the set-ordered restriction $\max f_{grd}(X)<\min f_{grd}(Y)$, without loss of generality,
$$
0=f_{grd}(x_1)<f_{grd}(x_2)<\cdots <f_{grd}(x_s)<f_{grd}(y_1)<f_{grd}(y_2)<\cdots <f_{grd}(y_t)=2q-1
$$
so each color $f_{grd}(x_i)$ for $i\in[1,s]$ is even, and each color $f_{grd}(y_j)$ for $j\in[1,t]$ is odd, and each edge $x_iy_j\in E(G)$ satisfies
\begin{equation}\label{eqa:odd-edge-graceful-difference-colorings}
\big | |f_{grd}(x_i)-f_{grd}(y_j)|-f_{grd}(x_iy_j)\big |=N_A\geq 0
\end{equation} as well as edge color set $f_{grd}(E(G))=\{f_{grd}(x_iy_j):x_iy_j\in E(G)\}=[1,2q-1]^o$.

Adding randomly $a_i$ new leaves $u_{i,k}\in L(x_i)=\{u_{i,k}:k\in [1,a_i]\}$ to each vertex $x_i\in X\subset V(G)$ by joining $u_{i,k}$ with $x_i$ together by new edges $x_iu_{i,k}$ for $k\in [1,a_i]$ and $i\in[1,s]$, and adding randomly $b_j$ new leaves $v_{j,r}\in L(y_j)=\{v_{j,r}:r\in [1,b_j]\}$ to each vertex $y_j\in Y\subset V(G)$ by joining $v_{j,r}$ with $y_j$ together by new edges $y_jv_{j,r}$ for $r\in [1,b_j]$ and $j\in[1,t]$, it may happen some $a_i=0$ or some $b_j=0$ here. The resultant graph is denoted as $G_A$.

Let $m=M_X+M_Y$, where $M_X=\sum ^{s}_{c=1}a_{c}$ and $M_Y=\sum ^{t}_{c=1}b_{c}$. Suppose that $f_{grd}(y_j)-f_{grd}(x_i)\geq f_{grd}(x_iy_j)\geq 1$ in Eq.(\ref{eqa:odd-edge-graceful-difference-colorings}), we define a coloring $f_{grd}^*$ of the leaf-added graph $G_A$ in the following steps.

\textbf{Step A-1.} Color edges $y_{j}v_{j,r}$ for leaves $v_{j,r}\in L(y_j)$ with $r\in [1,b_j]$ and $j\in[1,t]$ as:

(A-11) $f_{grd}^*(y_tv_{t,r})=2r-1$ for $r\in [1,b_t]$, $f_{grd}^*(y_tv_{t,b_t})=2b_t-1$;

(A-12) $f_{grd}^*(y_{t-1}v_{t-1,r})=2b_t+2r-1$ for $r\in [1,b_{t-1}]$, $f_{grd}^*(y_{t-1}v_{t-1,b_{t-1}})=2b_t+2b_{t-1}-1$;

(A-13) $f_{grd}^*(y_{t-j}v_{t-j,r})=2r-1+2\sum ^{t}_{c=t-j+1}b_{c}$ for $r\in [1,b_{t-j}]$ and $j\in[1,t-1]$, $f_{grd}^*(y_{t-j}v_{t-j,b_{t-j}})=2b_{t-j}-1+2\sum ^{t}_{c=t-j+1}b_{c}$;

(A-14) $f_{grd}^*(y_1v_{1,r})=2r-1+2\sum ^{t}_{c=2}b_{c}$ for $r\in [1,b_{1}]$, the last edge $y_1v_{1,b_{1}}$ is colored as
$$
f_{grd}^*(y_1v_{1,b_{1}})=2b_{1}-1+2\sum ^{t}_{c=2}b_{c}=-1+2\sum ^{t}_{c=1}b_{c }=2M_Y-1
$$

\textbf{Step A-2.} Color edges $x_iu_{i,k}$ for leaves $u_{i,k}\in L(x_i)$ with $k\in [1,a_i]$ and $i\in[1,s]$ as:

(A-21) $f_{grd}^*(x_su_{s,k})=2M_Y+2k-1$ for $k\in [1,a_s]$, $f_{grd}^*(x_su_{s,a_s})=2M_Y+2a_s-1$;

(A-22) $f_{grd}^*(x_{s-1}u_{s-1,k})=2M_Y+2a_s+2k-1$ for $k\in [1,a_{s-1}]$, $f_{grd}^*(x_{s-1}u_{s-1,a_{s-1}})=2M_Y+2a_s+2a_{s-1}-1$;

(A-23) $f_{grd}^*(x_{s-i}u_{s-i,k})=2k-1+2M_Y+2\sum ^{s}_{c=s-i+1}a_{c}$ for $k\in [1,a_i]$ and $i\in[1,s-1]$, $f_{grd}^*(x_{s-i}u_{s-i,a_i})=2a_i-1+2M_Y+2\sum ^{s}_{c=s-i+1}a_{c}$;

(A-24) $f_{grd}^*(x_1u_{1,k})=2k-1+2M_Y+2\sum ^{s}_{c=2}a_{c}$ for $k\in [1,a_{1}]$, the last edge $x_1u_{1,a_1}$ is colored with
$$
f_{grd}^*(x_1u_{1,a_1})=2a_1-1+2M_Y+2\sum ^{s}_{c=2}a_{c}=2M_Y-1+2\sum ^{s}_{c=1}a_{c}=2(M_Y+M_X)-1
$$

\textbf{Step A-3.} Recolor each element of $V(G)\cup E(G)$ by $f_{grd}^*(w)=f_{grd}(w)+2(M_Y+M_X)$ for $w\in E(G)$, and $f_{grd}^*(z)=f_{grd}(z)$ for $z\in V(G)$. So
\begin{equation}\label{eqa:graceful-difference-total-aa}
{
\begin{split}
\big | |f_{grd}^*(x_i)-f_{grd}^*(y_j)|-f_{grd}^*(x_iy_j)\big |&=\big | |f_{grd}(x_i)-f_{grd}(y_j)|-f_{grd}(x_iy_j)-2(M_Y+M_X)\big |\\
&=\big |N_A-2(M_Y+M_X)\big |.
\end{split}}
\end{equation} Let $N_A^*=\big |N_A-2(M_Y+M_X)\big |=|N_A-2m|$. Thereby, the edge color set $f_{grd}^*(E(G_A))$ of the leaf-added graph $G_A$ is as
\begin{equation}\label{eqa:555555}
f_{grd}^*(E(G_A))=\big [1,2(M_Y+M_X)+\max f_{grd}(E(G))\big ]^o=\big [1,2(m+q)-1\big ]^o
\end{equation}

\textbf{Step A-4.} Color the added leaves of $L(y_j)$ and $L(x_i)$ with $j\in[1,t]$ and $i\in[1,s]$.

\textbf{Step A-4.1.} Each leaf $v_{j,r}\in L(y_j)$ with $r\in [1,b_j]$ and $j\in[1,t]$ is colored by
\begin{equation}\label{eqa:graceful-difference-total-bb}
f_{grd}^*(v_{j,r})=N_A^*+f_{grd}^*(y_j)+f_{grd}^*(y_jv_{j,r})=N_A^*+f_{grd}(y_j)+f_{grd}^*(y_jv_{j,r})
\end{equation} which induces $\big | |f_{grd}^*(y_{j})-f_{grd}^*(v_{j,r})|-f_{grd}^*(y_{j}v_{j,r})\big |=N_A^*$ for $r\in [1,b_{j}]$ and $j\in[1,t]$.

\textbf{Step A-4.2.} Each leaf $u_{i,k}\in L(x_i)$ with $k\in [1,a_i]$ and $i\in[1,s]$ is colored by
\begin{equation}\label{eqa:graceful-difference-total-cc}
f_{grd}^*(u_{i,k})=N_A^*+f_{grd}^*(x_i)+f_{grd}^*(x_iu_{i,k})=N_A^*+f_{grd}(x_i)+f_{grd}^*(x_iu_{i,k})
\end{equation} which induces $\big | |f_{grd}^*(x_{i})-f_{grd}^*(u_{i,k})|-f_{grd}^*(x_{i}u_{i,k})\big |=N_A^*$ for $k\in [1,a_i]$ and $i\in[1,s]$.

\textbf{Step A-5.} Return the odd-edge graceful-difference total coloring $f_{grd}^*$ of the leaf-added graph $G_A$, and $f_{grd}^*(u_0)=0$ for some $u_0\in V(G_A)$.

\vskip 0.4cm

\begin{example}\label{exa:8888888888}
The examples shown in Fig.\ref{fig:A-graceful-difference} are for understanding the RLA-algorithm-A of the odd-edge graceful-difference total coloring:

(a) the graph $G_{\textrm{XYR-twin}}$ admits a twin set-ordered odd-edge graceful-difference total labeling $\alpha_{grd}$ holding $f_{grd}(V(G_{\textrm{XYR}}))\cup \alpha_{grd}(V(G_{\textrm{XYR-twin}}))\subseteq [0,20]$;

(b) the graph $G_{\textrm{XYR}}$ admits a set-ordered odd-edge graceful-difference total coloring $f_{grd}$ holding $\big | |f_{grd}(x)-f_{grd}(y)|-f_{grd}(xy)\big |=0$;

(c) the graph $G_{\textrm{XYR-leaf}}$ admits a set-ordered odd-edge graceful-difference total coloring $f_{grd}^*$ holding $\big | |f_{grd}^*(x)-f_{grd}^*(y)|-f_{grd}^*(xy)\big |=26$;

(d) the graph $H_{\textrm{XYR-leaf}}$ admits a set-ordered odd-edge graceful-difference total labeling $h_{grd}^*$ holding $\big | |h_{grd}^*(x)-h_{grd}^*(y)|-h_{grd}^*(xy)\big |=26$. There is a graph homomorphism $G_{\textrm{XYR-leaf}}\rightarrow H_{\textrm{XYR-leaf}}$.\qqed
\end{example}

\begin{figure}[h]
\centering
\includegraphics[width=16.4cm]{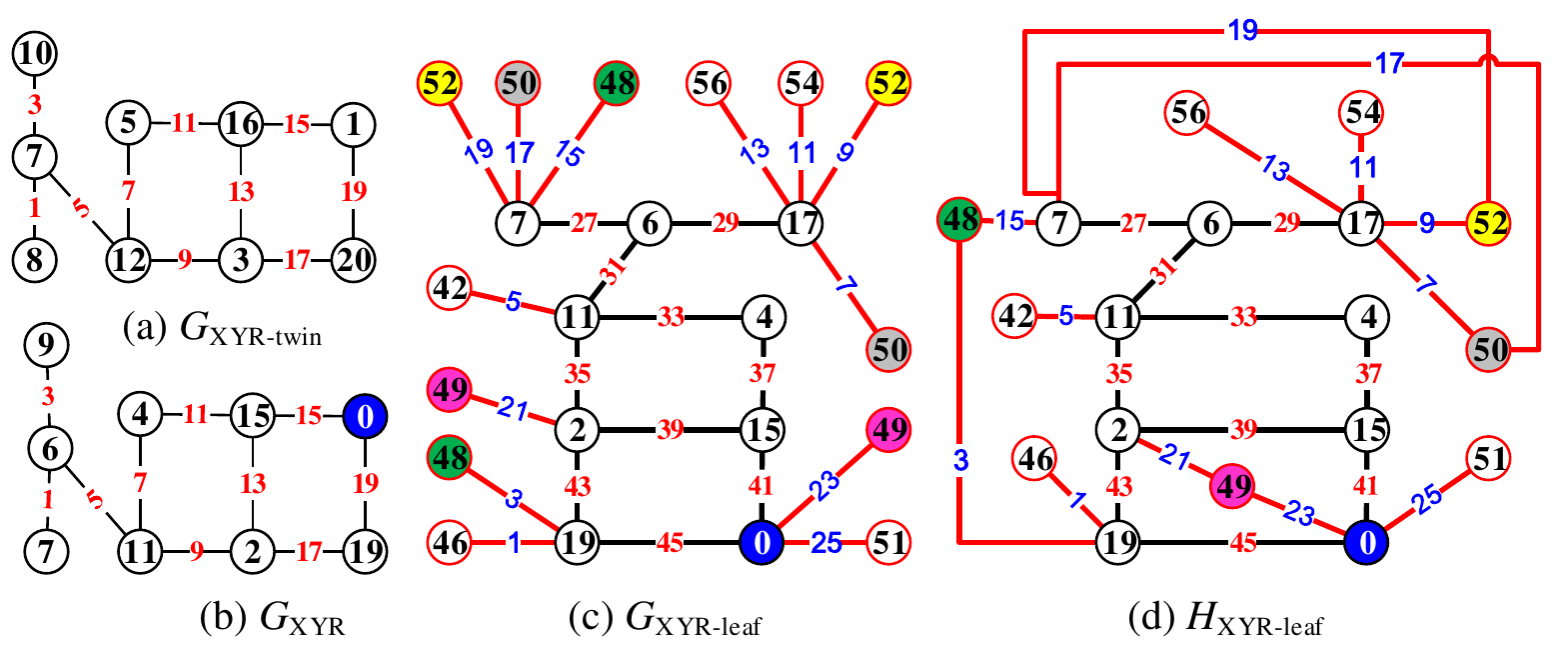}\\
\caption{\label{fig:A-graceful-difference} {\small A scheme for understanding RLA-algorithm-A of the odd-edge graceful-difference total coloring.}}
\end{figure}

\begin{problem}\label{problem:odd-edge-graceful-difference}
In the RLA-algorithm-A of the odd-edge graceful-difference total coloring, there are the following problems:

(i) \textbf{Integer Partition Problem.} We can select $k$ vertices from a $(p,q)$-graph $G$ for adding $m$ leaves to them, then we have $A^k_p=p(p-1)\cdots (p-k+1)$ selections, rather than ${p \choose k}=\frac{p!}{k!(p-k)!}$. Next, we decompose $m$ into a group of $k$ parts $m_1,m_2,\cdots ,m_k$ holding $m=m_1+m_2+\cdots +m_k$ with each $m_i\neq 0$. Suppose there is $n(m,k)$ groups of such $k$ parts. For a group of $k$ parts $m_1,m_2,\cdots ,m_k$, let $m_{i_1},m_{i_2},\cdots ,m_{i_k}$ be a permutation of $m_1,m_2,\cdots ,m_k$, so we have the number of such permutations is a factorial $k!$. Since the $(p,q)$-graph $G$ is colored well by the $W$-magic coloring/labeling $f$, then we have the number $A_{leaf}(G,m)$ of all possible adding $m$ leaves as follows
\begin{equation}\label{eqa:c3xxxxx}
A_{leaf}(G,m)=\sum^m_{k=1}A^k_p\cdot n(m,k)\cdot k!=\sum^m_{k=1} n(m,k)\cdot p!
\end{equation} where $n(m,k)=\sum ^k_{r=1}n(m-k,r)$. Here, computing $n(m,k)$ can be transformed into finding the number $A(m,k)$ of solutions of \emph{Diophantine equation} $m=\sum ^k_{i=1}ix_i$. There is a recursive formula
\begin{equation}\label{eqa:c3xxxxx}
A(m,k)=A(m,k-1)+A(m-k,k)
\end{equation}
with $0 \leq k\leq m$. It is not easy to compute the exact value of $A(m,k)$, for example, the authors in \cite{Shuhong-Wu-Accurate-2007} and \cite{WU-Qi-qi-2001} computed exactly
$${
\begin{split}
A(m,6)=&\biggr\lfloor \frac{1}{1036800}(12m^5 +270m^4+1520m^3-1350m^2-19190m-9081)+\\
&\frac{(-1)^m(m^2+9m+7)}{768}+\frac{1}{81}\left[(m+5)\cos \frac{2m\pi}{3}\right ]\biggr\rfloor
\end{split}}
$$

For any odd integer $m\geq 7$ it was conjectured $m=p_1+p_2+p_3$ with three primes $p_1,p_2,p_3$ from the famous Goldbach's conjecture: ``Every even integer, greater than 2, can be expressed as the sum of two primes.'' In other word, determining $A(m,3)$ is difficult, also, it is difficult to express an odd integer $m=\sum^{3n}_{k=1} p\,'_k$ with each $p\,'_k$ is a prime.

(ii) \textbf{Estimate} the extremum number
\begin{equation}\label{eqa:555555}
\min\{\max f_{grd}^*(V(G_A)):~f_{grd}^*\textrm{ is an odd-edge graceful-difference total coloring of}~G_A\}
\end{equation} over all odd-edge graceful-difference total colorings of the leaf-added graph $G_A$.

(iii) Notice that there are $m!$ permutations $w_{i_1},w_{i_2},\dots ,w_{i_m}$ from the added leaves of the leaf-added set $L^*(G_A)=\big (\bigcup^s_{i=1}L(x_i)\big )\bigcup \big (\bigcup^t_{j=1}L(y_j)\big )$ of the leaf-added $(p+m,q+m)$-graph $G_A$, notice that the leaf permutation in the RLA-algorithm-A of the odd-edge graceful-difference total coloring is one of these $m!$ permutations. We define a new coloring $F_{grd}$ for the leaf-added graph $G_A$ as: Color leaf-edge $w_{i_j}z_{i_j}$ with $F_{grd}(w_{i_j}z_{i_j})=2j-1$ for $j\in [1,m]$, where $z_{i_j}\in X\cup Y=V(G)$, and each element $e\in V(G)\cup E(G)=[V(G_A)\cup E(G_A)]\setminus L^*(G_A)$ is colored as $F_{grd}(e)=f_{ed}^*(e)$, as well as color each added leaf $w_{i_j}\in L^*(G_A)$ by
$$
F_{grd}(w_{i_j})=N_A^*+F_{grd}^*(z_{i_j})+F_{grd}(w_{i_j}z_{i_j}),~N_A^*=|N_A-2m|
$$ By Eq.(\ref{eqa:graceful-difference-total-aa}), Eq.(\ref{eqa:graceful-difference-total-bb}) and Eq.(\ref{eqa:graceful-difference-total-cc}), the coloring $F_{grd}$ is an odd-edge graceful-difference total coloring based on a permutation $w_{i_1},w_{i_2},\dots ,w_{i_m}$.

(vi) In Fig.\ref{fig:A-graceful-difference}, the Topcode-matrix $T_{code}(G_{\textrm{XYR-twin}})$ can be directly obtained from the Topcode-matrix $T_{code}(G_{\textrm{XYR}})$, and these two Topcode-matrices induce two graph sets $S(T_{code}(G_{\textrm{XYR-twin}}))$ and $S(T_{code}(G_{\textrm{XYR}}))$, we call theses two graph sets as \emph{twin odd-edge graceful-difference graph sets}. Thereby, a \emph{public-key graph} $G\in S(T_{code}(G_{\textrm{XYR}}))$ may correspond many \emph{private-key graphs} $S(T_{code}(G_{\textrm{XYR-twin}}))$. The Topcode-matrix $T_{code}(G_{\textrm{XYR-leaf}})$ can be made by adding leaves to the Topcode-matrix $T_{code}(G_{\textrm{XYR}})$.
\end{problem}

\begin{defn} \label{defn:derived-odd-edge-edge-difference-total}
$^*$ Suppose that a connected $(p,q)$-graph $G$ admits an odd-edge graceful-difference total coloring $f$, so $\big ||f(u)-f(v)|-f(uv)\big |=c$ for each edge $uv\in E(G)$, where $c$ is a non-negative integer.

(1) Let $w_{i_1},w_{i_2},\dots ,w_{i_p}$ be a permutation of vertices of $V(G)$, and $e_{j_1},e_{j_2},\dots ,e_{j_q}$ be a permutation of edges of $E(G)$, and $c_{i_1},c_{i_2},\dots ,c_{i_p}$ be a permutation of vertex colors of $f(V(G))$ and $c_{j_1},c_{j_2},\dots ,c_{j_q}$ be a permutation of edge colors of $f(E(G))$. We define a new total coloring $\theta$ for $G$ as: $\theta(w_{i_k})=c_{i_k}$ for $k\in [1,p]$, and $\theta(e_{j_s})=c_{j_s}$ for $s\in [1,q]$, such that

(i) each edge $uv\in E(G)$ corresponds two different vertices $x,z\in V(G)$ holding $\big ||\theta(x)-\theta(z)|-\theta(uv)\big |=c$; and

(ii) each vertex $y\in V(G)$ corresponds another vertex $t\in V(G)$ and an edge $ab\in E(G)$ holding $\big ||\theta(y)-\theta(t)|-\theta(ab)\big |=c$. We call $\theta$ a \emph{ve-separably derived odd-edge graceful-difference total coloring} of the total coloring $f$.

(2) Let $z_{i_1},z_{i_2},\dots ,z_{i_{p+q}}$ be a permutation of vertices and edges of $V(G)\cup E(G)$, and let $b_{j_1}$, $b_{j_2}$, $\dots ,b_{j_{p+q}}$ be a permutation of elements of $f(V(G))\cup f(E(G))$. We define a new total coloring $\varphi$ for $G$ as: $\varphi(z_{i_s})=b_{j_s}$ for $s\in [1,p+q]$, such that each element $z\in V(G)\cup E(G)$ corresponds two elements $x,y\in V(G)\cup E(G)$ holding one of $\big ||\varphi(z)-\varphi(x)|-\varphi(y)\big |=c$ and $\big ||\varphi(y)-\varphi(x)|-\varphi(z)\big |=c$ true, we call $\varphi$ a \emph{derived odd-edge graceful-difference total coloring} of the total coloring $f$.\qqed
\end{defn}

\begin{rem}\label{rem:333333}
About Definition \ref{defn:derived-odd-edge-edge-difference-total}, we have the following facts:

(1) Similarly with Definition \ref{defn:derived-odd-edge-edge-difference-total}, there are six derived-type magic-type total colorings: (ve-separably) derived odd-edge edge-difference total coloring, (ve-separably) derived odd-edge felicitous-difference total coloring, and (ve-separably) derived odd-edge edge-magic total coloring.

(2) Suppose that the connected $(p,q)$-graph $G$ admits $M$ odd-edge graceful-difference total colorings. For each odd-edge graceful-difference total coloring $f$ of these $M$ colorings, we have $(p!)^2(q!)^2$ ve-separably derived odd-edge graceful-difference total colorings of the total coloring $f$, we put them into a set $S^{gr}_{ve}\langle G,f\rangle $, then we get $M\cdot (p!)^2(q!)^2$ ve-separably derived odd-edge graceful-difference total colorings in total.

(3) Two Topcode-matrices $T_{code}(G,\theta)\neq T_{code}(G,\varphi)$ for two different odd-edge graceful-difference total colorings $\theta,\varphi\in S^{gr}_{ve}\langle G,f\rangle $, in general. And moreover, $T_{code}(G,\theta)\neq T_{code}(G,\varphi)$ for $\theta\in S^{gr}_{ve}\langle G,f\rangle $ and $\varphi\in S^{gr}_{ve}\langle G,g\rangle $, where $g$ is another odd-edge graceful-difference total coloring of $G$.

(4) These number-based strings $c_{i_1}c_{i_2}\cdots c_{i_p}c_{j_1}c_{j_2}\cdots c_{j_q}$ and $b_{j_1}b_{j_2}\cdots b_{j_{p+q}}$ defined in Definition \ref{defn:derived-odd-edge-edge-difference-total} differ from number-based strings $d_{i_1}d_{i_2}\cdots d_{i_{3q}}$ made by Topcode-matrices like $T_{code}(G,\theta)$ and $T_{code}(G,\varphi)$.\paralled
\end{rem}

\subsection{RLA-algorithm-B of the odd-edge edge-difference total coloring}

\noindent \textbf{RLA-algorithm-B of the odd-edge edge-difference total coloring.}

\textbf{Input:} A connected bipartite $(p,q)$-graph $G$ admitting a set-ordered odd-edge edge-difference total labeling $f_{edd}$.

\textbf{Output:} A connected bipartite $(p+m,q+m)$-graph $G_B$ admitting an odd-edge edge-difference total coloring $f_{edd}^*$, where $G_B$, called \emph{leaf-added graph}, is the result of adding randomly $m$ leaves to $G$.

\textbf{Initialization.} A connected bipartite $(p,q)$-graph $G$ has its own vertex set $V(G)=X\cup Y$ with $X\cap Y=\emptyset$, where $X=\{x_1,x_2,\dots ,x_s\}$ and $Y=\{y_1,y_2,\dots ,y_t\}$ with $s+t=p=|V(G)|$. By the definition of a set-ordered odd-edge edge-difference total labeling, so the set-ordered restriction $\max f_{edd}(X)<\min f_{edd}(Y)$ holds true, without loss of generality, there are inequalities
$$
0=f_{edd}(x_1)<f_{edd}(x_2)<\cdots <f_{edd}(x_s)<f_{edd}(y_1)<f_{edd}(y_2)<\cdots <f_{edd}(y_t)=2q-1
$$
so each color $f_{edd}(x_i)$ for $i\in[1,s]$ is even, and each color $f_{edd}(y_j)$ for $j\in[1,t]$ is odd, and each edge $x_iy_j\in E(G)$ holds
\begin{equation}\label{eqa:555555}
f_{edd}(x_iy_j)+\big |f_{edd}(x_i)-f_{edd}(y_j)\big |=N_B>0
\end{equation}
true, as well as $f_{edd}(E(G))=\{f_{edd}(x_iy_j):x_iy_j\in E(G)\}=[1,2q-1]^o$.

Adding randomly $a_i$ new leaves $u_{i,k}\in L(x_i)=\{u_{i,k}:k\in [1,a_i]\}$ to each vertex $x_i\in X\subset V(G)$ by joining $u_{i,k}$ with $x_i$ together by new edges $x_iu_{i,k}$ for $k\in [1,a_i]$ and $i\in[1,s]$, and adding randomly $b_j$ new leaves $v_{j,r}\in L(y_j)=\{v_{j,r}:r\in [1,b_j]\}$ to each vertex $y_j\in Y\subset V(G)$ by joining $v_{j,r}$ with $y_j$ together by new edges $y_jv_{j,r}$ for $r\in [1,b_j]$ and $j\in[1,t]$, it may happen some $a_i=0$ or some $b_j=0$. The resultant graph is denoted as $G_B$.

Let $M_X=\sum ^{s}_{c=1}a_{c}$ and $M_Y=\sum ^{t}_{c=1}b_{c}$, so $m=M_X+M_Y$. Define a coloring $f_{edd}^*$ of the leaf-added graph $G_B$ in the following steps.

\textbf{Step B-1.} Color edges $x_iu_{i,k}$ for leaves $u_{i,k}\in L(x_i)$ with $k\in [1,a_i]$ and $i\in[1,s]$ as follows:

(B-11) $f_{edd}^*(x_1u_{1,k})=2k-1$ for $k\in [1,a_1]$, $f_{edd}^*(x_1u_{1,a_1})=2a_s-1$;

(B-12) $f_{edd}^*(x_{i}u_{i,k})=2k-1+2\sum ^{i-1}_{c=1}a_{c}$ for $k\in [1,a_i]$ and $i\in[2,s]$, $f_{edd}^*(x_{i}u_{i,a_i})=2a_i-1+2\sum ^{i-1}_{c=1}a_{c}=-1+2\sum ^{i}_{c=1}a_{c}$; and

(B-13) The last edge $x_{s}u_{s,a_s}$ is colored by $f_{edd}^*(x_{s}u_{s,a_s})=2a_s-1+2\sum ^{s-1}_{c=1}a_{c}=-1+2\sum ^{s}_{c=1}a_{c}=2M_X-1$.

\textbf{Step B-2.} Color edges $y_{j}v_{j,r}$ for leaves $v_{j,r}\in L(y_j)$ with $r\in [1,b_j]$ and $j\in[1,t]$ as follows:

(B-21) $f_{edd}^*(y_{1}v_{1,r})=2M_X+2r-1$ for $k\in [1,b_1]$, $f_{edd}^*(y_{1}v_{1,b_1})=2M_X+2b_1-1$;

(B-22) $f_{edd}^*(y_{j}v_{j,r})=2M_X+2r-1+2\sum ^{j-1}_{c=1}b_{c}$ for $r\in [1,b_j]$ and $j\in[1,t]$, and the last edge $x_{s}u_{s,a_s}$ is colored by
$$
f_{edd}^*(y_{t}v_{t,b_t})=2M_X+2b_t-1+2\sum ^{t-1}_{c=1}b_{c}=2(M_X+M_Y)-1=2m-1
$$ The edge color set is
\begin{equation}\label{eqa:edge-difference-total00}
{
\begin{split}
f^*_{edd}(E(G_B))=&\{f_{edd}^*(x_iu_{i,k}):k\in [1,a_i],i\in[1,s]\}\cup \\
&\cup \{f_{edd}^*(y_{j}v_{j,r}):r\in [1,b_j],j\in[1,t]\}\cup f^*_{edd}(E(G))\\
=&[1,2(m+q)-1]^o,
\end{split}}
\end{equation}

\textbf{Step B-3.} Recolor each element of $V(G)\cup E(G)$ as: $f_{edd}^*(w)=f_{edd}(w)+2m$ for $w\in E(G)$, and $f_{edd}^*(z)=f_{edd}(z)$ for $z\in V(G)$, which induces
\begin{equation}\label{eqa:555555}
f_{edd}^*(x_iy_j)+|f_{edd}^*(x_i)-f_{edd}^*(y_j)|=f_{edd}(x_iy_j)+2m+|f_{edd}(x_i)-f_{edd}(y_j)|=2m+N_B
\end{equation} for each edge $x_iy_j\in E(G)$.

\textbf{Step B-4.} Let $N_B^*=2m+N_B$. Color the added leaves of $L(y_j)$ and $L(x_i)$ with $j\in[1,t]$ and $i\in[1,s]$.

\textbf{Step B-4.1.} Each leaf $v_{j,r}\in L(y_j)$ with $r\in [1,b_j]$ and $j\in[1,t]$ is colored by
\begin{equation}\label{eqa:edge-difference-total11}
f_{edd}^*(v_{j,r})=N_B^*+f_{edd}^*(y_j)-f_{edd}^*(y_jv_{j,r})
\end{equation} so $f_{edd}^*(y_{j}v_{j,r})+|f_{edd}^*(y_{j})-f_{edd}^*(v_{j,r})|=N_B^*$ for $r\in [1,b_{j}]$ and $j\in[1,t]$.

\textbf{Step B-4.2.} Each leaf $u_{i,k}\in L(x_i)$ with $k\in [1,a_i]$ and $i\in[1,s]$ is colored by
\begin{equation}\label{eqa:edge-difference-total22}
f_{edd}^*(u_{i,k})=N_B^*+f_{edd}^*(x_i)-f_{edd}^*(x_iu_{i,k})
\end{equation} immediately, $f_{edd}^*(x_{i}u_{i,k})+|f_{edd}^*(x_{i})-f_{edd}^*(u_{i,k})|=N_B^*$ for $k\in [1,a_i]$ and $i\in[1,s]$.

\textbf{Step B-5.} Return the odd-edge edge-difference total coloring $f_{edd}^*$ of the leaf-added graph $G_B$.

\vskip 0.4cm

\begin{example}\label{exa:8888888888}
In Fig.\ref{fig:B-edge-difference}, for illustrating the RLA-algorithm-B of the odd-edge edge-difference total coloring, we can see the following examples:

(a) the graph $G_{\textrm{XY-twin}}$ admits a twin set-ordered odd-edge edge-difference total labeling $\beta_{edd}$ holding $f_{edd}(V(G_{\textrm{XY}}))\cup \beta_{edd}(V(G_{\textrm{XY-twin}}))\subseteq [0,20]$;

(b) the graph $G_{\textrm{XY}}$ admits a set-ordered odd-edge edge-difference total coloring $f_{edd}$ holding $f_{edd}(xy)+|f_{edd}(x)-f_{edd}(y)|=20$;

(c) the graph $G_{\textrm{XY-leaf}}$ admits a set-ordered odd-edge edge-difference total coloring $f_{edd}^*$ holding $f_{edd}^*(xy)+|f_{edd}^*(x)-f_{edd}^*(y)|=46$;

(d) the graph $H_{\textrm{XY-leaf}}$ admits a set-ordered odd-edge edge-difference total labeling $h_{edd}^*$ holding $h_{edd}^*(xy)+|h_{edd}^*(x)-h_{edd}^*(y)|=46$. Notice that $G_{\textrm{XY-leaf}}\rightarrow H_{\textrm{XY-leaf}}$.\qqed
\end{example}

\begin{figure}[h]
\centering
\includegraphics[width=16.4cm]{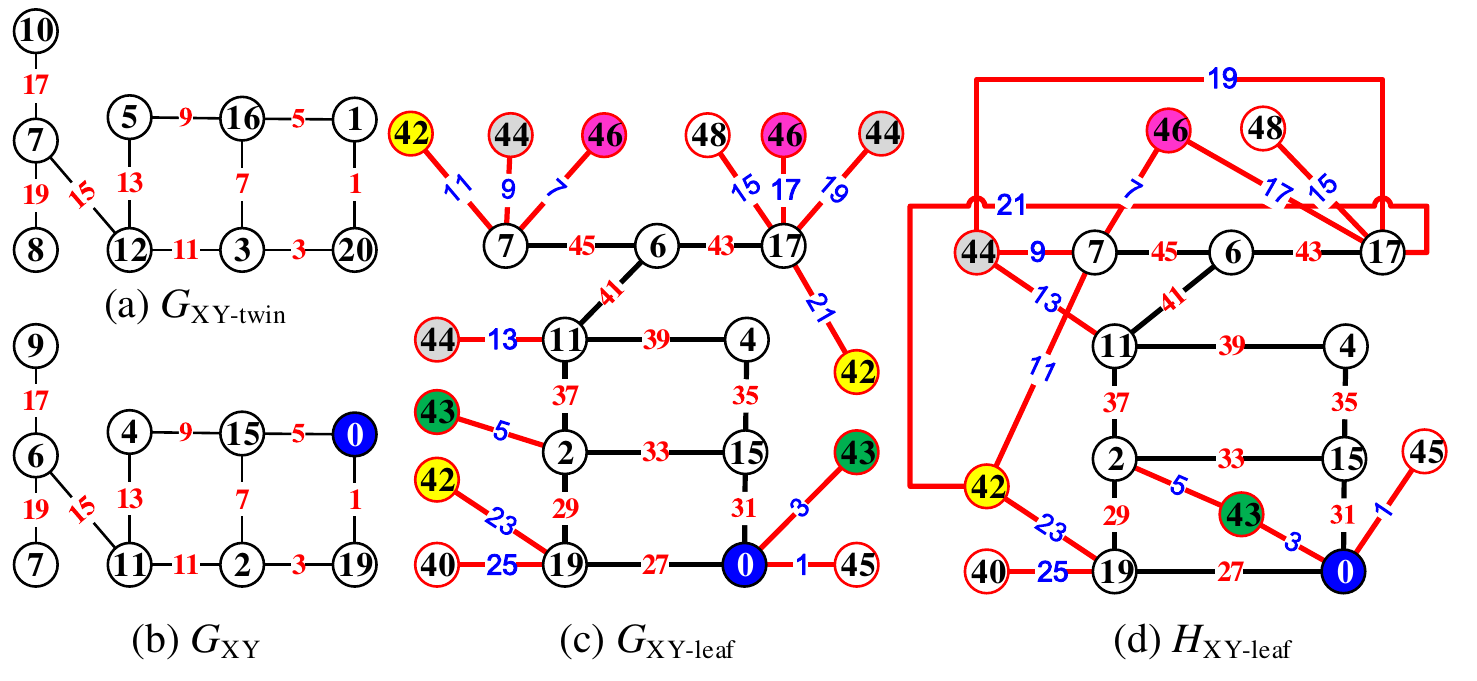}\\
\caption{\label{fig:B-edge-difference} {\small Examples for knowing RLA-algorithm-B of the odd-edge edge-difference total coloring.}}
\end{figure}

\begin{problem}\label{problem:555555}
In the RLA-algorithm-B of the odd-edge edge-difference total coloring, there are the following problems:

(i) \textbf{Estimate} the extremum number
\begin{equation}\label{eqa:555555}
\min\{\max f_{edd}^*(V(G_A)):~f_{edd}^*\textrm{ is an odd-edge edge-difference total coloring of}~G_B\}.
\end{equation} over all odd-edge edge-difference total colorings of the leaf-added graph $G_B$.

(ii) Since $m=M_X+M_Y$, there are $m!$ permutations $w_{i_1},w_{i_2},\dots ,w_{i_m}$ from the added leaves of the leaf-added set $L^*(G_B)=\big (\bigcup^s_{i=1}L(x_i)\big )\bigcup \big (\bigcup^t_{j=1}L(y_j)\big )$ of the leaf-added $(p+m,q+m)$-graph $G_B$, notice that the leaf permutation in the RLA-algorithm-B of the odd-edge edge-difference total coloring is one of these $m!$ permutations. We define a new coloring $F_{edd}$ for the leaf-added graph $G_B$ as: Color leaf-edge $w_{i_j}z_{i_j}$ with $F_{edd}(w_{i_j}z_{i_j})=2j-1$ for $j\in [1,m]$, where $z_{i_j}\in X\cup Y=V(G)$, and each element $e\in V(G)\cup E(G)=[V(G_B)\cup E(G_B)]\setminus L^*(G_B)$ is colored as $F_{edd}(e)=f_{edd}^*(e)$, as well as color each added leaf $w_{i_j}\in L^*(G_B)$ by
$$F_{edd}(w_{i_j})=N_B^*+F_{edd}^*(z_{i_j})-F_{edd}(w_{i_j}z_{i_j}),~N_B^*=2m+N_B
$$ By Eq.(\ref{eqa:edge-difference-total00}), Eq.(\ref{eqa:edge-difference-total11}) and Eq.(\ref{eqa:edge-difference-total22}), the coloring $F_{edd}$ is an odd-edge edge-difference total coloring based on a permutation $w_{i_1},w_{i_2},\dots ,w_{i_m}$.
\end{problem}

\subsection{RLA-algorithm-C of the odd-edge edge-magic total coloring}

\noindent \textbf{RLA-algorithm-C of the odd-edge edge-magic total coloring.}

\textbf{Input:} A connected bipartite $(p,q)$-graph $G$ admitting a set-ordered odd-edge edge-magic total labeling $f_{ema}$.

\textbf{Output:} A connected bipartite $(p+m,q+m)$-graph $G_C$ admitting an odd-edge edge-magic total coloring $f_{ema}^*$, where $G_C$, called \emph{leaf-added graph}, is the result of adding randomly $m$ leaves to $G$.

\textbf{Initialization.} A connected bipartite $(p,q)$-graph $G$ has its own vertex set $V(G)=X\cup Y$ with $X\cap Y=\emptyset$, where $X=\{x_1,x_2,\dots ,x_s\}$ and $Y=\{y_1,y_2,\dots ,y_t\}$ with $s+t=p=|V(G)|$. By the definition of a set-ordered odd-edge edge-magic total labeling, so we have the set-ordered restriction $\max f_{ema}(X)<\min f_{ema}(Y)$, without loss of generality, we have
$$
0=f_{ema}(x_1)<f_{ema}(x_2)<\cdots <f_{ema}(x_s)<f_{ema}(y_1)<f_{ema}(y_2)<\cdots <f_{ema}(y_t)=2q-1
$$
so each color $f_{ema}(x_i)$ for $i\in[1,s]$ is even, and each color $f_{ema}(y_j)$ for $j\in[1,t]$ is odd, and each edge $x_iy_j\in E(G)$ satisfies
\begin{equation}\label{eqa:555555}
f_{ema}(x_i)+f_{ema}(x_iy_j)+f_{ema}(y_j)=N_C> 0
\end{equation}
as well as $f_{ema}(E(G))=\{f_{ema}(x_iy_j):x_iy_j\in E(G)\}=[1,2q-1]^o$.

Adding randomly $a_i$ new leaves $u_{i,k}\in L(x_i)=\{u_{i,k}:k\in [1,a_i]\}$ to each vertex $x_i\in X\subset V(G)$ by joining $u_{i,k}$ with $x_i$ together by new edges $x_iu_{i,k}$ for $k\in [1,a_i]$ and $i\in[1,s]$, and adding randomly $b_j$ new leaves $v_{j,r}\in L(y_j)=\{v_{j,r}:r\in [1,b_j]\}$ to each vertex $y_j\in Y\subset V(G)$ by joining $v_{j,r}$ with $y_j$ together by new edges $y_jv_{j,r}$ for $r\in [1,b_j]$ and $j\in[1,t]$, it may happen some $a_i=0$ or some $b_j=0$. The resultant graph is denoted as $G_C$.

Let $M_X=\sum ^{s}_{c=1}a_{c}$ and $M_Y=\sum ^{t}_{c=1}b_{c}$, so $m=M_X+M_Y$. We define a coloring $f_{ema}^*$ of the leaf-added graph $G_C$ in the following steps:

\textbf{Step C-1.} Color edges $y_{j}v_{j,r}$ for leaves $v_{j,r}\in L(y_j)$ with $r\in [1,b_j]$ and $j\in[1,t]$ as follows:

(C-11) $f_{ema}^*(y_tv_{t,r})=2r-1$ for $r\in [1,b_t]$, $f_{ema}^*(y_tv_{t,b_t})=2b_t-1$;

(C-12) $f_{ema}^*(y_{t-1}v_{t-1,r})=2b_t+2r-1$ for $r\in [1,b_{t-1}]$, $f_{ema}^*(y_{t-1}v_{t-1,b_{t-1}})=2b_t+2b_{t-1}-1$;

(C-13) $f_{ema}^*(y_{t-j}v_{t-j,r})=2r-1+2\sum ^{t}_{c=t-j+1}b_{c}$ for $r\in [1,b_{t-j}]$ and $j\in[1,t-1]$, $f_{ema}^*(y_{t-j}v_{t-j,b_{t-j}})=2b_{t-j}-1+2\sum ^{t}_{c=t-j+1}b_{c}$;

(C-14) $f_{ema}^*(y_1v_{1,r})=2r-1+2\sum ^{t}_{c=2}b_{c}$ for $r\in [1,b_{1}]$, the last edge $y_1v_{1,b_{1}}$ is colored by
$$
f_{ema}^*(y_1v_{1,b_{1}})=2b_{1}-1+2\sum ^{t}_{c=2}b_{c}=-1+2\sum ^{t}_{c=1}b_{c }=2M_Y-1
$$

\textbf{Step C-2.} Color edges $x_iu_{i,k}$ for leaves $u_{i,k}\in L(x_i)$ with $k\in [1,a_i]$ and $i\in[1,s]$ as follows:

(C-21) $f_{ema}^*(x_su_{s,k})=2M_Y+2k-1$ for $k\in [1,a_s]$, $f_{ema}^*(x_su_{s,a_s})=2M_Y+2a_s-1$;

(C-22) $f_{ema}^*(x_{s-1}u_{s-1,k})=2M_Y+2a_s+2k-1$ for $k\in [1,a_{s-1}]$, $f_{ema}^*(x_{s-1}u_{s-1,a_{s-1}})=2M_Y+2a_s+2a_{s-1}-1$;

(C-23) $f_{ema}^*(x_{s-i}u_{s-i,k})=2k-1+2M_Y+2\sum ^{s}_{c=s-i+1}a_{c}$ for $k\in [1,a_i]$ and $i\in[1,s-1]$, $f_{ema}^*(x_{s-i}u_{s-i,a_i})=2a_i-1+2M_Y+2\sum ^{s}_{c=s-i+1}a_{c}$;

(C-24) $f_{ema}^*(x_1u_{1,k})=2k-1+2M_Y+2\sum ^{s}_{c=2}a_{c}$ for $k\in [1,a_{1}]$, and the last edge $x_1u_{1,a_1}$ is colored with
\begin{equation}\label{eqa:odd-edge-interval}
f_{ema}^*(x_1u_{1,a_1})=2a_1-1+2M_Y+2\sum ^{s}_{c=2}a_{c}=2M_Y-1+2\sum ^{s}_{c=1}a_{c}=2(M_Y+M_X)-1=2m-1.
\end{equation}

\textbf{Step C-3.} Recolor each element of $V(G)\cup E(G)$ in the following way: $f_{ema}^*(w)=f_{ema}(w)+2m$ for $w\in E(G)$, and $f_{ema}^*(z)=f_{ema}(z)$ for $z\in V(G)$. So, each edge $x_iy_j\in E(G)$ holds
\begin{equation}\label{eqa:odd-edge-interval-11}
f_{ema}^*(x_i)+f_{ema}^*(x_iy_j)+f_{ema}^*(y_j)=f_{ema}(x_i)+f_{ema}(x_iy_j)+2m+f_{ema}(y_j)=2m+N_C
\end{equation}
By Eq.(\ref{eqa:odd-edge-interval}) and Eq.(\ref{eqa:odd-edge-interval-11}), we have the edge color set $f_{ema}^*(E(G_C))$ of the graph $G_C$ as follows:
\begin{equation}\label{eqa:555555}
f_{ema}^*(E(G_C))=[1,2m-1]^o\cup [1+2m,2q-1+2m]^o=[1,2(q+m)-1]^o
\end{equation}

\textbf{Step C-4.} Let $N_C^*=2m+N_C$. Color the added leaves of $L(y_j)$ and $L(x_i)$ with $j\in[1,t]$ and $i\in[1,s]$.

\textbf{Step C-4.1.} Each leaf $v_{j,r}\in L(y_j)$ with $r\in [1,b_j]$ and $j\in[1,t]$ is colored by
\begin{equation}\label{eqa:odd-edge-interval-22}
f_{ema}^*(v_{j,r})=N_C^*-f_{ema}^*(y_j)-f_{ema}^*(y_jv_{j,r})
\end{equation} so $f_{ema}^*(y_{j})+f_{ema}^*(y_{j}v_{j,r})+f_{ema}^*(v_{j,r})=N_C^*$ for $r\in [1,b_{j}]$ and $j\in[1,t]$.

\textbf{Step C-4.2.} Each leaf $u_{i,k}\in L(x_i)$ with $k\in [1,a_i]$ and $i\in[1,s]$ is colored by
\begin{equation}\label{eqa:odd-edge-interval-33}
f_{ema}^*(u_{i,k})=N_C^*-f_{ema}^*(x_i)-f_{ema}^*(x_iu_{i,k})
\end{equation} immediately, $f_{ema}^*(x_{i})+f_{ema}^*(x_{i}u_{i,k})+f_{ema}^*(u_{i,k})=N_C^*$ for $k\in [1,a_i]$ and $i\in[1,s]$.

\textbf{Step C-5.} Return the odd-edge edge-magic total coloring $f_{ema}^*$ of the leaf-added graph $G_C$.

\vskip 0.4cm

\begin{example}\label{exa:8888888888}
For understanding the RLA-algorithm-C of the odd-edge edge-magic total coloring, see Fig.\ref{fig:C-edge-magic}, we have

(a) the graph $G_{\textrm{Y-twin}}$ admits a twin set-ordered odd-edge edge-magic total labeling $\gamma_{ema}$ holding $f_{ema}(V(G_{\textrm{Y}}))\cup \gamma_{ema}(V(G_{\textrm{Y-twin}}))\subseteq [0,20]$;

(b) the graph $G_{\textrm{Y}}$ admits a set-ordered odd-edge edge-magic total coloring $f_{ema}$ holding $f_{ema}(x)+f_{ema}(xy)+f_{ema}(y)=26$;

(c) the graph $G_{\textrm{Y-leaf}}$ admits a set-ordered odd-edge edge-magic total coloring $f_{ema}^*$ holding $f_{ema}^*(x)+f_{ema}^*(xy)+f_{ema}^*(y)=52$;

(d) the graph $H_{\textrm{Y-leaf}}$ admits a set-ordered odd-edge edge-magic total labeling $h_{ema}^*$ holding $h_{ema}^*(x)+h_{ema}^*(xy)+h_{ema}^*(y)=52$. There is a graph homomorphism $G_{\textrm{Y-leaf}}\rightarrow H_{\textrm{Y-leaf}}$.\qqed
\end{example}

\begin{figure}[h]
\centering
\includegraphics[width=16.4cm]{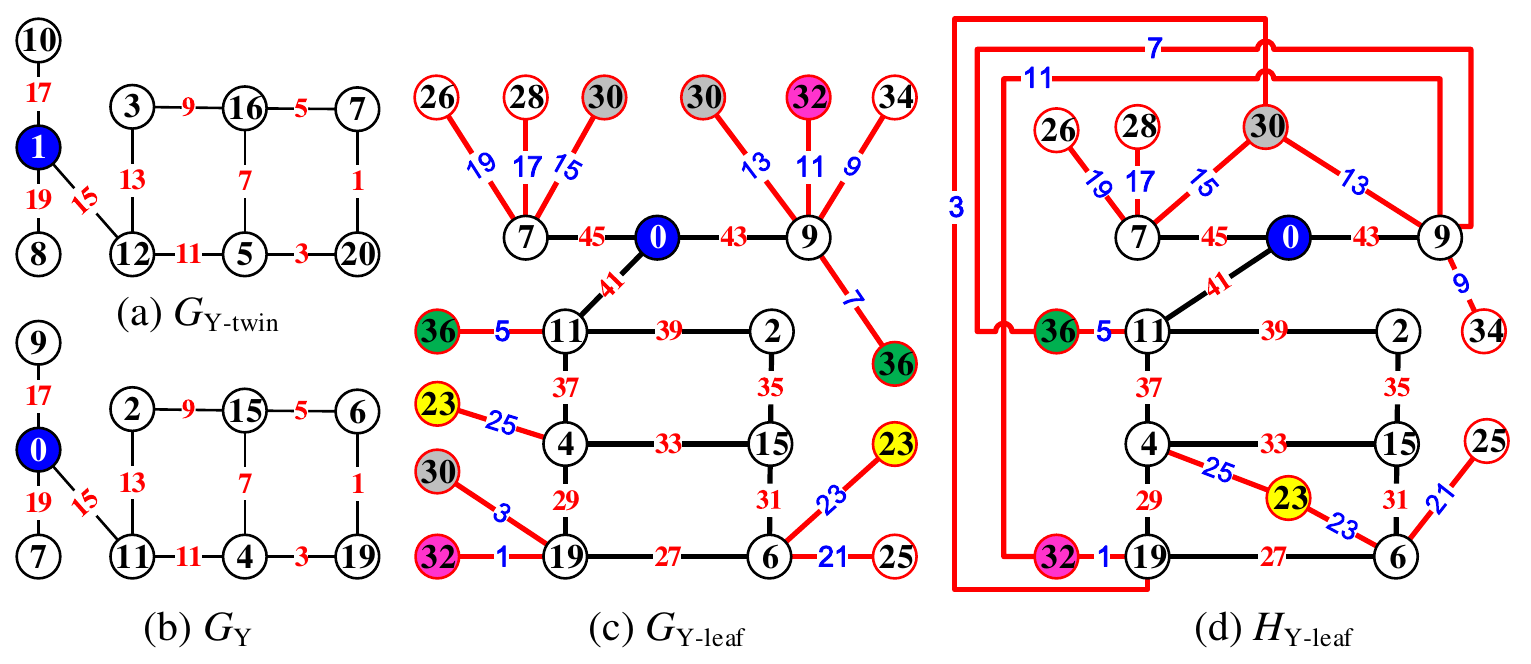}\\
\caption{\label{fig:C-edge-magic} {\small A scheme for understanding RLA-algorithm-C of the odd-edge edge-magic total coloring.}}
\end{figure}

\begin{problem}\label{problem:555555}
In the RLA-algorithm-C of the odd-edge edge-magic total coloring, there are the following problems:

(i) \textbf{Estimate} the extremum number
\begin{equation}\label{eqa:555555}
\min\{\max f_{ema}^*(V(G_A)):~f_{ema}^*\textrm{ is an odd-edge edge-magic total coloring of}~G_C\}
\end{equation} over all odd-edge edge-magic total colorings of the leaf-added graph $G_C$.

(ii) Notice that there are $m!$ permutations $w_{i_1},w_{i_2},\dots ,w_{i_m}$ from the added leaves of the leaf-added set $L^*(G_C)=\big (\bigcup^s_{i=1}L(x_i)\big )\bigcup \big (\bigcup^t_{j=1}L(y_j)\big )$ of the leaf-added $(p+m,q+m)$-graph $G_C$, notice that the leaf permutation in the RLA-algorithm-C of the odd-edge edge-magic total coloring is one of these $m!$ permutations. We define a new coloring $F_{ema}$ for the leaf-added graph $G_C$ as: Color leaf-edge $w_{i_j}z_{i_j}$ with $F_{ema}(w_{i_j}z_{i_j})=2j-1$ for $j\in [1,m]$, where $z_{i_j}\in X\cup Y=V(G)$, and each element $e\in V(G)\cup E(G)=[V(G_C)\cup E(G_C)]\setminus L^*(G_C)$ is colored as $F_{ema}(e)=f_{edd}^*(e)$, as well as color each added leaf $w_{i_j}\in L^*(G_C)$ by
$$F_{ema}(w_{i_j})=N_C^*-F_{ema}^*(z_{i_j})-F_{ema}(w_{i_j}z_{i_j}),~N_C^*=2m+N_C
$$ By Eq.(\ref{eqa:odd-edge-interval-11}), Eq.(\ref{eqa:odd-edge-interval-22}) and Eq.(\ref{eqa:odd-edge-interval-33}), the coloring $F_{ema}$ is an odd-edge edge-magic total coloring based on a permutation $w_{i_1},w_{i_2},\dots ,w_{i_m}$.
\end{problem}

\subsection{RLA-algorithm-D of the odd-edge felicitous-difference total coloring.}

\noindent \textbf{RLA-algorithm-D of the odd-edge felicitous-difference total coloring.}

\textbf{Input:} A connected bipartite $(p,q)$-graph $G$ admitting a set-ordered odd-edge felicitous-difference total labeling $h_{fed}$.

\textbf{Output:} A connected bipartite $(p+m,q+m)$-graph $G_D$ admitting an odd-edge felicitous-difference total coloring $h_{fed}^*$, where $G_D$, called \emph{leaf-added graph}, is the result of adding randomly $m$ leaves to $G$.

\textbf{Initialization.} A connected bipartite $(p,q)$-graph $G$ has its own vertex set $V(G)=X\cup Y$ with $X\cap Y=\emptyset$, where $X=\{x_1,x_2,\dots ,x_s\}$ and $Y=\{y_1,y_2,\dots ,y_t\}$ with $s+t=p=|V(G)|$. By the definition of a set-ordered odd-edge felicitous-difference total labeling, so we have the set-ordered restriction $\max h_{fed}(X)<\min h_{fed}(Y)$, without loss of generality,
$$
0=h_{fed}(x_1)<h_{fed}(x_2)<\cdots <h_{fed}(x_s)<h_{fed}(y_1)<h_{fed}(y_2)<\cdots <h_{fed}(y_t)=2q-1
$$
so each color $h_{fed}(x_i)$ for $i\in[1,s]$ is even, and each color $h_{fed}(y_j)$ for $j\in[1,t]$ is odd, and
\begin{equation}\label{eqa:555555}
\big |h_{fed}(x_i)+h_{fed}(y_j)-h_{fed}(x_iy_j)\big |=N_D\geq 0
\end{equation}
for each edge $x_iy_j\in E(G)$, as well as $h_{fed}(E(G))=\{h_{fed}(x_iy_j):x_iy_j\in E(G)\}=[1,2q-1]^o$.

Adding randomly $a_i$ new leaves $u_{i,k}\in L(x_i)=\{u_{i,k}:k\in [1,a_i]\}$ to each vertex $x_i\in X\subset V(G)$ by joining $u_{i,k}$ with $x_i$ together by new edges $x_iu_{i,k}$ for $k\in [1,a_i]$ and $i\in[1,s]$, and adding randomly $b_j$ new leaves $v_{j,r}\in L(y_j)=\{v_{j,r}:r\in [1,b_j]\}$ to each vertex $y_j\in Y\subset V(G)$ by joining $v_{j,r}$ with $y_j$ together by new edges $y_jv_{j,r}$ for $r\in [1,b_j]$ and $j\in[1,t]$, it may happen some $a_i=0$ or some $b_j=0$. The resultant graph is denoted as $G_D$. Let $A_X=\sum ^{s}_{l=1}a_l$ and $B_X=\sum ^{t}_{l=1}b_l$, so $m=A_X+B_X$.

Define a coloring $h_{fed}^*$ for $G_D$ in the following steps:

\textbf{Step D-1.} Color edges $x_iu_{i,k}$ by setting $h_{fed}^*(x_1u_{1,k})=2k-1$ for $k\in [1,a_1]$, and
\begin{equation}\label{eqa:555555}
h_{fed}^*(x_{i+1}u_{i+1,r})=2r-1+2\sum ^{i}_{l=1}a_l,~r\in [1,a_{i+1}],~i\in [1,s-1]
\end{equation}
and $h_{fed}^*(x_{s}u_{s,r})=2r-1+\sum ^{s-1}_{l=1}2a_l$ for $s\in [1,a_s]$, so the last edge $x_{s}u_{s,a_s}$ is colored with $h_{fed}^*(x_{s}u_{s,a_s})=-1+\sum ^{s}_{l=1}2a_l$.

\textbf{Step D-2.} Color edges $y_tv_{t,k}$ with $h_{fed}^*(y_tv_{t,k})=2A_X+2k-1$ for $k\in [1,b_t]$, and
$$h_{fed}^*(y_tv_{t,b_t})=2A_X-1+2b_t,~h_{fed}^*(y_{t-1}v_{t-1,k})=2A_X-1+2b_t+2k,~k\in [1,b_{t-1}]
$$ and $h_{fed}^*(y_{t-1}v_{t-1,b_{t-1}})=2A_X-1+2b_t+2b_{t-1}$, and
\begin{equation}\label{eqa:555555}
h_{fed}^*(y_{t-j}v_{t-j,k})=2k+2A_X-1+2\sum ^{t}_{l=t-j+1}b_l,~k\in [1,b_{t-j}],~j\in [1,t-2]
\end{equation} the last edge $y_{1}v_{1,b_1}$ is colored with $h_{fed}^*(y_{1}v_{1,b_1})=2A_X+2B_X-1=2m-1$. Thereby, the edge color set of the leaf-added graph $G_D$ is as
\begin{equation}\label{eqa:555555}
h_{fed}^*(E(G_D))=\big [1,2(A_X+B_X)+\max h_{fed}(E(G))\big ]^o=\big [1,2(q+m)-1\big ]^o.
\end{equation}

\textbf{Step D-3.} Recolor each element $w\in V(G)\cup E(G)$ with $h_{fed}^*(w)=h_{fed}(w)+2m$. So,
\begin{equation}\label{eqa:felicitous-difference-total-00}
\big |h_{fed}^*(x_i)+h_{fed}^*(y_j)-h_{fed}^*(x_iy_j)\big |=2m+\big |h_{fed}(x_i)+h_{fed}(y_j)-h_{fed}(x_iy_j)\big |=N_D+2m
\end{equation}

\textbf{Step D-4.} We color added-leaves $u_{i,k}\in L(x_i)$ with $h_{fed}^*(u_{i,k})=N_D+h_{fed}^*(x_iu_{i,k})-h_{fed}(x_i)$ for $k\in [1,a_i]$ and $i\in [1,s]$, since
\begin{equation}\label{eqa:felicitous-difference-total-11}
N_D+2m=\big |h_{fed}^*(u_{i,k})+h_{fed}^*(x_i)-h_{fed}^*(x_iu_{i,k})\big |=h_{fed}^*(u_{i,k})+h_{fed}(x_i)+2m-h_{fed}^*(x_iu_{i,k})
\end{equation}

Again, we color leaves $v_{j,r}\in L(y_j)$ with $h_{fed}^*(v_{j,r})=N_D+h_{fed}^*(v_{j,r}y_{j})-h_{fed}(y_j)$ for $r\in [1,b_{j}]$ and $j\in [1,t]$, because of
\begin{equation}\label{eqa:felicitous-difference-total-22}
N_D+2m=\big |h_{fed}^*(v_{j,r})+h_{fed}^*(y_j)-h_{fed}^*(v_{j,r}y_j)\big |=h_{fed}^*(v_{j,r})+h_{fed}(y_j)+2m-h_{fed}^*(v_{j,r}y_j)
\end{equation}
Thereby, $h_{fed}^*$ is an odd-edge felicitous-difference total coloring of the leaf-added graph $G_D$.

\textbf{Step D-5.} Return the odd-edge felicitous-difference total coloring $h_{fed}^*$ of the leaf-added graph $G_D$.

\vskip 0.4cm

\begin{example}\label{exa:8888888888}
For understanding the RLA-algorithm-D of the odd-edge felicitous-difference total coloring, Fig.\ref{fig:D-felicitous-difference} shows us the following facts:

(a) the graph $G_{\textrm{YR-twin}}$ admits a twin set-ordered odd-edge felicitous-difference total labeling $\theta_{fed}$ holding $h_{fed}(V(G_{\textrm{YR}}))\cup \theta_{fed}(V(G_{\textrm{YR-twin}}))\subseteq [0,20]$;

(b) the graph $G_{\textrm{YR}}$ admits a set-ordered odd-edge felicitous-difference total coloring $h_{fed}$ holding $\big |h_{fed}(x)+h_{fed}(y)-h_{fed}(xy)\big |=6$;

(c) the graph $G_{\textrm{YR-leaf}}$ admits a set-ordered odd-edge felicitous-difference total coloring $h_{fed}^*$ holding $\big |h_{fed}^*(x)+h_{fed}^*(y)-h_{fed}^*(xy)\big |=32$;

(d) the graph $H_{\textrm{YR-leaf}}$ admits a set-ordered odd-edge felicitous-difference total labeling $g_{fed}^*$ holding $\big |g_{fed}^*(x)+g_{fed}^*(y)-g_{fed}^*(xy)\big |=32$, as well as $G_{\textrm{YR-leaf}}\rightarrow H_{\textrm{YR-leaf}}$.\qqed
\end{example}

\begin{figure}[h]
\centering
\includegraphics[width=16.4cm]{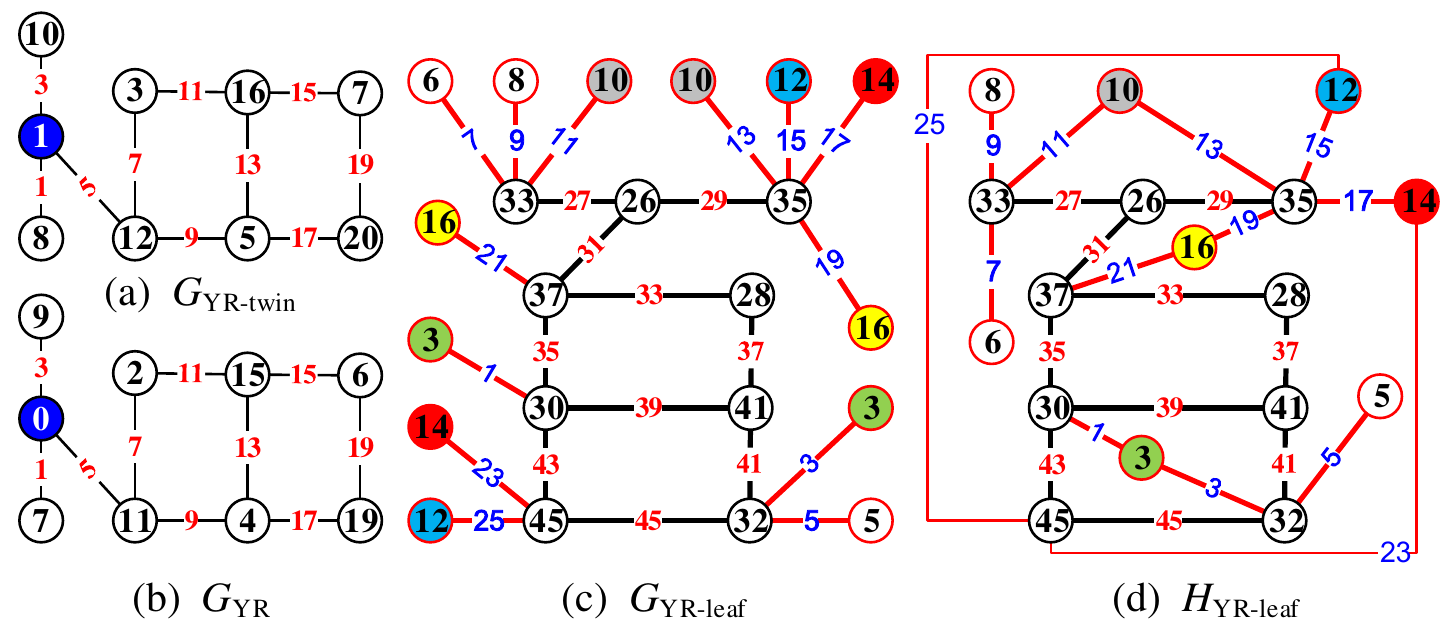}\\
\caption{\label{fig:D-felicitous-difference} {\small Examples for illustrating RLA-algorithm-D of the odd-edge felicitous-difference total coloring.}}
\end{figure}

\begin{problem}\label{problem:odd-edge-felicitous-difference}
In the RLA-algorithm-D of the odd-edge felicitous-difference total coloring, there are the following problems:

(i) \textbf{Estimate} the extremum number
\begin{equation}\label{eqa:555555}
\min\{\max h_{fed}^*(V(G_A)):~h_{fed}^*\textrm{ is an odd-edge felicitous-difference total coloring of}~G_D\}.
\end{equation} over all odd-edge felicitous-difference total colorings of the leaf-added graph $G_D$.

(ii) \textbf{Find} other ways for constructing graphs admitting odd-edge felicitous-difference total colorings/labelings. Two examples $G_{\textrm{X-add}}$ and $G_{\textrm{YR-add}}$ show in Fig.\ref{fig:E-add-vertex-edges} are obtained by adding directly vertices and edges to two graphs $G_{\textrm{X}}$ and $G_{\textrm{YR}}$, and they admit two set-ordered odd-edge felicitous-difference total labelings.

(iii) We have $m!$ permutations $w_{i_1},w_{i_2},\dots ,w_{i_m}$ from the added leaves of the leaf-added set $L^*(G_D)=\big (\bigcup^s_{i=1}L(x_i)\big )\bigcup \big (\bigcup^t_{j=1}L(y_j)\big )$ of the leaf-added $(p+m,q+m)$-graph $G_D$, so the leaf permutation in the RLA-algorithm-D of the odd-edge felicitous-difference total coloring is one of these $m!$ permutations. We define a new coloring $F_{fed}$ for the leaf-added graph $G_D$ as: Color leaf-edge $w_{i_j}z_{i_j}$ with $F_{fed}(w_{i_j}z_{i_j})=2j-1$ for $j\in [1,m]$, where $z_{i_j}\in X\cup Y=V(G)$, and each element $e\in V(G)\cup E(G)=[V(G_D)\cup E(G_D)]\setminus L^*(G_D)$ is colored as $F_{fed}(e)=h_{fed}^*(e)$, as well as color each added leaf $w_{i_j}\in L^*(G_D)$ by
$$F_{fed}(w_{i_j})=N_D^*+F_{fed}(w_{i_j}z_{i_j})-F_{fed}^*(z_{i_j}),~N_D^*=2m+N_D
$$ By Eq.(\ref{eqa:felicitous-difference-total-00}), Eq.(\ref{eqa:felicitous-difference-total-11}) and Eq.(\ref{eqa:felicitous-difference-total-22}), the coloring $F_{fed}$ is an odd-edge felicitous-difference total coloring based on a permutation $w_{i_1},w_{i_2},\dots ,w_{i_m}$.
\end{problem}

\begin{example}\label{exa:8888888888}
In Fig.\ref{fig:E-add-vertex-edges}, we can see:

(a) The graph $G_{\textrm{X}}$ admits a set-ordered odd-edge felicitous-difference total labeling $f_{fed}$ holding $\big |f_{fed}(x)+f_{fed}(y)-f_{fed}(xy)\big |=6$;

(b) the graph $G_{\textrm{X-add}}$ obtained from $G_{\textrm{X}}$ by adding vertices and edges admits a set-ordered odd-edge felicitous-difference total labeling $f_{fed}^*$ holding $\big |f_{fed}^*(x)+f_{fed}^*(y)-f_{fed}^*(xy)\big |=6$;

(c) the graph $G_{\textrm{YR}}$ admits a set-ordered odd-edge felicitous-difference total labeling $i_{fed}$ holding $\big |i_{fed}(x)+i_{fed}(y)-i_{fed}(xy)\big |=6$;

(d) the graph $G_{\textrm{YR-add}}$ admits a set-ordered odd-edge felicitous-difference total labeling $j_{fed}$ holding $\big |j_{fed}(x)+j_{fed}(y)-j_{fed}(xy)\big |=6$.\qqed
\end{example}

\begin{figure}[h]
\centering
\includegraphics[width=16.4cm]{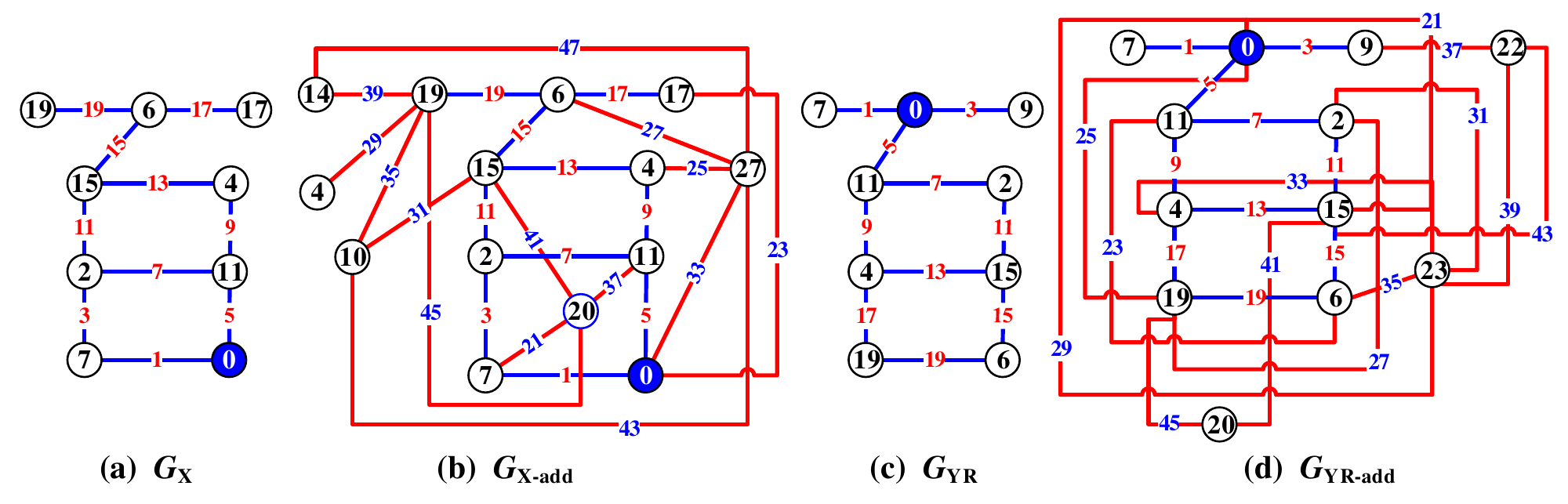}\\
\caption{\label{fig:E-add-vertex-edges} {\small The examples for adding no leaf to graphs admitting $W$-magic total colorings.}}
\end{figure}

\subsection{RLA-algorithm-E for adding leaves continuously}

\noindent \textsc{RLA-algorithm-E for adding leaves continuously.}

\textbf{Input:} A connected bipartite $(p,q)$-graph $G$ admitting an odd-edge graceful-difference total coloring $h^{gr}_0$.

\textbf{Output:} A connected bipartite $(p+m,q+m)$-graph $H$ admitting an odd-edge graceful-difference total coloring $h^{gr}_1$, where the $(p+m,q+m)$-graph $H$, called \emph{leaf-added graph}, is the result of adding randomly $m$ leaves to $G$.

\textbf{Initialization.} Let $G$ be a bipartite $(p,q)$-graph admitting an odd-edge graceful-difference total coloring $h^{gr}_0$, and let $V(G)=\{u_1,u_2,\dots, u_p\}$. By the definition of an odd-edge edge-magic total labeling, we have
\begin{equation}\label{eqa:555555}
0\leq h^{gr}_0(u_1)<h^{gr}_0(u_2)<\cdots <h^{gr}_0(u_p)\leq 2q-1
\end{equation}
so that each edge $u_iv_j\in E(G)$ satisfies the following equation
\begin{equation}\label{eqa:odd-edge-graceful-difference-colorings}
\big | |h^{gr}_0(u_i)-h^{gr}_0(v_j)|-h^{gr}_0(u_iv_j)\big |=\lambda_0
\end{equation} where integer $\lambda_0\geq 0$, as well as $h^{gr}_0(E(G))=\{h^{gr}_0(u_iv_j):u_iv_j\in E(G)\}=[1,2q-1]^o$.

\textbf{Step-E-1.} Adding new leaves $w_{i,1},w_{i,2},\dots ,w_{i,n_i}$ to each vertex $u_i$ of $G$ produces a leaf set $L(u_i)=\{w_{i,1},w_{i,2},\dots ,w_{i,n_i}\}$ with $i\in [1,p]$, here, it is allowed $n_j=0$ for some $j\in [1,p]$. The resultant graph is denoted as $H=G+E^*(G)$, where the leaf edge set
$$
E^*(G)=\{w_{i,j}u_i: ~w_{i,j}\in L(u_i),~ j\in [1,n_i],~i\in [1,p]\}
$$ and $m=|E^*(G)|$.

\textbf{Step-E-2.} Define a coloring $h^{gr}_1$ for the leaf-added graph $H$ as:

\textbf{Step-E-2.1.} \textbf{The ascending-order sub-algorithm.}

(1-1) Color $h^{gr}_1(w_{1,j}u_1)=2j-1$ with $j\in [1,n_1]$;

(1-2) Color $h^{gr}_1(w_{i,j}u_i)=2j-1+2\sum ^{i-1}_{k=1}n_k$ with $j\in [1,n_i]$ and $i\in [2,p]$.

(1-3) Color leaves $w_{s,t}\in \bigcup ^p_{i=1}L(u_i)$ with $h^{gr}_1(w_{s,t})$ holding
\begin{equation}\label{eqa:adding-leaves-continuously}
\big | |h^{gr}_1(u_s)-h^{gr}_1(w_{s,t})|-h^{gr}_1(w_{s,t}u_s)\big |=\lambda_0,~t\in [1,n_s],~s\in [1,p]
\end{equation}

(1-4) Color each element $z\in V(G)\cup E(G)$ with $h^{gr}_1(z)=h^{gr}_0(z)$.

Thereby, $h^{gr}_1(u_1)=0$, and
\begin{equation}\label{eqa:continus-adding-leaves}
\big | |h^{gr}_1(x)-h^{gr}_1(y)|-h^{gr}_1(xy)\big |=\lambda_0,~xy\in E(H),~h^{gr}_1(E(H))=\left [1,~|E(G)|+\sum ^{p}_{k=1}n_k\right ]^o
\end{equation}

\textbf{Step-E-2.2.} \textbf{The descending-order sub-algorithm.}

(2-1) Color $h^{gr}_1(w_{p,j}u_1)=2j-1$ with $j\in [1,n_p]$;

(2-2) Color $h^{gr}_1(w_{i,j}u_i)=2j-1+2\sum ^{p-i}_{k=1}n_{p-k+1}$ with $j\in [1,n_i]$ and $i\in [1,p-1]$.

(2-3) Color leaves $w_{s,t}\in \bigcup ^p_{i=1}L(u_i)$ holding Eq.(\ref{eqa:adding-leaves-continuously}) true.

(2-4) Color each element $z\in V(G)\cup E(G)$ with $h^{gr}_1(z)=h^{gr}_0(z)$.

We get $h^{gr}_1(u_1)=0$ and Eq.(\ref{eqa:continus-adding-leaves}).

\textbf{Step-E-2.3.} \textbf{The random-order sub-algorithm.}

(3-1) Color $h^{gr}_1(e_j)=2j-1$ with $j\in [1,A]$, where edges $e_1,e_2,\dots ,e_A$ is a permutation of leaf edges $w_{i,j}u_i$ with $j\in [1,n_i]$ and $i\in [1,p]$, and $A=\sum ^p_{i=1}|L(u_i)|$.

(3-2) Color leaves $w_{s,t}\in \bigcup ^p_{i=1}L(u_i)$ with $h^{gr}_1(w_{s,t})$ with $h^{gr}_1(w_{s,t})$ holding Eq.(\ref{eqa:adding-leaves-continuously}) true.

(3-3) Color each element $z\in V(G)\cup E(G)$ with $h^{gr}_1(z)=h^{gr}_0(z)$.

Thereby, $h^{gr}_1(u_1)=0$ and Eq.(\ref{eqa:continus-adding-leaves}) holds true.

\textbf{Step-E-3.} Return an odd-edge graceful-difference total coloring $h^{gr}_1$ of the leaf-added graph $H$.

\vskip 0.4cm

\begin{example}\label{exa:8888888888}
Fig.\ref{fig:add3times-graceful-difference-1} and Fig.\ref{fig:add3times-graceful-difference-2} show us examples for illustrating the RLA-algorithm for adding leaves continuously under the odd-edge graceful-difference total coloring:

In Fig.\ref{fig:add3times-graceful-difference-1}: (a) A graph $H_{\textrm{1-leaf}}$ admits an odd-edge graceful-difference total coloring $f^{gr}_1$ holding $\big | |f^{gr}_1(x)-f^{gr}_1(y)|-f^{gr}_1(xy)\big |=26$ for $xy\in E(H_{\textrm{1-leaf}})$; (b) a graph $AH_{\textrm{2-leaf}}$ obtained by adding leaves to the graph $H_{\textrm{1-leaf}}$; (c) a graph $H_{\textrm{2-leaf}}$ based on the graph $AH_{\textrm{1-leaf}}$ admits an odd-edge graceful-difference total coloring $f^{gr}_2$ holding $\big | |f^{gr}_2(x)-f^{gr}_2(y)|-f^{gr}_2(xy)\big |=26$ for $xy\in E(H_{\textrm{2-leaf}})$.

In Fig.\ref{fig:add3times-graceful-difference-2}: (a) A graph $AH_{\textrm{2-leaf}}$ is obtained by adding leaves to the graph $H_{\textrm{2-leaf}}$ shown in Fig.\ref{fig:add3times-graceful-difference-1} (c); (b) a graph $H_{\textrm{3-leaf}}$ based on the graph $AH_{\textrm{2-leaf}}$ admits an odd-edge graceful-difference total coloring $f^{gr}_3$ holding $\big | |f^{gr}_3(x)-f^{gr}_3(y)|-f^{gr}_3(xy)\big |=26$ for $xy\in E(H_{\textrm{3-leaf}})$.\qqed
\end{example}

\begin{figure}[h]
\centering
\includegraphics[width=16.4cm]{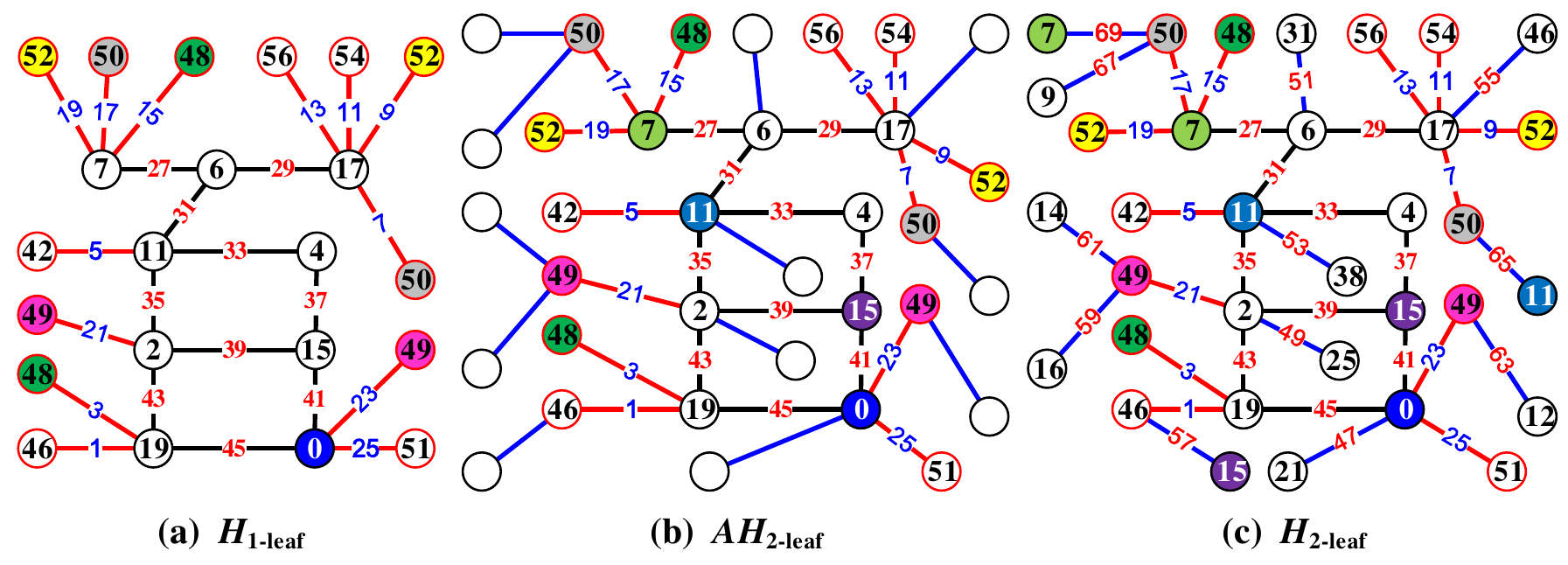}\\
\caption{\label{fig:add3times-graceful-difference-1} {\small Examples for the RLA-algorithm-E for adding leaves continuously under the odd-edge graceful-difference total coloring.}}
\end{figure}

\begin{figure}[h]
\centering
\includegraphics[width=16.4cm]{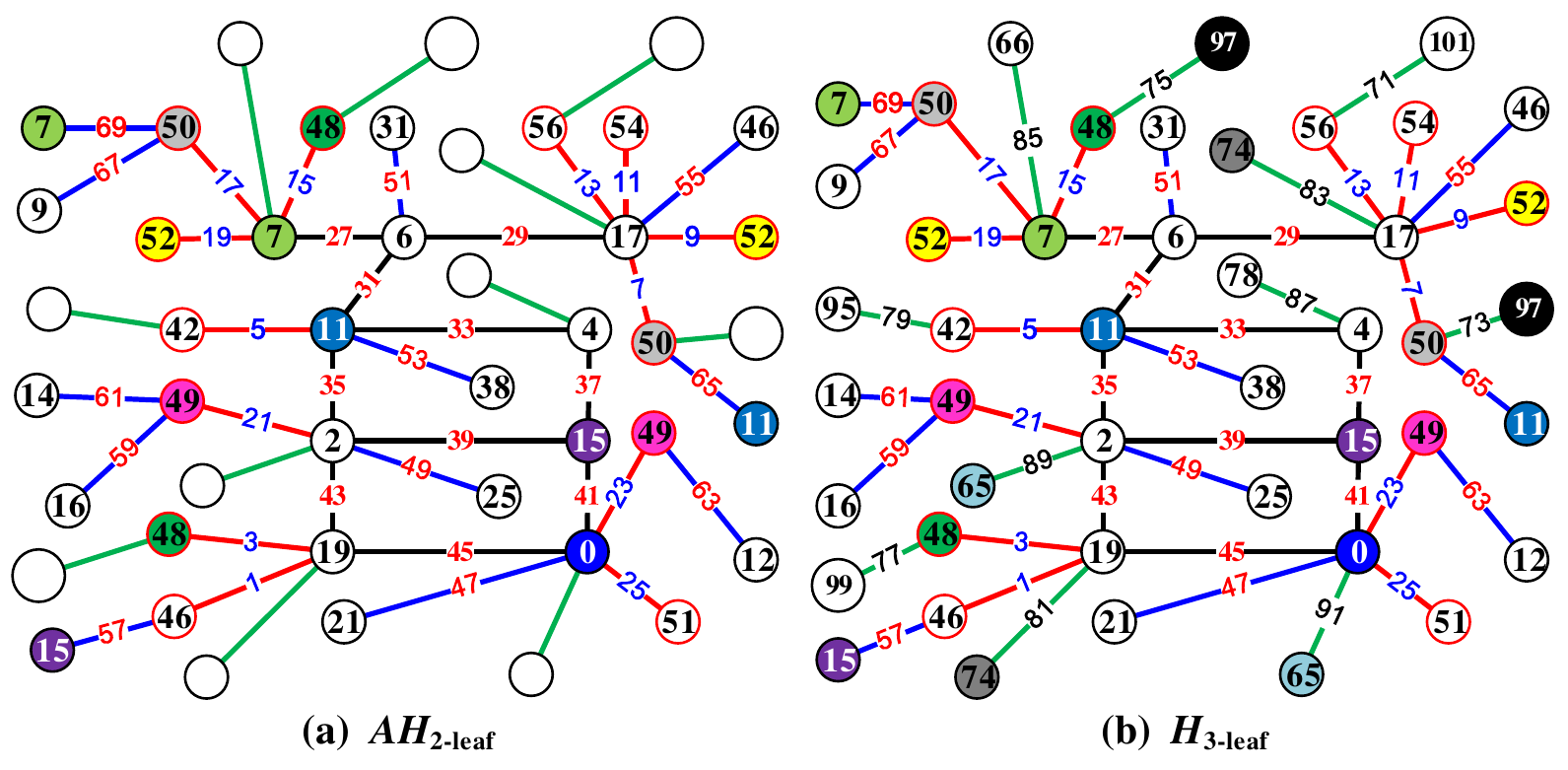}\\
\caption{\label{fig:add3times-graceful-difference-2} {\small An example for the RLA-algorithm-E for adding leaves continuously under the odd-edge graceful-difference total coloring.}}
\end{figure}

By the RLA-algorithm-E for adding leaves continuously, we have

\begin{thm}\label{thm:adding-leaves-continuously}
Suppose that a connected bipartite $(p,q)$-graph $G_0$ admits an odd-edge graceful-difference total coloring $h^{gr}_0$, then there are connected bipartite graph sequence $\{G_k\}^n_{k=0}$ such that each connected bipartite graph $G_k\in \{G_k\}^n_{k=0}$ is obtained by adding randomly $a_k~(\geq 1)$ leaves to $G_{k-1}$ and admits an odd-edge graceful-difference total coloring $h^{gr}_k$ with $k\geq 1$, and moreover $h^{gr}_i(V(G_i))\cap h^{gr}_j(V(G_j))\neq \emptyset$ for any pair integers $i,j\in [0,n]$.
\end{thm}

\begin{thm}\label{thm:tree-adding-leaves-continuously}
Each tree admits an odd-edge graceful-difference total coloring.
\end{thm}
\begin{proof} Let $T$ be a tree of $q~(\geq 1)$ edges. If $T$ is a star, that is, $T$ has its own vertex set $V(T)=\{u,x_i:i\in [1,q]\}$ and edge set $E(T)=\{ux_i:i\in [1,q]\}$. It is not hard to present an odd-edge graceful-difference total coloring (or labeling) for the star $T$, see six stars shown in Fig.\ref{fig:star-odd-edge-graceful-difference}.

So, the leaf-removed graph $T_1=T-L(T)$ is a tree still. If the leaf-removed graph $T_1$ is not a star, then we get the leaf-removed tree $T_2=T_1-L(T_1)$, go on in this way, we have leaf-removed trees $T_k=T_{k-1}-L(T_{k-1})$ with $k\leq \frac{D(T)}{2}$, where $D(T)$ is the diameter of $T$. Suppose that $T_m=T_{m-1}-L(T_{m-1})$ is a star, and $T_m$ admits an odd-edge graceful-difference total coloring (or labeling). Adding the leaves of $L(T_{k-1})$ to each leaf-removed tree $T_k$ for getting the tree $T_{k-1}$, and the tree $T_{k-1}$ admits an odd-edge graceful-difference total coloring (or labeling) $f^{gr}_{k-1}$, since $T_k$ admits an odd-edge graceful-difference total coloring (or labeling) $f^{gr}_{k}$ by the RLA-algorithm for adding leaves continuously, such that $f^{gr}_i(V(T_i))\cap f^{gr}_j(V(T_j))\neq \emptyset$.

The proof is complete.
\end{proof}

\begin{figure}[h]
\centering
\includegraphics[width=16.4cm]{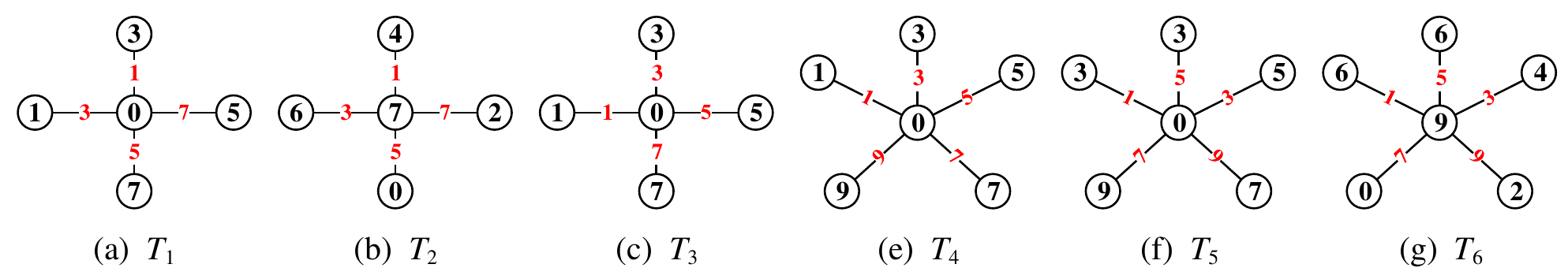}\\
\caption{\label{fig:star-odd-edge-graceful-difference} {\small Each tree $T_i$ admits an odd-edge graceful-difference total coloring (or labeling) $h^{gr}_i$ holding $\big | |h^{gr}_i(x)-h^{gr}_i(y)|-h^{gr}_i(xy)\big |=k_i$ for $xy\in E(T_i)$ and $i\in [1,6]$, where $k_1=k_2=k_5=k_6=2$, $k_3=k_4=0$.}}
\end{figure}

\subsection{RLA-algorithm-F for adding leaves continuously}

\noindent \textsc{RLA-algorithm-F for adding leaves continuously.}

\textbf{Input:} A connected bipartite $(p,q)$-graph $G$ admitting an odd-edge edge-magic total coloring $h^{ma}_0$.

\textbf{Output:} A connected bipartite $(p+m,q+m)$-graph $H$ admitting an odd-edge edge-magic total coloring $h^{ma}_1$, where the $(p+m,q+m)$-graph $H$, called \emph{leaf-added graph}, is the result of adding randomly $m$ leaves to $G$.

\textbf{Initialization.} Let $G$ be a bipartite $(p,q)$-graph admitting an odd-edge edge-magic total coloring $h^{ma}_0$, and let $V(G)=\{u_1,u_2,\dots, u_p\}$. By the definition of an odd-edge edge-magic total labeling, we have
\begin{equation}\label{eqa:edge-magic}
0\leq h^{ma}_0(u_1)<h^{ma}_0(u_2)<\cdots <h^{ma}_0(u_p)\leq 2q-1
\end{equation}
so that each edge $u_iv_j\in E(G)$ satisfies the following equation
\begin{equation}\label{eqa:odd-edge-edge-magic-colorings}
h^{ma}_0(u_i)+h^{ma}_0(v_j)+h^{ma}_0(u_iv_j)=\mu_0
\end{equation} where integer $\mu_0\geq 0$, as well as $h^{ma}_0(E(G))=\{h^{ma}_0(u_iv_j):u_iv_j\in E(G)\}=[1,2q-1]^o$.

\textbf{Step-F-1.} Adding new leaves $w_{i,1},w_{i,2},\dots ,w_{i,n_i}$ to each vertex $u_i$ of $G$ produces a leaf set $L(u_i)=\{w_{i,1},w_{i,2},\dots ,w_{i,n_i}\}$ with $i\in [1,p]$, here, it is allowed $n_j=0$ for some $j\in [1,p]$. The resultant graph is denoted as $H=G+E^*(G)$, where the leaf edge set
\begin{equation}\label{eqa:555555}
E^*(G)=\{w_{i,j}u_i: ~w_{i,j}\in L(u_i), ~j\in [1,n_i],~i\in [1,p]\}
\end{equation} and $m=|E^*(G)|$.

\textbf{Step-F-2.} Define a coloring $h^{ma}_1$ for the leaf-added graph $H$ in the following:

\textbf{Step-F-2.1.} \textbf{The ascending-order sub-algorithm.}

(1-1) Color $h^{ma}_1(w_{1,j}u_1)=2j-1$ with $j\in [1,n_1]$;

(1-2) Color $h^{ma}_1(w_{i,j}u_i)=2j-1+2\sum ^{i-1}_{k=1}n_k$ with $j\in [1,n_i]$ and $i\in [2,p]$.

(1-3) Color leaves $w_{s,t}\in \bigcup ^p_{i=1}L(u_i)$ with $h^{ma}_1(w_{s,t})$ holding
\begin{equation}\label{eqa:edge-magic-adding-leaves-continuously}
h^{ma}_1(u_s)+h^{ma}_1(w_{s,t})+h^{ma}_1(w_{s,t}u_s)=\mu^*_0,~t\in [1,n_s],~s\in [1,p]
\end{equation} where $\mu^*_0=\mu_0+2m$.

(1-4) Color each edge $uv\in E(G)$ with $h^{ma}_1(uv)=h^{ma}_0(uv)+2m$, and color each vertex $w\in V(G)$ with $h^{ma}_1(w)=h^{ma}_0(w)$.

Thereby, $h^{ma}_1(x)+h^{ma}_1(y)+h^{ma}_1(xy)=\mu^*_0$ for each edge $xy\in E(H)$, and $h^{ma}_1(u_1)=0$, as well as the edge color set
\begin{equation}\label{eqa:edge-magic-continus-adding-leaves}
h^{ma}_1(E(H))=\left [1,\quad |E(G)|+\sum ^{p}_{k=1}n_k\right ]^o
\end{equation}

\textbf{Step-F-2.2.} \textbf{The descending-order sub-algorithm.}

(2-1) Color $h^{ma}_1(w_{p,j}u_1)=2j-1$ with $j\in [1,n_p]$;

(2-2) Color $h^{ma}_1(w_{i,j}u_i)=2j-1+2\sum ^{p-i}_{k=1}n_{p-k+1}$ with $j\in [1,n_i]$ and $i\in [1,p-1]$.

(2-3) Color leaves $w_{s,t}\in \bigcup ^p_{i=1}L(u_i)$ holding Eq.(\ref{eqa:edge-magic-adding-leaves-continuously}) true.

(2-4) Color each edge $uv\in E(G)$ with $h^{ma}_1(uv)=h^{ma}_0(uv)+2m$, and color each vertex $w\in V(G)$ with $h^{ma}_1(w)=h^{ma}_0(w)$.

We get $h^{ma}_1(u_1)=0$ and Eq.(\ref{eqa:edge-magic-continus-adding-leaves}).

\textbf{Step-F-2.3.} \textbf{The random-order sub-algorithm.}

(3-1) Color $h^{ma}_1(e_j)=2j-1$ with $j\in [1,A]$, where edges $e_1,e_2,\dots ,e_A$ is a permutation of leaf edges $w_{i,j}u_i$ with $j\in [1,n_i]$ and $i\in [1,p]$, and $A=\sum ^p_{i=1}|L(u_i)|$.

(3-2) Color leaves $w_{s,t}\in \bigcup ^p_{i=1}L(u_i)$ with $h^{ma}_1(w_{s,t})$ with $h^{ma}_1(w_{s,t})$ holding Eq.(\ref{eqa:edge-magic-adding-leaves-continuously}) true.

(3-3) Color each edge $uv\in E(G)$ with $h^{ma}_1(uv)=h^{ma}_0(uv)+2m$, and color each vertex $w\in V(G)$ with $h^{ma}_1(w)=h^{ma}_0(w)$.

Thereby, $h^{ma}_1(u_1)=0$ and Eq.(\ref{eqa:edge-magic-continus-adding-leaves}) holds true.

\textbf{Step-F-3.} Return an odd-edge edge-magic total coloring $h^{ma}_1$ of the leaf-added graph $H$.

\vskip 0.4cm
\begin{example}\label{exa:8888888888}
Fig.\ref{fig:add3times-edge-magic-1} and Fig.\ref{fig:add3times-edge-magic-2} show us examples for illustrating the RLA-algorithm-F for adding leaves continuously under the odd-edge edge-magic total coloring:

In Fig.\ref{fig:add3times-edge-magic-1}: (a) A graph $T_{\textrm{1-leaf}}$ admits an odd-edge edge-magic total coloring $h^{ma}_1$ holding $h^{ma}_1(x)+h^{ma}_1(y)+h^{ma}_1(xy)=52$ for $xy\in E(T_{\textrm{1-leaf}})$; (b) a graph $AT_{\textrm{1-leaf}}$ obtained by adding leaves to the graph $T_{\textrm{1-leaf}}$; (c) a graph $T_{\textrm{2-leaf}}$ based on the graph $AT_{\textrm{1-leaf}}$ admits an odd-edge edge-magic total coloring $h^{ma}_2$ holding $h^{ma}_2(x)+h^{ma}_2(y)+h^{ma}_2(xy)=78$ for $xy\in E(T_{\textrm{2-leaf}})$.

In Fig.\ref{fig:add3times-edge-magic-2}: (a) A graph $AT_{\textrm{2-leaf}}$ obtained by adding leaves to the graph $T_{\textrm{2-leaf}}$ shown in Fig.\ref{fig:add3times-edge-magic-1} (c); (b) a graph $T_{\textrm{3-leaf}}$ based on the graph $AT_{\textrm{2-leaf}}$ admits an odd-edge edge-magic total coloring $h^{ma}_3$ holding $h^{ma}_3(x)+h^{ma}_3(y)+h^{ma}_3(xy)=102$ for $xy\in E(T_{\textrm{3-leaf}})$.\qqed
\end{example}

\begin{figure}[h]
\centering
\includegraphics[width=16.4cm]{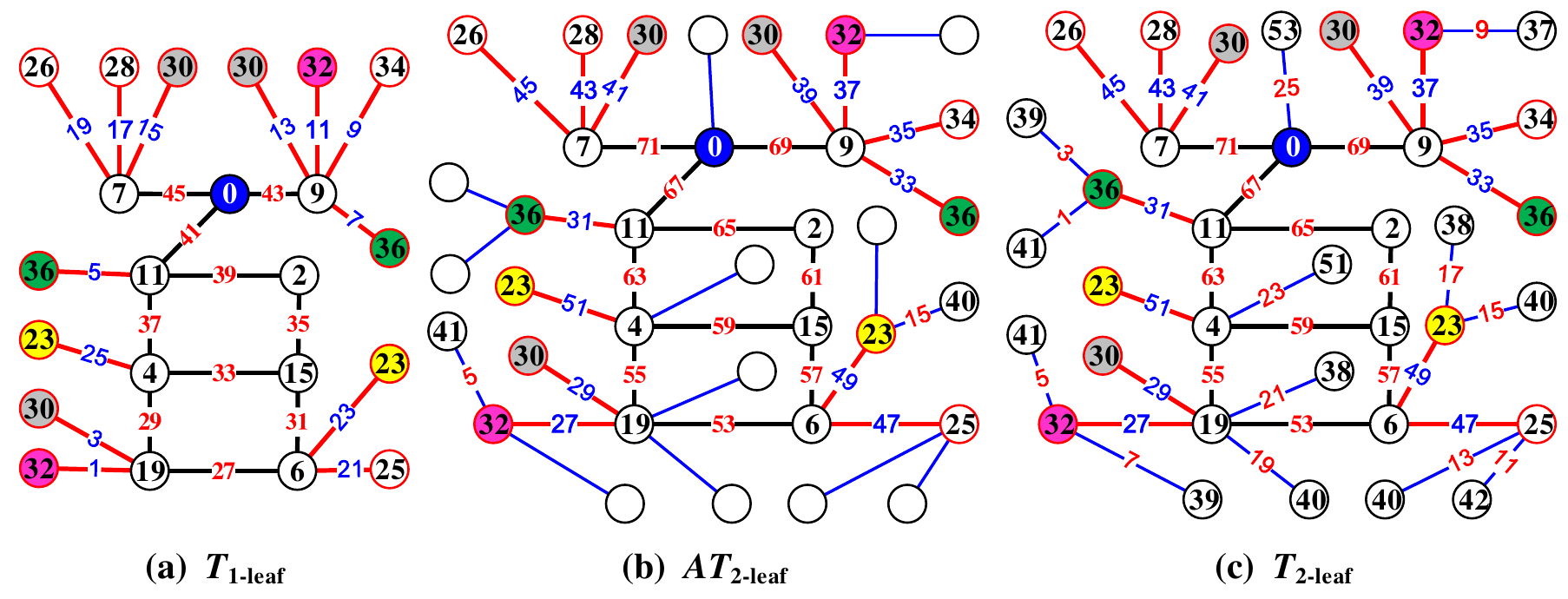}\\
\caption{\label{fig:add3times-edge-magic-1} {\small Examples for the RLA-algorithm-F for adding leaves continuously under the odd-edge edge-magic total coloring.}}
\end{figure}

\begin{figure}[h]
\centering
\includegraphics[width=16.4cm]{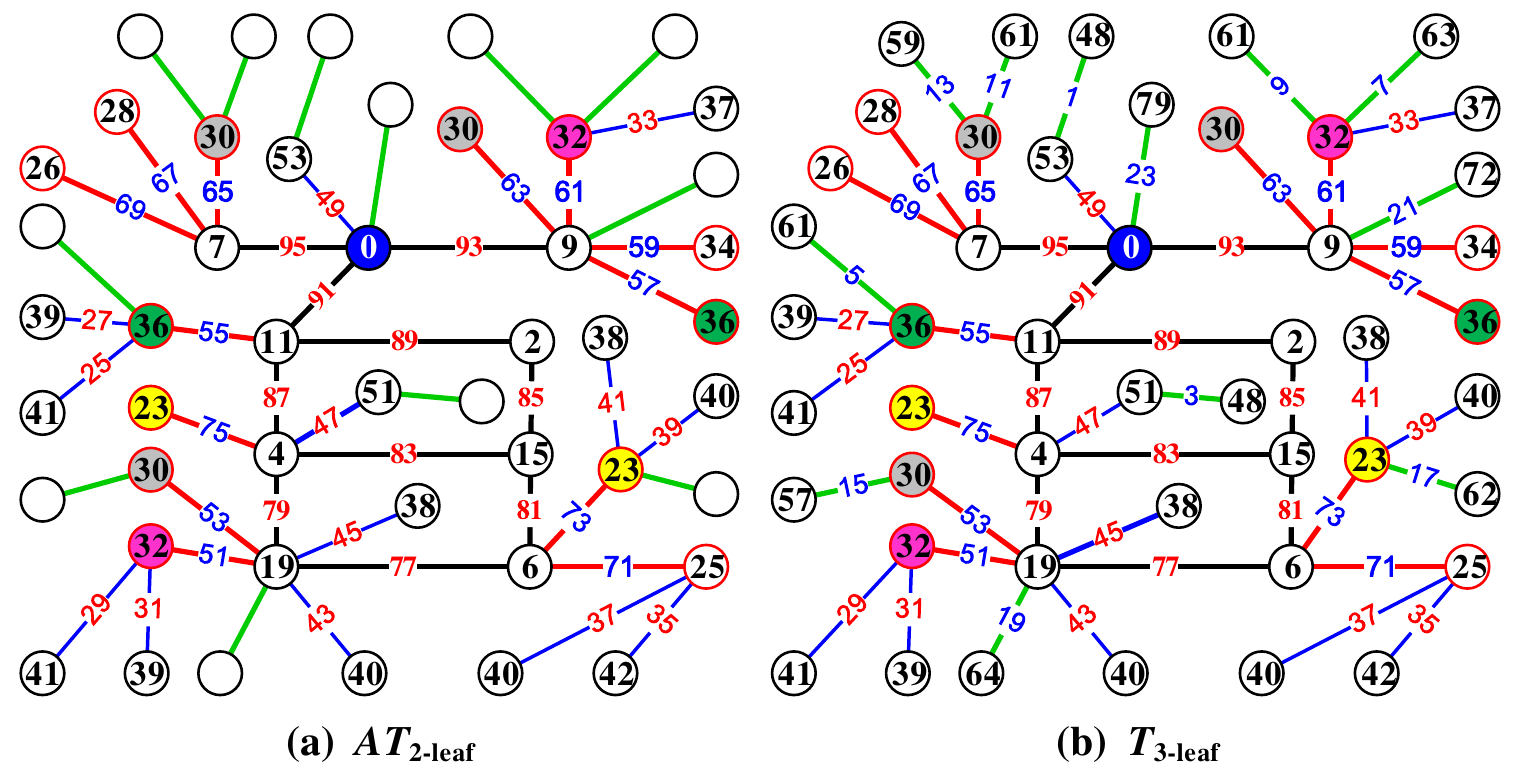}\\
\caption{\label{fig:add3times-edge-magic-2} {\small An example for the RLA-algorithm-F for adding leaves continuously under the odd-edge edge-magic total coloring.}}
\end{figure}

\begin{thm}\label{thm:edge-magic-adding-leaves-continuously}
Suppose that a connected bipartite $(p,q)$-graph $G_0$ admits an odd-edge edge-magic total coloring $h^{gr}_0$, then there are connected bipartite graph sequence $\{G_k\}^n_{k=0}$ such that each connected bipartite graph $G_k\in \{G_k\}^n_{k=0}$ is obtained by adding randomly $a_k~(\geq 1)$ leaves to $G_{k-1}$ and admits an odd-edge edge-magic total coloring $h^{gr}_k$ with $k\geq 1$, and moreover $h^{gr}_i(V(G_i))\cap h^{gr}_j(V(G_j))\neq \emptyset$ for any pair integers $i,j\in [0,n]$.
\end{thm}

\begin{thm}\label{thm:edge-magic-tree-adding-leaves-continuously}
Each tree admits an odd-edge edge-magic total coloring.
\end{thm}

\begin{figure}[h]
\centering
\includegraphics[width=16.4cm]{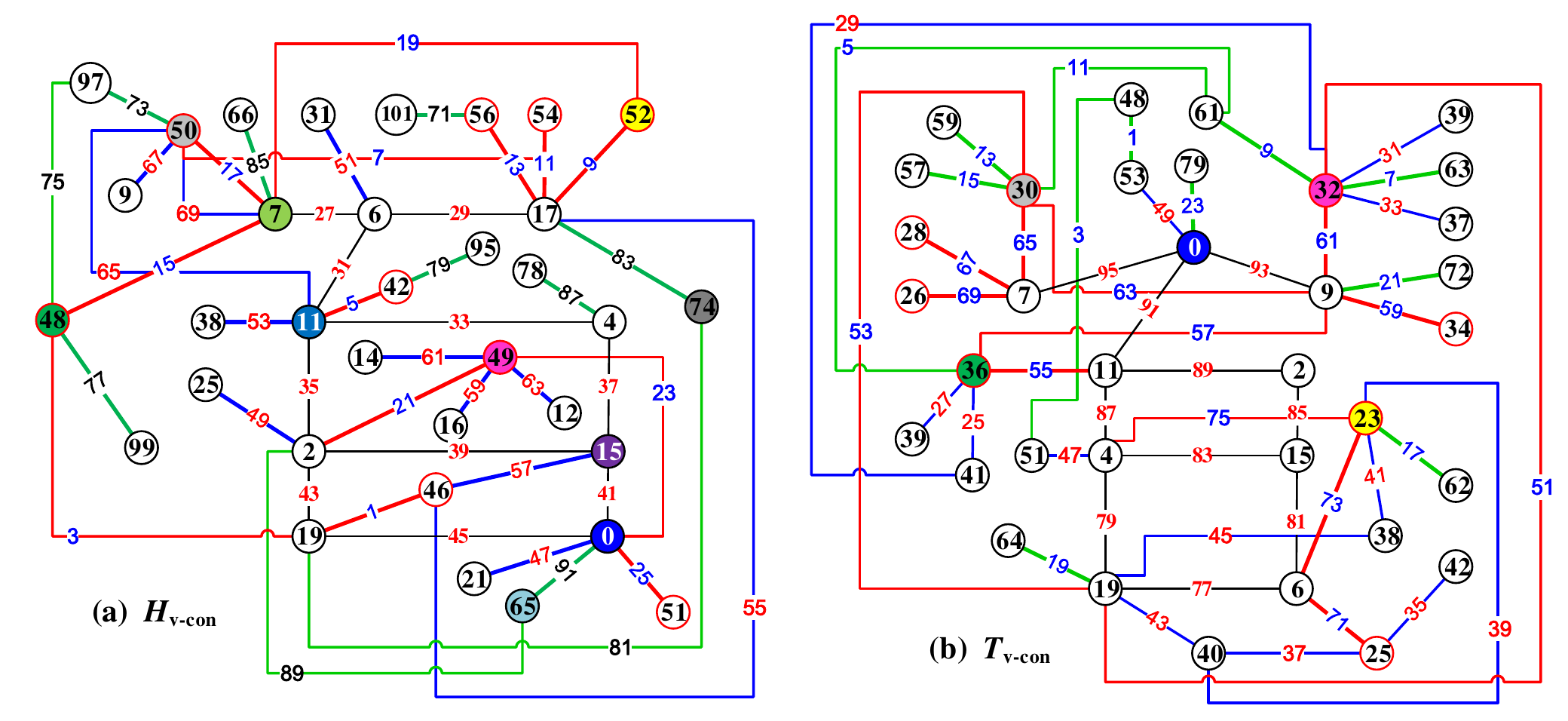}\\
\caption{\label{fig:vertex-coincide} {\small (a) The graph $H_{\textrm{v-con}}$ is obtained by vertex-coinciding vertices colored with the same colors in the graph $H_{\textrm{3-leaf}}$ shown in Fig.\ref{fig:add3times-graceful-difference-2}(c); (b) the graph $T_{\textrm{v-con}}$ is obtained by vertex-coinciding vertices colored with the same colors in the graph $T_{\textrm{3-leaf}}$ shown in Fig.\ref{fig:add3times-edge-magic-2}(c).}}
\end{figure}

\subsection{RLA-algorithm-G for adding leaves continuously}

\noindent \textsc{RLA-algorithm-G for adding leaves continuously.}

\textbf{Input:} A connected bipartite $(p,q)$-graph $G$ admitting an odd-edge edge-difference total coloring $h^{ed}_0$.

\textbf{Output:} A connected bipartite $(p+m,q+m)$-graph $H$ admitting an odd-edge edge-difference total coloring $h^{ed}_1$, where the $(p+m,q+m)$-graph $H$, called \emph{leaf-added graph}, is the result of adding randomly $m$ leaves to $G$.

\textbf{Initialization.} Let $G$ be a bipartite $(p,q)$-graph admitting an odd-edge edge-difference total coloring $h^{ed}_0$, and let $V(G)=\{u_1,u_2,\dots, u_p\}$. By the definition of an odd-edge edge-difference total labeling, we have
\begin{equation}\label{eqa:edge-difference}
0\leq h^{ed}_0(u_1)<h^{ed}_0(u_2)<\cdots <h^{ed}_0(u_p)\leq 2q-1
\end{equation}
so that each edge $u_iv_j\in E(G)$ satisfies the following equation
\begin{equation}\label{eqa:odd-edge-edge-difference-colorings}
h^{ed}_0(u_iv_j)+|h^{ed}_0(u_i)-h^{ed}_0(v_j)|=\rho_0
\end{equation} where integer $\rho_0\geq 0$, as well as $h^{ed}_0(E(G))=\{h^{ed}_0(u_iv_j):u_iv_j\in E(G)\}=[1,2q-1]^o$.

\textbf{Step-G-1.} Adding new leaves $w_{i,1},w_{i,2},\dots ,w_{i,n_i}$ to each vertex $u_i$ of $G$ produces a leaf set $L(u_i)=\{w_{i,1},w_{i,2},\dots ,w_{i,n_i}\}$ with $i\in [1,p]$, here, it is allowed $n_j=0$ for some $j\in [1,p]$. The resultant graph is denoted as $H=G+E^*(G)$, where the leaf edge set
\begin{equation}\label{eqa:555555}
E^*(G)=\{w_{i,j}u_i: ~w_{i,j}\in L(u_i), ~j\in [1,n_i],~i\in [1,p]\}
\end{equation} having $m=|E^*(G)|$ edges.

\textbf{Step-G-2.} Define a coloring $h^{ed}_1$ for the leaf-added graph $H$ in the following steps:

\textbf{Step-G-2.1.} \textbf{The ascending-order sub-algorithm.}

(1-1) Color $h^{ed}_1(w_{1,j}u_1)=2j-1$ with $j\in [1,n_1]$;

(1-2) Color $h^{ed}_1(w_{i,j}u_i)=2j-1+2\sum ^{i-1}_{k=1}n_k$ with $j\in [1,n_i]$ and $i\in [2,p]$.

(1-3) Color leaves $w_{s,t}\in \bigcup ^p_{i=1}L(u_i)$ with $h^{ed}_1(w_{s,t})$ holding
\begin{equation}\label{eqa:edge-difference-adding-leaves-continuously}
h^{ed}_1(w_{s,t}u_s)+|h^{ed}_1(u_s)-h^{ed}_1(w_{s,t})|=\rho^*_0,~t\in [1,n_s],~s\in [1,p]
\end{equation} with $\rho^*_0=\rho_0+2m$.

(1-4) Color each edge $uv\in E(G)$ with $h^{ed}_1(uv)=h^{ed}_0(uv)+2m$, and color each vertex $w\in V(G)$ with $h^{ed}_1(w)=h^{ed}_0(w)$.

Thereby, $h^{ed}_1(xy)+|h^{ed}_1(x)-h^{ed}_1(y)|=\rho^*_0$ for each edge $xy\in E(H)$, $h^{ed}_1(u_1)=0$, and
\begin{equation}\label{eqa:edge-difference-continus-adding-leaves}
h^{ed}_1(E(H))=\left [1,\quad |E(G)|+\sum ^{p}_{k=1}n_k\right ]^o
\end{equation}

\textbf{Step-G-2.2.} \textbf{The descending-order sub-algorithm.}

(2-1) Color $h^{ed}_1(w_{p,j}u_1)=2j-1$ with $j\in [1,n_p]$;

(2-2) Color $h^{ed}_1(w_{i,j}u_i)=2j-1+2\sum ^{p-i}_{k=1}n_{p-k+1}$ with $j\in [1,n_i]$ and $i\in [1,p-1]$.

(2-3) Color leaves $w_{s,t}\in \bigcup ^p_{i=1}L(u_i)$ holding Eq.(\ref{eqa:edge-difference-adding-leaves-continuously}) true.

(2-4) Color each edge $uv\in E(G)$ with $h^{ed}_1(uv)=h^{ed}_0(uv)+2m$, and color each vertex $w\in V(G)$ with $h^{ed}_1(w)=h^{ed}_0(w)$.

We get $h^{ed}_1(u_1)=0$ and Eq.(\ref{eqa:edge-difference-continus-adding-leaves}).

\textbf{Step-G-2.3.} \textbf{The random-order sub-algorithm.}

(3-1) Color $h^{ed}_1(e_j)=2j-1$ with $j\in [1,A]$, where edges $e_1,e_2,\dots ,e_A$ is a permutation of leaf edges $w_{i,j}u_i$ with $j\in [1,n_i]$ and $i\in [1,p]$, and $A=\sum ^p_{i=1}|L(u_i)|$.

(3-2) Color leaves $w_{s,t}\in \bigcup ^p_{i=1}L(u_i)$ with $h^{ed}_1(w_{s,t})$ with $h^{ed}_1(w_{s,t})$ holding Eq.(\ref{eqa:edge-difference-adding-leaves-continuously}) true.

(3-3) Color each edge $uv\in E(G)$ with $h^{ed}_1(uv)=h^{ed}_0(uv)+2m$, and color each vertex $w\in V(G)$ with $h^{ed}_1(w)=h^{ed}_0(w)$.

Thereby, $h^{ed}_1(u_1)=0$ and Eq.(\ref{eqa:edge-difference-continus-adding-leaves}) holds true.

\textbf{Step-G-3.} Return an odd-edge edge-difference total coloring $h^{ed}_1$ of the leaf-added graph $H$.

\vskip 0.4cm

\begin{example}\label{exa:8888888888}
Fig.\ref{fig:add3times-edge-difference-1} and Fig.\ref{fig:add3times-edge-difference-2} show us examples for illustrating the RLA-algorithm-G for adding leaves continuously under odd-edge edge-difference total coloring:

In Fig.\ref{fig:add3times-edge-difference-1}: (a) a graph $G_{\textrm{1-leaf}}$ admits an odd-edge edge-difference total coloring $h^{ed}_1$ holding $h^{ed}_1(xy)+|h^{ed}_1(x)-h^{ed}_1(y)|=46$ for $xy\in E(G_{\textrm{1-leaf}})$; (b) a graph $AG_{\textrm{1-leaf}}$ is obtained by adding leaves to the graph $G_{\textrm{1-leaf}}$; (c) a graph $G_{\textrm{2-leaf}}$ based on the graph $AG_{\textrm{1-leaf}}$ admits an odd-edge edge-difference total coloring $h^{ed}_2$ holding $h^{ed}_2(xy)+|h^{ed}_2(x)-h^{ed}_2(y)|=70$ for $xy\in E(G_{\textrm{2-leaf}})$.

In Fig.\ref{fig:add3times-edge-difference-2}, (a) a graph $AG_{\textrm{2-leaf}}$ is obtained by adding leaves to the graph $G_{\textrm{2-leaf}}$ shown in Fig.\ref{fig:add3times-edge-difference-1} (c); (b) a graph $G_{\textrm{3-leaf}}$ based on the graph $AG_{\textrm{2-leaf}}$ admits an odd-edge edge-difference total coloring $h^{ed}_3$ holding $h^{ed}_3(xy)+|h^{ed}_3(x)-h^{ed}_3(y)|=92$ for $xy\in E(G_{\textrm{3-leaf}})$. \qqed
\end{example}

\begin{figure}[h]
\centering
\includegraphics[width=16.4cm]{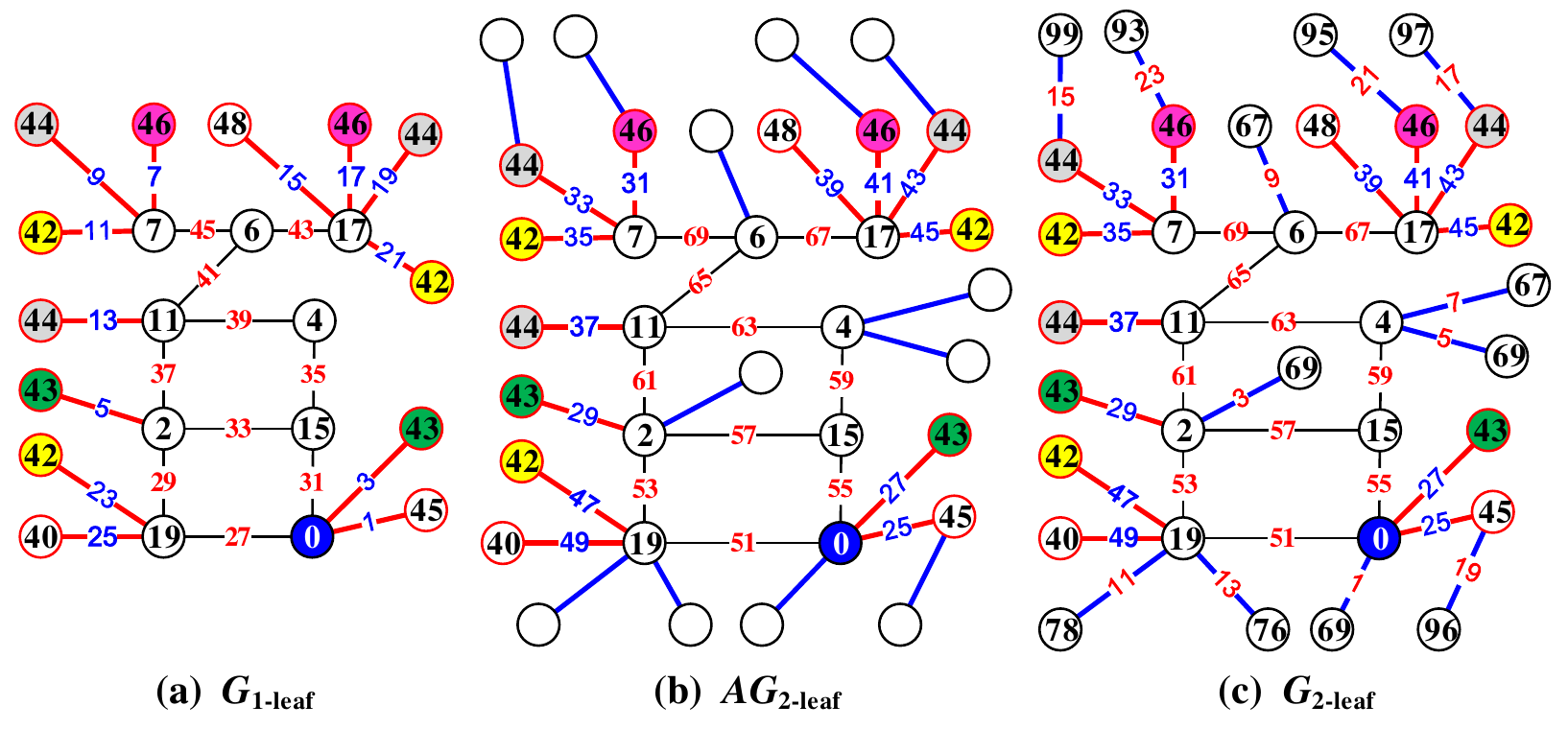}\\
\caption{\label{fig:add3times-edge-difference-1} {\small Examples for the RLA-algorithm-G for adding leaves continuously under odd-edge edge-difference total coloring.}}
\end{figure}

\begin{figure}[h]
\centering
\includegraphics[width=16.4cm]{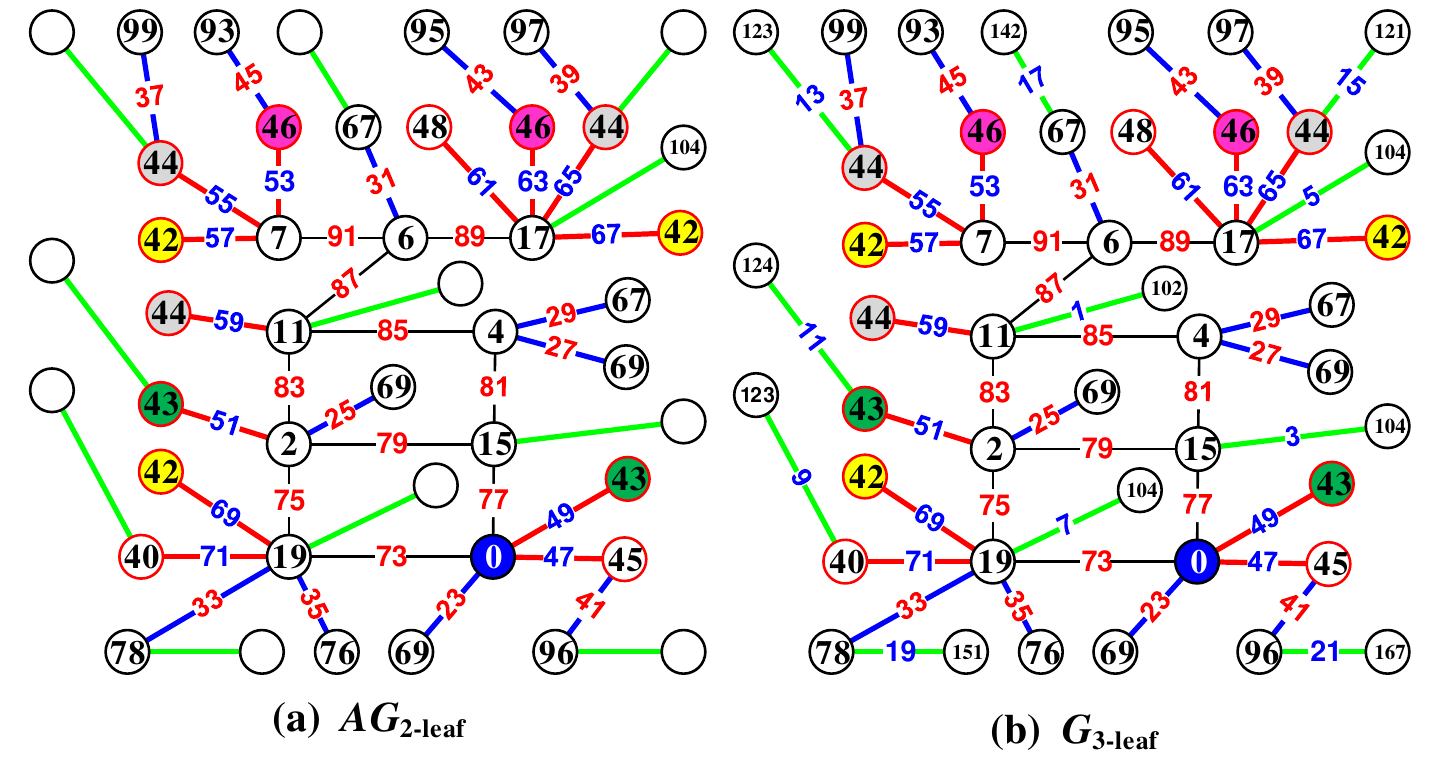}\\
\caption{\label{fig:add3times-edge-difference-2} {\small An example for the RLA-algorithm-G for adding leaves continuously under odd-edge edge-difference total coloring.}}
\end{figure}

\begin{thm}\label{thm:edge-difference-adding-leaves-continuously}
Suppose that a connected bipartite $(p,q)$-graph $G_0$ admits an odd-edge edge-difference total coloring $h^{gr}_0$, then there are connected bipartite graph sequence $\{G_k\}^n_{k=0}$ such that each connected bipartite graph $G_k\in \{G_k\}^n_{k=0}$ is obtained by adding randomly $a_k~(\geq 1)$ leaves to $G_{k-1}$ and admits an odd-edge edge-difference total coloring $h^{gr}_k$ with $k\geq 1$, and moreover $h^{gr}_i(V(G_i))\cap h^{gr}_j(V(G_j))\neq \emptyset$ for any pair integers $i,j\in [0,n]$.
\end{thm}

\begin{thm}\label{thm:edge-difference-tree-adding-leaves-continuously}
Each tree admits an odd-edge edge-difference total coloring.
\end{thm}

\subsection{RLA-algorithm-H for adding leaves continuously}

\noindent \textsc{RLA-algorithm-H for adding leaves continuously.}

\textbf{Input:} A connected bipartite $(p,q)$-graph $G$ admitting an odd-edge felicitous-difference total coloring $h^{fe}_0$.

\textbf{Output:} A connected bipartite $(p+m,q+m)$-graph $H$ admitting an odd-edge felicitous-difference total coloring $h^{fe}_1$, where the $(p+m,q+m)$-graph $H$, called \emph{leaf-added graph}, is the result of adding randomly $m$ leaves to $G$.

\textbf{Initialization.} Let $G$ be a bipartite $(p,q)$-graph admitting an odd-edge felicitous-difference total coloring $h^{fe}_0$, and let $V(G)=\{u_1,u_2,\dots, u_p\}$. By the definition of an odd-edge felicitous-difference total labeling, we have
\begin{equation}\label{eqa:felicitous-difference}
0\leq h^{fe}_0(u_1)<h^{fe}_0(u_2)<\cdots <h^{fe}_0(u_p)\leq 2q-1
\end{equation}
so that each edge $u_iv_j\in E(G)$ satisfies the following equation
\begin{equation}\label{eqa:odd-edge-felicitous-difference-colorings}
\big |h^{fe}_0(u_i)+h^{fe}_0(v_j)-h^{fe}_0(u_iv_j)\big |=\xi_0
\end{equation} where integer $\xi_0\geq 0$, as well as $h^{fe}_0(E(G))=\{h^{fe}_0(u_iv_j):u_iv_j\in E(G)\}=[1,2q-1]^o$.

\textbf{Step-H-1.} Adding new leaves $w_{i,1},w_{i,2},\dots ,w_{i,n_i}$ to each vertex $u_i$ of $G$ produces a leaf set $L(u_i)=\{w_{i,1},w_{i,2},\dots ,w_{i,n_i}\}$ with $i\in [1,p]$, here, it is allowed $n_j=0$ for some $j\in [1,p]$. The resultant graph is denoted as $H=G+E^*(G)$, where the leaf edge set
\begin{equation}\label{eqa:555555}
E^*(G)=\{w_{i,j}u_i: ~w_{i,j}\in L(u_i), ~j\in [1,n_i],~i\in [1,p]\}
\end{equation} having $m=|E^*(G)|$ edges.

\textbf{Step-H-2.} Define a coloring $h^{fe}_1$ for the leaf-added graph $H$ in the following steps:

\textbf{Step-H-2.1.} \textbf{The ascending-order sub-algorithm.}

(1-1) Color $h^{fe}_1(w_{1,j}u_1)=(2q-1)+2j$ with $j\in [1,n_1]$;

(1-2) Color $h^{fe}_1(w_{i,j}u_i)=(2q-1)+2j+2\sum ^{i-1}_{k=1}n_k$ with $j\in [1,n_i]$ and $i\in [2,p]$.

(1-3) Color leaves $w_{s,t}\in \bigcup ^p_{i=1}L(u_i)$ with $h^{fe}_1(w_{s,t})$ holding
\begin{equation}\label{eqa:felicitous-difference-adding-leaves-continuously}
\big |h^{fe}_1(u_s)+h^{fe}_1(w_{s,t})-h^{fe}_1(w_{s,t}u_s)\big |=\xi_0,~t\in [1,n_s],~s\in [1,p]
\end{equation}

(1-4) Color each edge $w\in V(G)\cup E(G)$ with $h^{fe}_1(w)=h^{fe}_0(w)$.

Thereby, we get $h^{fe}_1(u_1)=0$ and
\begin{equation}\label{eqa:felicitous-difference-continus-adding-leaves}
\big |h^{fe}_1(x)+h^{fe}_1(y)-h^{fe}_1(xy)\big |=\xi_0,~xy\in E(H);~h^{fe}_1(E(H))=\left [1,|E(G)|+\sum ^{p}_{k=1}n_k\right ]^o
\end{equation}

\textbf{Step-H-2.2.} \textbf{The descending-order sub-algorithm.}

(2-1) Color $h^{fe}_1(w_{p,j}u_1)=(2q-1)+2j$ with $j\in [1,n_p]$;

(2-2) Color $h^{fe}_1(w_{i,j}u_i)=(2q-1)+2j+2\sum ^{p-i}_{k=1}n_{p-k+1}$ with $j\in [1,n_i]$ and $i\in [1,p-1]$.

(2-3) Color leaves $w_{s,t}\in \bigcup ^p_{i=1}L(u_i)$ holding Eq.(\ref{eqa:felicitous-difference-adding-leaves-continuously}) true.

(2-4) Color each edge $w\in V(G)\cup E(G)$ with $h^{fe}_1(w)=h^{fe}_0(w)$.

We get $h^{fe}_1(u_1)=0$ and Eq.(\ref{eqa:felicitous-difference-continus-adding-leaves}).

\textbf{Step-H-2.3.} \textbf{The random-order sub-algorithm.}

(3-1) Color $h^{fe}_1(e_j)=(2q-1)+2j$ with $j\in [1,A]$, where edges $e_1,e_2,\dots ,e_A$ is a permutation of leaf edges $w_{i,j}u_i$ with $j\in [1,n_i]$ and $i\in [1,p]$, and $A=\sum ^p_{i=1}|L(u_i)|$.

(3-2) Color leaves $w_{s,t}\in \bigcup ^p_{i=1}L(u_i)$ with $h^{fe}_1(w_{s,t})$ with $h^{fe}_1(w_{s,t})$ holding Eq.(\ref{eqa:felicitous-difference-adding-leaves-continuously}) true.

(3-3) Color each edge $w\in V(G)\cup E(G)$ with $h^{fe}_1(w)=h^{fe}_0(w)$.

Thereby, $h^{fe}_1(u_1)=0$ and Eq.(\ref{eqa:felicitous-difference-continus-adding-leaves}) holds true.

\textbf{Step-H-3.} Return an odd-edge felicitous-difference total coloring $h^{fe}_1$ of the leaf-added graph $H$.

\vskip 0.4cm
\begin{example}\label{exa:8888888888}
Fig.\ref{fig:add3times-felicitous-difference-1} and Fig.\ref{fig:add3times-felicitous-difference-2} show us examples for illustrating the RLA-algorithm-H for adding leaves continuously under the odd-edge felicitous-difference total coloring:

In Fig.\ref{fig:add3times-felicitous-difference-1}: (a) A graph $J_{\textrm{1-leaf}}$ admits an odd-edge felicitous-difference total coloring $h^{fe}_1$ holding $\big |h^{fe}_1(x)+h^{fe}_1(y)-h^{fe}_1(xy)\big |=20$ for $xy\in E(J_{\textrm{1-leaf}})$; (b) a graph $AJ_{\textrm{1-leaf}}$ obtained by adding leaves to the graph $J_{\textrm{1-leaf}}$; (c) a graph $J_{\textrm{2-leaf}}$ based on the graph $AJ_{\textrm{1-leaf}}$ admits an odd-edge felicitous-difference total coloring $h^{fe}_2$ holding $\big |h^{fe}_2(x)+h^{fe}_2(y)-h^{fe}_2(xy)\big |=20$ for $xy\in E(J_{\textrm{2-leaf}})$.

In Fig.\ref{fig:add3times-felicitous-difference-1}: (a) a graph $AJ_{\textrm{2-leaf}}$ obtained by adding leaves to the graph $J_{\textrm{2-leaf}}$ shown in Fig.\ref{fig:add3times-felicitous-difference-1} (c); (b) a graph $J_{\textrm{3-leaf}}$ based on the graph $AJ_{\textrm{2-leaf}}$ admits an odd-edge felicitous-difference total coloring $h^{fe}_3$ holding $\big |h^{fe}_3(x)+h^{fe}_3(y)-h^{fe}_3(xy)\big |=20$ for $xy\in E(J_{\textrm{3-leaf}})$.\qqed
\end{example}

\begin{figure}[h]
\centering
\includegraphics[width=16.4cm]{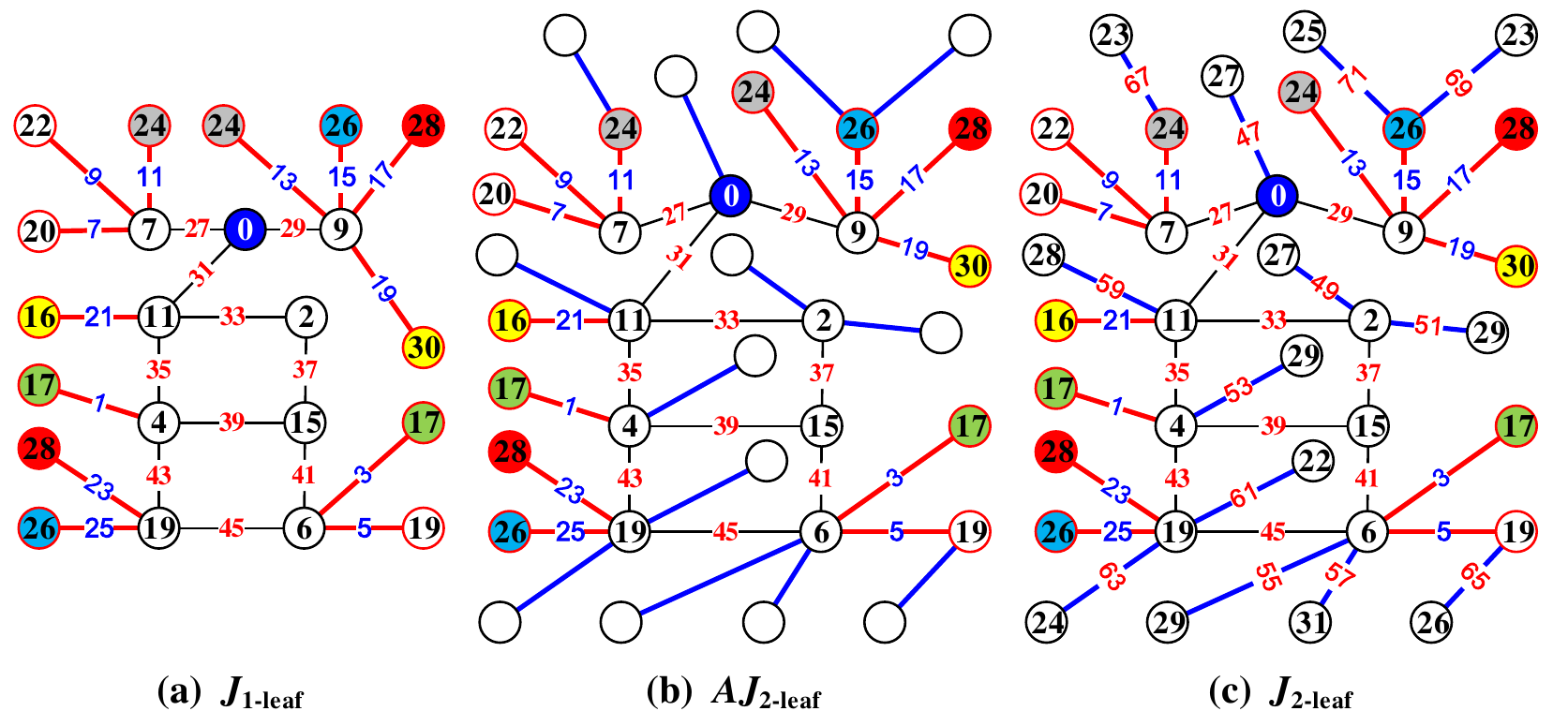}\\
\caption{\label{fig:add3times-felicitous-difference-1} {\small Examples for the RLA-algorithm-H for adding leaves continuously under the odd-edge felicitous-difference total coloring.}}
\end{figure}

\begin{figure}[h]
\centering
\includegraphics[width=16.4cm]{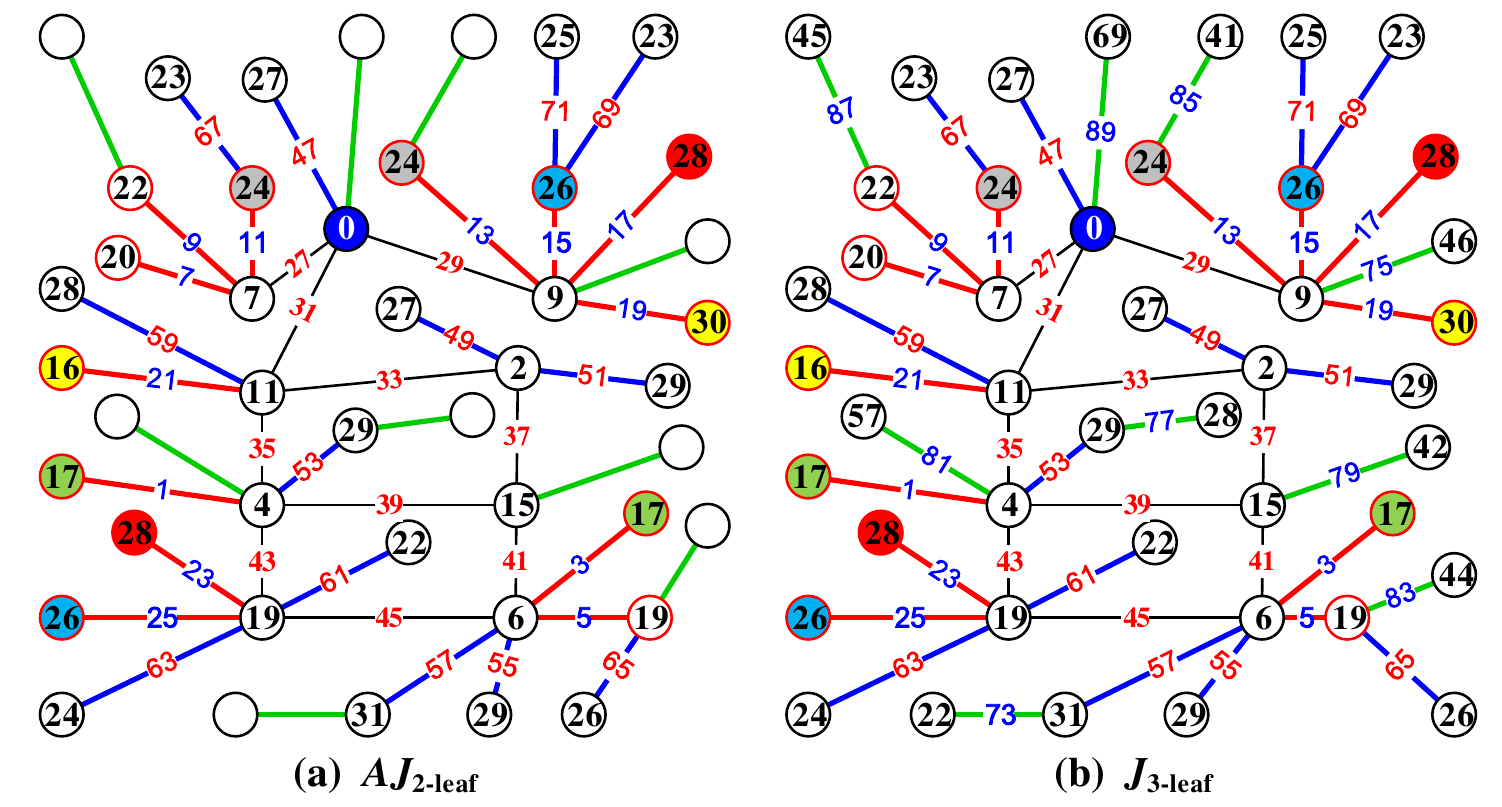}\\
\caption{\label{fig:add3times-felicitous-difference-2} {\small An examples for the RLA-algorithm-H for adding leaves continuously under the odd-edge felicitous-difference total coloring.}}
\end{figure}

\begin{thm}\label{thm:felicitous-difference-adding-leaves-continuously}
Suppose that a connected bipartite $(p,q)$-graph $G_0$ admits an odd-edge felicitous-difference total coloring $h^{gr}_0$, then there are connected bipartite graph sequence $\{G_k\}^n_{k=0}$ such that each connected bipartite graph $G_k\in \{G_k\}^n_{k=0}$ is obtained by adding randomly $a_k~(\geq 1)$ leaves to $G_{k-1}$ and admits an odd-edge felicitous-difference total coloring $h^{gr}_k$ with $k\geq 1$, and moreover $h^{gr}_i(V(G_i))\cap h^{gr}_j(V(G_j))\neq \emptyset$ for any pair integers $i,j\in [0,n]$.
\end{thm}

\begin{thm}\label{thm:felicitous-difference-tree-adding-leaves-continuously}
Each tree admits an odd-edge felicitous-difference total coloring.
\end{thm}

\subsection{Theorems based on the $W$-magic total colorings}

If a graph admits a set-ordered graceful total coloring, then it admits a set-ordered odd-edge $W$-magic total coloring, where $W$-magic $\in \{$edge-magic, edge-difference, felicitous-difference, graceful-difference$\}$.

\begin{thm}\label{thmc:5-equivalent-therems}
If a connected graph admits a set-ordered odd-graceful total labeling, then it admits the following total colorings:

(1) a set-ordered odd-edge edge-magic total labeling;

(2) a set-ordered odd-edge edge-difference total labeling;

(3) a set-ordered odd-edge felicitous-difference total labeling;

(4) a set-ordered odd-edge graceful-difference total labeling,\\
which are equivalent to each other.
\end{thm}
\begin{proof} Suppose that a connected bipartite $(p,q)$-graph $G$ admits a set-ordered odd-graceful total labeling $f$. Let $(X,Y)$ be the bipartition of $V(G)$, where $X=\{x_1,x_2,\dots,x_s\}$ and $Y=\{y_1,y_2,\dots,y_t\}$ ($s+t=p$). Since $f$ is a set-ordered odd-graceful total labeling, without loss of generality, the vertex colors can be arranged into
\begin{equation}\label{equ:set-ordered-graceful-labeling}
0=f(x_1)<f(x_2)<\cdots <f(x_s)<f(y_1)<f(y_2)<\cdots <f(y_t)=2q-1
\end{equation} also $\max f(X)<\min f(Y)$, and each edge $x_iy_j\in E(G)$ holds
$$f(x_iy_j)=|f(y_j)-f(x_i)|=f(y_j)-f(x_i)\in [1,2q-1]^o$$
and the edge color set $f(E(G))=[1,2q-1]^o$.

\textbf{Algorithm-1.} We define a dual total labeling $f^*_{oed}$ by means of the set-ordered odd-graceful total labeling $f$ of the graph $G$ as follows: Each vertex $w\in V(G)$ is colored as
\begin{equation}\label{equ:edge-difference-total-labeling11}
f^*_{oed}(w)=\max f(V(G))+\min f(V(G))-f(w)=2q-1-f(w)
\end{equation} and each edge $x_iy_j\in E(G)$ is colored as
\begin{equation}\label{equ:edge-difference-total-labeling22}
f^*_{oed}(x_iy_j)=\max f(E(G))+\min f(E(G))-f(x_iy_j)=2q-f(x_iy_j)
\end{equation} Then the edge color set $f^*_{oed}(E(G))=f(E(G))=[1,2q-1]^o$. Since
\begin{equation}\label{equ:edge-difference-total-labeling33}
f^*_{oed}(x_iy_j)+|f^*_{oed}(y_j)-f^*_{oed}(x_i)|=2q-f(x_iy_j)+|f(y_j)-f(x_i)|=2q
\end{equation} and
\begin{equation}\label{equ:123456}
0=f^*_{oed}(y_t)<f^*_{oed}(y_{t-1})<\cdots <f^*_{oed}(y_1)<f^*_{oed}(x_{s})<f^*_{oed}(x_{s-1})<\cdots <f^*_{oed}(x_1)=2q-1
\end{equation}
By Definition \ref{defn:group-total-labelings-definition}, we claim that the dual total labeling $f^*_{oed}$ is really a a set-ordered odd-edge edge-difference total labeling of $G$.

\textbf{Algorithm-2.} We use the set-ordered odd-graceful total labeling $f$ of the graph $G$ to define a total labeling $g^*_{ogd}$ as : $g^*_{ogd}(x_i)=\max f(X)+\min f(X)-f(x_i)=\max f(X)-f(x_i)$ for $x_i\in X$ and
$$
g^*_{ogd}(y_j)=\max f(Y)+\min f(Y)-f(y_j)=2q-1+\min f(Y)-f(y_j),~y_j\in Y
$$ and each edge $x_iy_j\in E(G)$ is recolored as
$$
g^*_{ogd}(x_iy_j)=\max f(E(G))+\min f(E(G))-f(x_iy_j)=2q-f(x_iy_j),
$$ so the edge color set $g^*_{ogd}(E(G))=f(E(G))=[1,2q-1]^o$. Because of
\begin{equation}\label{equ:2222222}
{
\begin{split}
& \quad |g^*_{ogd}(x_i)-g^*_{ogd}(y_j)|=g^*_{ogd}(y_j)-g_{ogd}(x_i)\\
&=[\max f(Y)+\min f(Y)-f(y_j)]-[\max f(X)+\min f(X)-f(x_i)]\\
&=[\max f(Y)+\min f(Y)]-[\max f(X)+\min f(X)]-f(x_iy_j)\\
&=2q-1+\min f(Y)-\max f(X)-f(x_iy_j)
\end{split}}
\end{equation}
and $g^*_{ogd}(u)\neq g^*_{ogd}(v)$ for any pair of two vertices $u,v\in V(G)$, we can compute
\begin{equation}\label{eqa:555555}
{
\begin{split}
\big | |g^*_{ogd}(y_j)-g^*_{ogd}(x_i)|-g^*_{ogd}(x_iy_j)\big |&=\big | f(x_iy_j)-[2q-f(x_iy_j)]\big |\\
&=\min f(Y)-\max f(X)-1,
\end{split}}
\end{equation} for each edge $x_iy_j\in E(G)$, notice that $\min f(Y)-\max f(X)-1$ is a constant, and $\max g^*_{ogd}(X)<\min g^*_{ogd}(Y)$, thereby we say that the total labeling $g^*_{ogd}$ is a set-ordered odd-edge graceful-difference total labeling of $G$ according to Definition \ref{defn:group-total-labelings-definition}.

\textbf{Algorithm-3.} The set-ordered odd-graceful total labeling $f$ of the graph $G$ can induce a total labeling $h_{ofd}$ as: $h_{ofd}(x_i)=\max f(X)+\min f(X)-f(x_i)=\max f(X)-f(x_i)$ for $x_i\in X$, $h_{ofd}(y_j)=f(y_j)$ for $y_j\in Y$, and each edge $x_iy_j$ is recolored by $h_{ofd}(x_iy_j)=f(x_iy_j)$ ($x_iy_j\in E(G)$), clearly, $h_{ofd}(E(G))=f(E(G))=[1,2q-1]^o$. Since
\begin{equation}\label{eqa:555555}
{
\begin{split}
h_{ofd}(x_i)+h_{ofd}(y_j)-h_{ofd}(x_iy_j)&=\max f(X)+\min f(X)-f(x_i)+f(y_j)-f(x_iy_j)\\
&=\max f(X)
\end{split}}
\end{equation} Definition \ref{defn:group-total-labelings-definition} shows that the total labeling $h_{ofd}$ is a set-ordered odd-edge felicitous-difference total labeling of $G$.

Again we define a total labeling $h^*_{ofd}$ in the following way: $h^*_{ofd}(w)=h_{ofd}(w)$ for $w\in V(G)$, and each edge $x_iy_j\in E(G)$ is recolored as
$$
h^*_{ofd}(x_iy_j)=\max f(E(G))+\min f(E(G))-f(x_iy_j)=2q-f(x_iy_j)
$$ so the edge color set $h^*_{ofd}(E(G))=f(E(G))=[1,2q-1]^o$. Notice that
\begin{equation}\label{eqa:555555}
{
\begin{split}
h^*_{ofd}(x_i)+h^*_{ofd}(x_iy_j)+h^*_{ofd}(y_j)&=h_{ofd}(x_i)+2q-f(x_iy_j)+h_{ofd}(y_j)\\
&=\max f(X)+h_{ofd}(x_iy_j)+2q-f(x_iy_j)\\
&=\max f(X)+f(x_iy_j)+2q-f(x_iy_j)\\
&=2q+\max f(X)
\end{split}}
\end{equation} we claim that the total labeling $h^*_{ofd}$ is a set-ordered odd-edge edge-magic total labeling of $G$ from Definition \ref{defn:group-total-labelings-definition}.

\textbf{Algorithm-4.} Using the set-ordered odd-graceful total labeling $f$, a total labeling $h^*_{setY}$ of $G$ can be defined as: $h^*_{setY}(x_i)=f(x_i)$ for $x_i\in X$, $h^*_{setY}(y_j)=\max f(Y)+\min f(Y)-f(y_j)=2q-1+\min f(Y)-f(y_j)$ for $y_j\in Y$, as well as each edge $x_iy_j\in E(G)$ is colored by
$$
h^*_{setY}(x_iy_j)=\max f(E(G))+\min f(E(G))-f(x_iy_j)=2q-f(x_iy_j)
$$ we get the edge color set $h^*_{setY}(E(G))=f(E(G))=[1,2q-1]^o$. Moreover, we have
\begin{equation}\label{eqa:555555}
{
\begin{split}
h^*_{setY}(x_i)+h^*_{setY}(y_j)-h^*_{setY}(x_iy_j)=&f(x_i)+2q-1+\min f(Y)-f(y_j)-[2q-f(x_iy_j)]\\
=&\min f(Y)-1.
\end{split}}
\end{equation} Definition \ref{defn:group-total-labelings-definition} enables us to claim that the total labeling $h^*_{setY}$ is a set-ordered odd-edge felicitous-difference total labeling of $G$.

Notice that all transformations in the above four algorithms are linear transformations, thereby, we have completed the necessity and sufficient proof of the theorem.
\end{proof}

\begin{thm}\label{thmc:more-total-colorings}
If a graph admits a set-ordered graceful total coloring, then it admits a set-ordered odd-edge $W$-magic total coloring, where $W$-magic is one of edge-magic, edge-difference, felicitous-difference and graceful-difference.
\end{thm}

\section{Graph lattices based on uniformly-$k^*$ $W$-magic total colorings}

An odd-edge $W$-magic total labeling (or coloring) is one of odd-edge felicitous-difference total labeling (or coloring), odd-edge edge-difference total labeling (or coloring), odd-edge graceful-difference total labeling (or coloring) and odd-edge edge-magic total labeling (or coloring) in this section.

Graph lattices include: \emph{linear-graphic lattices} and \emph{non-linear-graphic lattices}. Simply, a linear-graph lattice is a set of trees obtained from a tree-base $\textbf{\textrm{T}}^c=\{T_1^c,T_2^c,\dots ,T_m^c\}$ and graph operations, where each $T_i$ is a tree; and a non-linear-graphic lattice is the set of graphs obtained from a non-tree base $\textbf{\textrm{H}}^c=\{H_1^c,H_2^c,\dots ,H_m^c\}$ and graph operations, where there is at least one colored graph
 $H_i^c$ to be a non-tree graph admitting an odd-edge $W$-magic total labeling.

\subsection{Uniformly-$k^*$ $W$-magic graphic lattices}

\subsubsection{Uniformly-$k^*$ graceful-difference graphic lattices}
Let $\textbf{\textrm{G}}^c_{k^*\textrm{-magic}}=\{G_1^c,G_2^c,\dots ,G_m^c\}$ be a \emph{uniformly-$k^*$ graceful-difference base} with each $G_i^c$ is a connected bipartite graph and admits an odd-edge graceful-difference total labeling (or coloring) $h^{gr}_i$ holding
\begin{equation}\label{eqa:555555}
\big | |h^{gr}_i(x)-h^{gr}_i(y)|-h^{gr}_i(xy)\big |=k^*\geq 0,~xy\in E(G_i^c),~i\in [1,m]
\end{equation} as well as $h^{gr}_i(u_i)=0$ for some $u_i\in E(G_i^c)$, and $G^c_i\not\cong G^c_j$ for $i\neq j$.

We abbreviate ``Leaf-addling-randomly vertex-coinciding'' as ``LARVC'' in the following argument.

\noindent \textbf{LARVC uniformly-$k^*$ graceful-difference algorithm.}

\textbf{Larvc-Step-1.} Let $T_{i_1},T_{i_2},\dots ,T_{i_M}$ be a permutation of graphs $\alpha_1G_1^c,\alpha_2G_2^c,\dots ,\alpha_mG_m^c$ based on a uniformly-$k^*$ graceful-difference base $\textbf{\textrm{G}}^c_{k^*\textrm{-magic}}$, where $M=\sum ^m_{k=1}\alpha_k\geq 1$, and each graph $T_{i_s}$ is connected and admits an odd-edge graceful-difference total labeling (or coloring) $h^{gr}_{i_s}$.

\textbf{Larvc-Step-2.} Adding $m_{i_j}~(\geq 1)$ leaves to some vertices of each connected graph $T_{i_j}$ for $j\in [1,M]$ produces a connected graph $H_{i_j}$, denoted as $H_{i_j}=\langle m_{i_j}[\ominus_e] T_{i_j}\rangle $, admitting an odd-edge graceful-difference total labeling (or coloring) $f^{gr}_{i_j}$ induced from the odd-edge graceful-difference total labeling (or coloring) of the connected graph $T_{i_j}$, such that
\begin{equation}\label{eqa:555555}
{
\begin{split}
&\big | |f^{gr}_{i_j}(x)-f^{gr}_{i_j}(y)|-f^{gr}_{i_j}(xy)\big |=k^*,xy\in E(H_{i_j});\\
&h^{gr}_{i_s}(V(H_{i_s}))\cap h^{gr}_{i_{t}}(V(H_{i_{t}}))\neq \emptyset,1\leq s,t\leq M
\end{split}}
\end{equation} by the RLA-algorithms for adding leaves continuously. Notice that RLA-algorithm of the odd-edge graceful-difference total coloring tells us: The leaf-added graph $G_A$ admits an odd-edge graceful-difference total coloring $f_{grd}^*$, and $f_{grd}^*(u_0)=0$ for some $u_0\in V(G_A)$, thereby, $f^{gr}_{i_j}(w_0)=0$ for some $w_0\in V(H_{i_j})$.

\textbf{Larvc-Step-3.} We vertex-coincide a vertex $u\in V(H_{i_s})$ with a vertex $v\in V(H_{i_j})$ into one vertex $u\odot v$ if $h^{gr}_{i_s}(u)=h^{gr}_{i_j}(v)$ for $s\neq j$ and $1\leq s,t\leq M$, such that the resultant graphs are connected and has no multiple edges, called \emph{simple vertex-coincided graphs}, and we denote these simple vertex-coincided graphs by the following form
\begin{equation}\label{eqa:k-magic-add-leaf-colorings}
\odot ^M_{j=1}\langle H_{i_j}\rangle=\odot ^M_{j=1}\langle m_{i_j}[\ominus_e] T_{i_j}\rangle =[\odot ^m_{k=1}]\ominus _e\langle \alpha_kG_k^c\rangle
\end{equation}

By the LARVC uniformly-$k^*$ graceful-difference algorithm introduced above, we get a \emph{uniformly-$k^*$ magic-type graphic lattice}
\begin{equation}\label{eqa:uniformly-k-magic-add-leaf-graphic-lattices}
\textbf{\textrm{L}}_{grd}(Z^0[\odot_v]\langle \ominus_e\rangle \textbf{\textrm{G}}^c_{k^*\textrm{-magic}})=\left \{[\odot ^m_{k=1}]\ominus _e\langle \alpha_kG_k^c\rangle :~\alpha_k\in Z^0,~G_k^c\in \textbf{\textrm{G}}^c_{k^*\textrm{-magic}}\right \}
\end{equation} with $\sum ^m_{k=1}\alpha_k\geq 1$.

\begin{rem}\label{rem:333333}
Each graph $T^*\in \textbf{\textrm{L}}_{grd}(Z^0[\odot_v]\langle \ominus_e\rangle \textbf{\textrm{G}}^c_{k^*\textrm{-magic}})$ is a connected bipartite graph and admits a \emph{compound odd-edge graceful-difference total labeling} (or coloring) $F$ holding $\big | |F(x)-F(y)|-F(xy)\big |=k^*$ for $xy\in E(T^*)$, where $F(z)=f^{gr}_{i_j}(z)$ for $z\in V(H_{i_j})\cup E(H_{i_j})\subset V(T^*)\cup E(T^*)$. Each uniformly-$k^*$ magic-type graphic lattice was made by two graph operations: one is the vertex-coinciding operation ``$\odot$'' and, another is the leaf-adding operation ``$\ominus$'', so the lattice $\textbf{\textrm{L}}_{grd}(Z^0[\odot_v]\langle \ominus_e\rangle \textbf{\textrm{G}}^c_{k^*\textrm{-magic}})$ is, also, a \emph{multiple-operation lattice}.

We do the vertex-coinciding operation to each graph $T^*\in \textbf{\textrm{L}}_{grd}(Z^0[\odot_v]\langle \ominus_e\rangle \textbf{\textrm{G}}^c_{k^*\textrm{-magic}})$ by vertex-coinciding those vertices of $T^*$ colored the the same color, and avoiding that case of multiple-edges, the resultant graph is denoted as $T^*_{\textrm{v-coin}}$, then we get another uniformly-$k^*$ magic-type graphic lattice as follows:
\begin{equation}\label{eqa:uniformly-k-magic-add-leaf-graphic-lattices}
{
\begin{split}
\textbf{\textrm{L}}_{grd}(Z^0[\odot_v]^2\langle \rightarrow\rangle \textbf{\textrm{G}}^c_{k^*\textrm{-magic}})=&\big \{T^*_{\textrm{v-coin}}:~T^*\rightarrow _{\textrm{v-coin}}T^*_{\textrm{v-coin}},\\
&T^*\in \textbf{\textrm{L}}_{grd}(Z^0[\odot_v]\langle \ominus_e\rangle \textbf{\textrm{G}}^c_{k^*\textrm{-magic}})\big \}
\end{split}}
\end{equation} so we call the following one graph set being homomorphism to another graph set
\begin{equation}\label{eqa:555555}
\textbf{\textrm{L}}_{grd}(Z^0[\odot_v]\langle \ominus_e\rangle \textbf{\textrm{G}}^c_{k^*\textrm{-magic}}) \rightarrow _{\textrm{v-coin}}\textbf{\textrm{L}}_{grd}(Z^0[\odot_v]^2\langle \rightarrow\rangle \textbf{\textrm{G}}^c_{k^*\textrm{-magic}})
\end{equation} a \emph{uniformly-$k^*$ magic-type graphic-lattice homomorphism}.\paralled
\end{rem}

\begin{problem}\label{problem:99999}
\textbf{Does} there is a group of graphs $T_1^c,T_2^c,\dots ,T_m^c$, such that $\textbf{\textrm{T}}^c_{k^*\textrm{-magic}}=\{T_1^c,T_2^c,\dots ,T_m^c\}$ forms a uniformly-$k^*$ graceful-difference base, and two uniformly-$k^*$ graceful-difference graphic lattices hold
$$\textbf{\textrm{L}}_{grd}(Z^0[\odot_v]\langle \ominus_e\rangle \textbf{\textrm{T}}^c_{k^*\textrm{-magic}})=\textbf{\textrm{L}}_{grd}(Z^0[\odot_v]\langle \ominus_e\rangle \textbf{\textrm{G}}^c_{k^*\textrm{-magic}})?$$
\end{problem}

\subsubsection{Complexity of uniformly-$k^*$ graceful-difference graphic lattices}

According to the LARVC uniformly-$k^*$ graceful-difference algorithm, we have the following complexity analysis:

\textbf{Case-1.} In Larvc-Step-1 of the LARVC uniformly-$k^*$ graceful-difference algorithm, there are $M!$ permutations $T_{i_1},T_{i_2},\dots ,T_{i_M}$ obtained from graphs $\alpha_1G_1^c,\alpha_2G_2^c,\dots ,\alpha_mG_m^c$ based on a uniformly-$k^*$ graceful-difference base $\textbf{\textrm{G}}^c_{k^*\textrm{-magic}}$, where $M=\sum ^m_{k=1}\alpha_k\geq 1$.

\textbf{Case-2.} In Larvc-Step-2 of the LARVC uniformly-$k^*$ graceful-difference algorithm, adding $m_{i_j}~(\geq 1)$ leaves to some vertices of each connected graph $T_{i_j}$ with $j\in [1,M]$, we will meet:

(i) \textbf{Integer Partition Problem}: $m_{i_j}=a_{i_j,1}+a_{i_j,2}+\cdots +a_{i_j,b_{i_j}}$ with integers $a_{i_j,k}\geq 1$ and $b_{i_j}\geq 2$. This problem is related with an odd integer $m=p_1+p_2+p_3$ for primes $p_1,p_2,p_3$, also, the famous Goldbach's Conjecture. Suppose that we have $P(m_{i_j},b_{i_j})$ different ways.

(ii) Selecting randomly $b_{i_j}$ vertices of $T_{i_j}$ produces ${p(T_{i_j}) \choose b_{i_j}}$ methods for adding $m_{i_j}$ leaves, where vertex number $p(T_{i_j})=|V(T_{i_j})|$. So, we have $b_{i_j}!\cdot {p(T_{i_j}) \choose b_{i_j}}$ different methods in total.

Summing up the above works, then we have $\prod ^{M}_{i_j=1}P(m_{i_j},b_{i_j})(b_{i_j}!){p(T_{i_j}) \choose b_{i_j}}$ different methods for adding leaves to a permutation $T_{i_1},T_{i_2},\dots ,T_{i_M}$.

\textbf{Case-3.} In Larvc-Step-3 of the LARVC uniformly-$k^*$ graceful-difference algorithm, computing the number $n(\odot ^A_{j=1}T_{i_j})$ of graphs in the form $[\odot ^m_{k=1}]\ominus _e\langle \alpha_kG_k^c\rangle$ defined in Eq.(\ref{eqa:k-magic-add-leaf-colorings}) is extremely difficult.

\vskip 0.4cm

By Case-1, Case-2 and Case-3, we can say that for each $M=\sum ^m_{k=1}\alpha_k\geq 1$, there are at least
\begin{equation}\label{eqa:exponential-level-numbers}
n(M,\textbf{\textrm{G}}^c_{k^*\textrm{-magic}})=M!\cdot n(\odot ^A_{j=1}T_{i_j})\cdot \prod ^{A}_{i_j=1}P(m_{i_j},b_{i_j})(b_{i_j}!){p(T_{i_j}) \choose b_{i_j}}
\end{equation}
simple vertex-coincided graphs in the lattice $\textbf{\textrm{L}}_{grd}(Z^0[\odot_v]\langle \ominus_e\rangle \textbf{\textrm{G}}^c_{k^*\textrm{-magic}})$ defined in Eq.(\ref{eqa:uniformly-k-magic-add-leaf-graphic-lattices}).

In a simple case, we vertex-coincide a vertex $u\in V(T_{i_s})$ with a vertex $v\in V(T_{i_{s+1}})$ into one vertex $u\odot v$ as $h^{gr}_{i_s}(u)=h^{gr}_{i_{s+1}}(v)$ for $s\in [1,M-1]$, the resultant graph is called a \emph{string-form graph}, so we have $M!$ string-form graphs in the form $[\odot ^m_{k=1}]\ominus _e\langle \alpha_kG_k^c\rangle$, and for each $M=\sum ^m_{k=1}\alpha_k\geq 1$, there are at least
\begin{equation}\label{eqa:exponential-level-number-11}
n_{\textrm{string}}(M,\textbf{\textrm{G}}^c_{k^*\textrm{-magic}})=\frac{n(M,\textbf{\textrm{G}}^c_{k^*\textrm{-magic}})}{n(\odot ^A_{j=1}T_{i_j})}=M!\cdot \prod ^{A}_{i_j=1}P(m_{i_j},b_{i_j})(b_{i_j}!){p(T_{i_j}) \choose b_{i_j}}
\end{equation}
string-form graphs in the lattice $\textbf{\textrm{L}}_{grd}(Z^0[\odot_v]\langle \ominus_e\rangle \textbf{\textrm{G}}^c_{k^*\textrm{-magic}})$ defined in Eq.(\ref{eqa:uniformly-k-magic-add-leaf-graphic-lattices}).

\begin{rem}\label{rem:55555555555}
In the application of topological authentication, a uniformly-$k^*$ graceful-difference base $\textbf{\textrm{G}}^c_{k^*\textrm{-magic}}=\{G_1^c$, $G_2^c$, $\dots $, $G_m^c\}$ can be considered as a \emph{public-key base}, given a fixed group of non-negative integers $\alpha_1,\alpha_2,\dots ,\alpha_m$ with $\sum ^m_{k=1}\alpha_k\geq 1$, the simple vertex-coincided graphs $[\odot ^m_{k=1}]\ominus _e\langle \alpha_kG_k^c\rangle$ admitting odd-edge graceful-difference total labelings (or colorings) in the lattice $\textbf{\textrm{L}}_{grd}(Z^0[\odot_v]\langle \ominus_e\rangle \textbf{\textrm{G}}^c_{k^*\textrm{-magic}})$ are as \emph{private-keys}.

However, finding a particular private-key from those simple vertex-coincided graphs in the lattice $\textbf{\textrm{L}}_{grd}(Z^0[\odot_v]\langle \ominus_e\rangle \textbf{\textrm{G}}^c_{k^*\textrm{-magic}})$ is not relaxed, since it will meet the Graph Isomorphic Problem, and it will be encountered with exponential level calculations, refer to Eq.(\ref{eqa:exponential-level-numbers}) and Eq.(\ref{eqa:exponential-level-number-11}).

The sentence ``LARVC uniformly-$k^*$ magic-type'' is one of LARVC uniformly-$k^*$ edge-magic, LARVC uniformly-$k^*$ edge-difference, LARVC uniformly-$k^*$ graceful-difference and LARVC uniformly-$k^*$ felicitous-difference, and each LARVC uniformly-$k^*$ magic-type algorithm is exactly like the LARVC uniformly-$k^*$ graceful-difference algorithm.

Since the complexities of uniformly-$k^*$ magic-type graphic lattices for other uniformly-$k^*$ magic-types are like the complexity analysis of uniformly-$k^*$ graceful-difference graphic lattices, we omit them here. \paralled
\end{rem}

\subsubsection{Twin uniformly $k^*$-magic-graphic lattices}

\begin{thm}\label{them:twin-w-type-total-colorings}
If a connected bipartite $(p,q)$-graph $G$ admits an odd-edge graceful-difference total coloring $f^*$, then there exists a bipartite graph $G^*$ admitting an odd-edge graceful-difference total coloring $g^*$, such that $\langle f^*, g^*\rangle $ is a twin set-ordered odd-edge graceful-difference total coloring of the $(p,q)$-graph $G$ and the graph $G^*$.
\end{thm}
\begin{proof} Suppose the connected bipartite $(p,q)$-graph $G$ admits an odd-edge graceful-difference total coloring $f^*: V(G)\cup E(G)\rightarrow [0,2q-1]$, then there exists a bipartite graph $G^*$ admitting an odd-edge graceful-difference total coloring $g^*: V(G^*)\cup E(G^*)\rightarrow [1,2q]$, such that $G\cong G^*$, and $g^*(w)=f^*(w)+1$ for $w\in V(G)=V(G^*)$, as well as $g^*(e)=f^*(e)$ for $w\in E(G)=E(G^*)$. By Definition \ref{defn:group-definition-twin-total-labelingss}, $\langle f^*, g^*\rangle $ is a twin set-ordered odd-edge graceful-difference total coloring of the $(p,q)$-graph $G$ and the graph $G^*$.
\end{proof}

Let $\textbf{\textrm{G}}^{\textrm{twin}}_{n^*\textrm{-magic}}=\{G_1^{\textrm{twin}},G_2^{\textrm{twin}},\dots ,G_m^{\textrm{twin}}\}$ be a uniformly-$n^*$ graceful-difference graphic base, each graph $G_i^{\textrm{twin}}$ is connected and admits an odd-edge graceful-difference total coloring $\alpha^{gr}_i$, and holds
\begin{equation}\label{eqa:twin-uniformly-k-graceful-difference-graphic-bases}
\big | |\alpha^{gr}_i(x)-\alpha^{gr}_i(y)|-\alpha^{gr}_i(xy)\big |=n^*\geq 0,~xy\in E(G_i^{\textrm{twin}}),~i\in [1,m]
\end{equation} and $\alpha^{gr}_i(u_i)=1$ for $u_i\in E(G_i^{\textrm{twin}})$, as well as $G^{\textrm{twin}}_i\not\cong G^{\textrm{twin}}_j$ if $i\neq j$.

From $i=1$ to $i=m$, if $(h^{gr}_i,\alpha^{gr}_i)$ is the twin odd-edge graceful-difference total labeling/total coloring of the uniformly-$k^*$ graceful-difference graphic base $\textbf{\textrm{G}}^c_{k^*\textrm{-magic}}$ and the uniformly-$n^*$ graceful-difference graphic base $\textbf{\textrm{G}}^c_{k^*\textrm{-magic}}$ (refer to Definition \ref{defn:group-definition-twin-total-labelingss}), then two graceful-difference graphic base $\textbf{\textrm{G}}^c_{k^*\textrm{-magic}}$ and $\textbf{\textrm{G}}^{\textrm{twin}}_{n^*\textrm{-magic}}$ form a \emph{twin uniformly-$(k^*,n^*)$ graceful-difference graphic base}.

\vskip 0.4cm

By the above LARVC-algorithm, we have a \emph{uniformly-$n^*$ graceful-difference graphic lattice} based on the uniformly-$n^*$ graceful-difference graphic base $\textbf{\textrm{G}}^{\textrm{twin}}_{n^*\textrm{-magic}}$ as follows
\begin{equation}\label{eqa:twin-uniformly-k-magic-add-leaf-graphic-lattices}
\textbf{\textrm{L}}_{grd}(Z^0[\odot_v]\langle \ominus_e\rangle \textbf{\textrm{G}}^{\textrm{twin}}_{n^*\textrm{-magic}})=\big \{[\odot ^m_{k=1}]\ominus _e\langle \beta_kG_k^{\textrm{twin}}\rangle :~\beta_k\in Z^0,~G_k^{\textrm{twin}}\in \textbf{\textrm{G}}^{\textrm{twin}}_{n^*\textrm{-magic}}\big \}
\end{equation} where $\sum ^m_{k=1}\beta_k\geq 1$, such that each graph $J^*\in \textbf{\textrm{L}}_{grd}(Z^0[\odot_v]\langle \ominus_e\rangle \textbf{\textrm{G}}^{\textrm{twin}}_{n^*\textrm{-magic}})$ is connected and admits an odd-edge graceful-difference total coloring $g$ holding $\big | |g(x)-g(y)|-g(xy)\big |=n^*$ for $xy\in E(J^*)$.

We call two graphic lattices $\textbf{\textrm{L}}_{grd}(Z^0[\odot_v]\langle \ominus_e\rangle \textbf{\textrm{G}}^{\textrm{twin}}_{n^*\textrm{-magic}})$ and $\textbf{\textrm{L}}_{grd}(Z^0[\odot_v]\langle \ominus_e\rangle \textbf{\textrm{G}}^c_{k^*\textrm{-magic}})$ \emph{twin uniformly-$\langle k^*,n^*\rangle $ graceful-difference graphic lattice}. In real application, the uniformly-$k^*$ graceful-difference graphic lattice is as a \emph{public-key lattice} and the uniformly-$n^*$ graceful-difference graphic lattice is as a \emph{private-key lattice}.

\subsection{Realization of uniformly-$n^*$ $W$-magic graphic lattices \cite{Zhang-Zhang-Yao-integer-lattices-2022}}

\subsubsection{Connection between complex graphs and integer lattices}

1. \textbf{Complex graphs and integer lattices.} Let $\textbf{\textrm{F}}^*(n)$ be the set of complex graphs of $n$ vertices (refer to Definition \ref{defnc:complex-graph-definition}). A graph base $\textbf{\textrm{H}}=(H_1,H_2,\dots ,H_n)$ consists of $n$ vertex-disjoint connected complex graphs, each connected complex graph $H_k$ has just $m$ vertices and its own degree sequence $\textbf{\textrm{d}}_k=(d_{k,1},d_{k,2},\dots ,d_{k,m})=(d_{k,i})^m_{i=1}$, where $d_{k,i}\geq d_{k,i+1}$ for $i\in [1,m-1]$. We vertex-coincide a vertex $x_k$ of a connected complex graph $G\in \textbf{\textrm{F}}^*$ with a vertex $w_{k,j}$ of the connected complex graph $H_k$ into a vertex $x_k\odot w_{k,j}$ for $k\in [1,n]$, the resultant graph is denoted as $F\odot ^n_{k=1}H_k$, and we get a degree sequence $\textbf{\textrm{d}}'_k=(d\,'_{k,1},d\,'_{k,2},\dots ,d\,'_{k,m})=(d\,'_{k,j})^m_{j=1}$, where only one $d\,'_{k,j}=d_{k,j}+\textrm{deg}_G(x_k)$, here $\textrm{deg}_G(x_k)$ is the degree of vertex $x_k$ of the connected complex graph $G$. Thereby, the graph $L=F\odot ^n_{k=1}H_k$ has its own degree sequence
\begin{equation}\label{eqa:555555}
\textbf{\textrm{d}}_L=\sum^n_{k=1} \textbf{\textrm{d}}\,'_k=\left (\sum^n_{k=1} d\,'_{k,1},~\sum^n_{k=1} d\,'_{k,2},~\dots , ~\sum^n_{k=1} d\,'_{k,m}\right )
\end{equation} and these degree sequences forms an integer lattice
\begin{equation}\label{eqa:555555}
\textbf{\textrm{L}}(\textbf{\textrm{d}})=\left \{\sum^n_{k=1} \lambda_k\textbf{\textrm{d}}\,'_k:~\lambda_k \in Z,\textbf{\textrm{d}}\,'_k=(d_{k,i})^m_{i=1},~ H_{k}\in \textbf{\textrm{H}},~ F\in \textbf{\textrm{F}}^*(n)\right \}.
\end{equation}

2. \textbf{A connection between integer lattices and complex graphs.} An integer lattice $\textbf{\textrm{L}}(Z\textbf{\textrm{B}})$ is defined in Definition \ref{defn:traditional-lattice00}. By the lattice base $\textbf{\textrm{B}}=(\textbf{\textrm{b}}_1,\textbf{\textrm{b}}_2,\dots ,\textbf{\textrm{b}}_m)$ with $m\leq n$, we have a vector
\begin{equation}\label{eqa:lattice-degree-sequence}
\sum^m_{i=1} x_i\textbf{\textrm{b}}_i=\left (\sum^m_{i=1} b_{i,1}, \sum^m_{i=1} b_{i,2}, \dots ,\sum^m_{i=1} b_{i,m}\right )=(\alpha_1,\alpha_2, \dots ,\alpha_m),
\end{equation}
which corresponds to a complex graph $C_{gra}$, such that the degree sequence of the complex graph $C_{gra}$ is just $\textrm{deg}(C_{gra})=(\alpha_1,\alpha_2, \dots ,\alpha_m)$, we call the set $\textbf{\textrm{A}}_{ccom}(\textbf{\textrm{L}}(Z\textbf{\textrm{B}})\rightarrow C_{gra} )$ of these complex graphs like $C_{gra}$ as \emph{complement complex-graphic lattice} of the integer lattice $\textbf{\textrm{L}}(Z\textbf{\textrm{B}})$.

\begin{problem}\label{problem:99999}
\textbf{Partition} a degree sequence $(\alpha_1,\alpha_2, \dots ,\alpha_m)$ (as a public-key) into $\sum^m_{i=1} x_i\textbf{\textrm{b}}_i$ defined in Eq.(\ref{eqa:lattice-degree-sequence}) (as a private-key). Clearly, the number of ways of partitioning a degree sequence is not unique and difficult to estimate, see the Integer Partition Problem.
\end{problem}

\subsubsection{Caterpillar-graphic lattices, Complementary graphic lattices}

The authors in \cite{yirong-sun-bing-yao-2022-nwnu} present the following ODD-GRACEFUL subdivision-algorithm.

\vskip 0.2cm

\noindent \textbf{ODD-GRACEFUL subdivision-algorithm}

\textbf{Input:} A caterpillar $H$ having $m=\sum^n_{i=1}|L(u_i)|=\sum^n_{i=1}m_i$ leaves, where $P=u_1u_2\cdots u_n$ is the spine path of the caterpillar $H$.

\textbf{Output:} A set-ordered odd-graceful labeling of the caterpillar $H$.

\textbf{Step 1.} Notice that the caterpillar $H$ has $m=\sum^n_{i=1}|L(u_i)|=\sum^n_{i=1}m_i$ leaves. We take a star tree $K_{1,m}$ with vertex set $V(K_{1,m})=\{u_1,v_{1,j}\mid j\in [1,m]\}$ and edge set $E(K_{1,m})=\{u_1v_{1,1},u_1v_{1,2},\dots ,u_1v_{1,m}\}$. Let $H_1=K_{1,m}$, and $L(H_1)=V(H_1)\setminus \{u_1\}$ is the set of leaves of $K_{1,m}$. We define a labeling $\alpha_1$ for the star tree $H_1$ as: $\alpha_1(u_1)=0$, $\alpha_1(v_{1,j})=2j-1$. So, $H_1$ has its own vertex color set $\alpha_1(V(H_1))=[0,2m-1]$ and edge color set $\alpha_1(E(H_1))=[1,2m-1]^o$. Obviously, this labeling $\alpha_1$ is just a set-ordered odd-graceful labeling of $H_1$.

\textbf{Step 2.} Add a new vertex $u_2$ to the star tree $H_1$, and join the vertex $u_1$ with the vertex $u_2$ by a new edge $u_1u_2$; and partition the leaf set $L(H_1)$ into two subsets $L(u_1)$ and $L(u_2)$, where $L(u_1)=\{v_{1,m},v_{1,m-1},\dots ,v_{1,m-m_1+1}\}$ and $L(u_2)=\{v_{2,m-m_1},v_{2,m-m_1-1},\dots ,v_{2,1}\}$, we have $v_{2,m-m_1+1-k}=v_{1,m-m_1+1-k}$ for $k\in [1,m-m_1]$. Thereby, we get a caterpillar, denoted as $H_2$, having its spine path $P_2=u_1u_2$.

\textbf{Step 3.} Notice that the set-ordered odd-graceful labeling $\alpha_1$ of the star tree $H_1$ holds $\alpha_1(u_1)=0$, $\alpha_1(v_{1,m-j+1})$ $=2(m-j+1)-1$ for $j\in [0,m_1-1]$, $\alpha_1(v_{2,m-m_1+1-k})=2(m-m_1+1-k)-1$ for $k\in [1,m-m_1]$. We define a labeling $\alpha_2$ for the caterpillar $H_2$ in the following way: $\alpha_2(u_1)=0$, $\alpha_2(v_{1,m-j+1})=\alpha_1(v_{1,m-j+1})+2$ for $j\in [0,m_1-1]$; $\alpha_2(u_2)=2(m-m_1)+1$, $\alpha_2(v_{2,m-m_1+1-k})=\alpha_2(u_2)-(2k-1)$ for $k\in [1,m-m_1]$. It is easy to verify that $\alpha_2$ is just a set-ordered odd-graceful labeling of $H_2$.

\textbf{Step 4.} Add a new vertex $u_{k+1}$ to the caterpillar $H_k$, and join the vertex $u_{k}$ of the spine path of the caterpillar $H_k$ with new vertex $u_{k+1}$ by a new edge $u_{k}u_{k+1}$, and partition the leaf set $L(u_{k})$ of $H_k$ into two leaf subsets, then join the leaves of two leaf subsets with $u_{k}$ and $u_{k+1}$ respectively, the resultant graph is just a new caterpillar $H_{k+1}$, and the caterpillar $H_{k+1}$ has its own spine path $P_{k+1}=u_1u_2\cdots u_{k+1}$. Repeat the works in Step 2 and Step 3, until we get a set-ordered odd-graceful labeling of the caterpillar $H$.

\vskip 0.2cm

\begin{figure}[h]
\centering
\includegraphics[width=16.4cm]{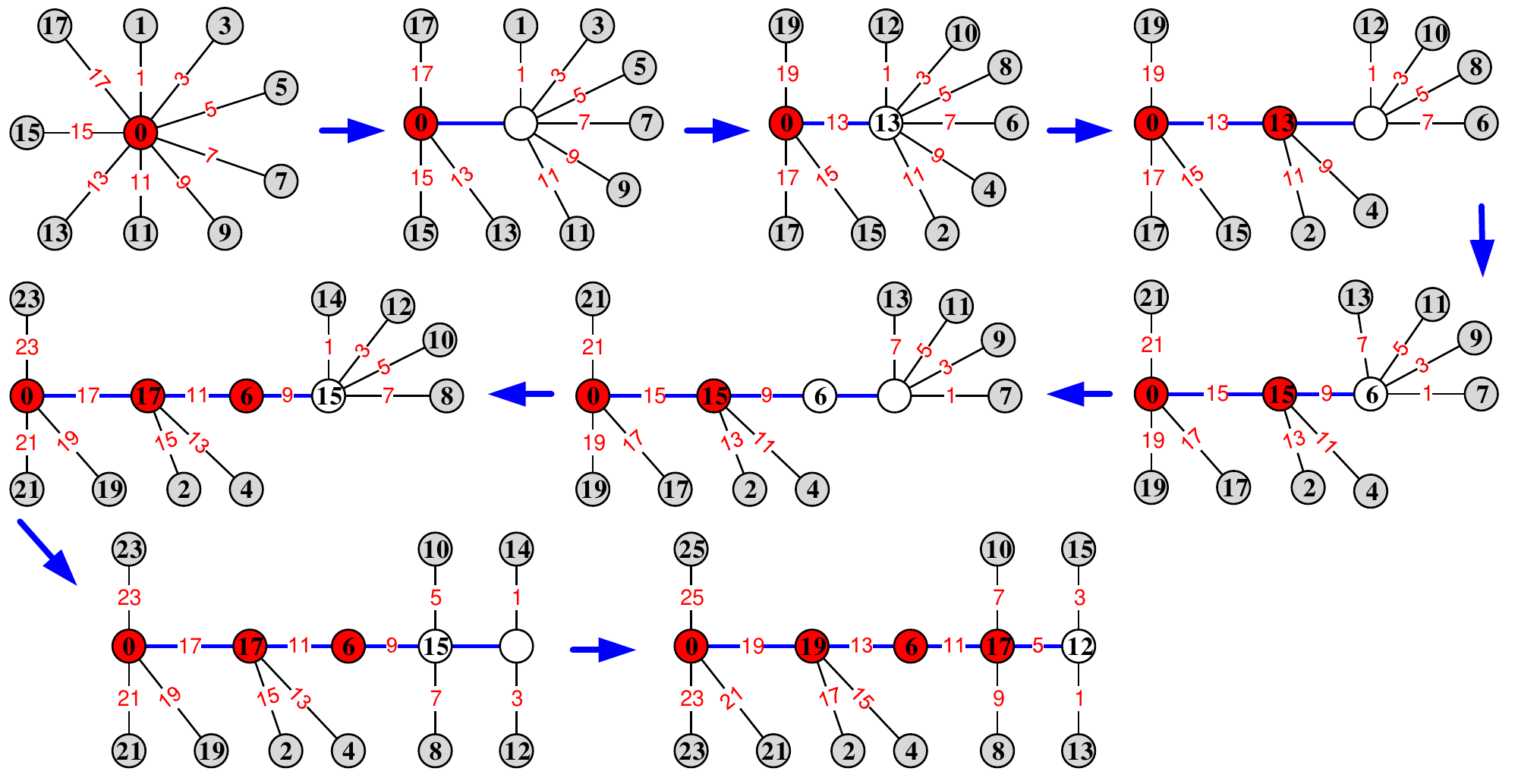}\\
\caption{\label{fig:caterpillar-subdivision-algorithm} {\small An example for illustrating the ODD-GRACEFUL subdivision-algorithm.}}
\end{figure}

See an example for the ODD-GRACEFUL subdivision-algorithm shown in Fig.\ref{fig:caterpillar-subdivision-algorithm}. Notice that a set-ordered odd-graceful labeling of a caterpillar $H$ is just a set-ordered odd-graceful total labeling. By the ODD-GRACEFUL subdivision-algorithm and Theorem \ref{thmc:5-equivalent-therems}, each caterpillar admits one of the labelings and colorings defined in Definition \ref{defn:basic-W-type-labelings} and Definition \ref{defn:group-total-labelings-definition}, and there are algorithms that can be effectively and quickly apply theses labelings and colorings to practice.

\vskip 0.2cm

1. \textbf{Caterpillar-graphic lattices.} A \emph{caterpillar base} $\textbf{\textrm{T}}_{\textrm{cater}}=(T_1,T_2,\dots ,T_m)$ consists of $m$ vertex-disjoint caterpillars $T_1,T_2,\dots ,T_m$. Under a graph operation ``$(\bullet)$'', we call the following set
\begin{equation}\label{eqa:caterpillae-graph-lattice}
\textbf{\textrm{L}}(\textbf{\textrm{G}}(\bullet) \textbf{\textrm{T}}_{\textrm{cater}})=\big \{J(\bullet)^n_{k=1}\beta_kT_k:~\beta_k\in Z^0, T_k\in \textbf{\textrm{T}}_{\textrm{cater}}, J\in \textbf{\textrm{G}}\big \}
\end{equation} as \emph{$(\bullet)$-operation caterpillar-graphic lattice}, where $\sum ^n_{k=1}\beta_k\geq 1$.

\vskip 0.2cm

2. \textbf{A connection between integer lattices and caterpillar-graphic lattices.} A caterpillar $T_i$ has its own spine path $P_{i,n}=x_{i,1}x_{i,2}\cdots x_{i,p}$ with $p\geq 2$, and each vertex $x_{i,k}$ of the spine path $P_{i,p}$ has its own leaf set $L(x_{i,k})=\{y_{i,k,j}:j\in [1,b_{i,k}]\}$ for $k\in [1,p]$. We define the \emph{leaf topological vector} of the caterpillar $T_i$ by $\textbf{\textrm{V}}_{ec}(T_i)=(c_{i,1},c_{i,2},\dots ,c_{i,p})$, where $|c_{i,k}|=|L(x_{i,k})|$ is the number of leaves adjacent with the vertex $x_{i,k}$, also, $c_{i,k}$ is a \emph{leaf-degree} or a \emph{leaf-image-degree}. In Fig.\ref{fig:maomao-degree-sequence}, a caterpillar $A_1$ has its own leaf topological vector $\textbf{\textrm{V}}_{ec}(A_1)=(7,3,0,3,0,6)$, and another caterpillar $A_2$ has its own leaf topological vector $\textbf{\textrm{V}}_{ec}(A_2)=(7,-3,0,-3,0,6)$.

A \emph{complex caterpillar base} $\textbf{\textrm{T}}_{\textrm{cater}}=(T_1,T_2,\dots ,T_m)$ has its own \emph{leaf topological vector base} $\textbf{\textrm{V}}_{ec}(\textbf{\textrm{T}})=(\textbf{\textrm{V}}_{ec}(T_1)$, $\textbf{\textrm{V}}_{ec}(T_2)$, $\dots $, $\textbf{\textrm{V}}_{ec}(T_m))$, immediately, we get an integer lattice
\begin{equation}\label{eqa:topological-vector-lattice}
\textbf{\textrm{L}}(Z(\Sigma )\textbf{\textrm{V}}_{ec}(\textbf{\textrm{T}}))=\left \{\sum ^m_{k=1}\lambda_k\textbf{\textrm{V}}_{ec}(T_k):~\lambda_k\in Z, \textbf{\textrm{V}}_{ec}(T_k)\in \textbf{\textrm{V}}_{ec}(\textbf{\textrm{T}})\right \}.
\end{equation} where $\sum ^m_{k=1}\lambda_k\geq 1$.

\vskip 0.2cm

3. \textbf{Complement caterpillar-graphic lattices.} For two caterpillar bases $\textbf{\textrm{T}}_{\textrm{cater}}=(T_1,T_2$, $\dots $, $T_m)$ and $\textbf{\textrm{T}}^*_{\textrm{cater}}=(T^*_1,T^*_2,\dots ,T^*_m)$,
each caterpillar $T_j$ has its own spine path $P_j=x_{j,1}x_{j,2}\cdots x_{j,n}$ with $n\geq 1$, where each vertex of $P_j$ has its own leaf set $L_{eaf}(x_{j,i})=\{y_{j,i,s}:s\in [1, b_{j,i}]\}$ for $i\in [1,n]$ and $j\in [1,m]$; and each caterpillar $T^*_j$ has its own spine path $P^*_j=x^*_{j,1}x^*_{j,2}\cdots x^*_{j,n}$ with $n\geq 1$, where each vertex of $P^*_j$ has its own leaf set $L_{eaf}(x^*_{j,i})=\{y^*_{j,i,s}:s\in [1, b^*_{j,i}]\}$ for $i\in [1,n]$ and $j\in [1,m]$. If $|L_{eaf}(x_{j,i})|+|L_{eaf}(x^*_{j,i})|=M$ for two caterpillars $T_j$ and $T^*_j$, we call two caterpillar bases $\textbf{\textrm{T}}_{\textrm{cater}}$ and $\textbf{\textrm{T}}^*_{\textrm{cater}}$ as \emph{$M$-leaf complement caterpillar base matching}, denoted as $M\langle \textbf{\textrm{T}}_{\textrm{cater}},\textbf{\textrm{T}}^*_{\textrm{cater}}\rangle$.
By Eq.(\ref{eqa:caterpillae-graph-lattice}), two graph lattices $\textbf{\textrm{L}}(\textbf{\textrm{G}}(\bullet) \textbf{\textrm{T}}_{\textrm{cater}})$ and $\textbf{\textrm{L}}(\textbf{\textrm{G}}(\bullet) \textbf{\textrm{T}}^*_{\textrm{cater}})$ form a \emph{complement caterpillar-graphic lattice matching}.

\begin{problem}\label{problem:99999}
A uniformly $M$-complement sequence $\{(a_{i,1},a_{i,2}, \dots ,a_{i,n})\}^p_{i=1}$ holds $\sum ^p_{i=1}a_{i,j}=M$ for $j\in [1,n]$, $p\geq 2$ and $n\geq 2$. \textbf{Generalize} the complement caterpillar-graphic lattice matching to general graphs.
\end{problem}

\subsubsection{Application examples}

\begin{example}\label{exa:8888888888}
(\textbf{Lobsters}) By Theorem \ref{thmc:5-equivalent-therems} and Theorem \ref{thmc:more-total-colorings}, each caterpillar $T$ admits each set-ordered odd-edge $W$-magic total labeling defined in Definition \ref{defn:basic-W-type-labelings} and Definition \ref{defn:group-total-labelings-definition}. So, we have lobsters obtained by adding leaves to caterpillars, by the methods introduced in \cite{Zhou-Yao-Chen2013} and \cite{Zhou-Yao-Chen-Tao2012}, these lobsters admit odd-graceful labelings and odd-edge $W$-magic total labelings.\qqed
\end{example}

\begin{example}\label{exa:8888888888}
(\textbf{Complement caterpillar-graphic lattice}) As a graph operation ``$(\bullet)$'' guarantees that each graph $G$ in the $(\bullet)$-operation caterpillar-graphic lattice $\textbf{\textrm{L}}(\textbf{\textrm{G}}(\bullet) \textbf{\textrm{T}}_{\textrm{cater}})$ defined in (\ref{eqa:caterpillae-graph-lattice}) is a caterpillar, then there are lobsters obtained by adding leaves to each caterpillar $G\in \textbf{\textrm{L}}(\textbf{\textrm{G}}(\bullet) \textbf{\textrm{T}}_{\textrm{cater}})$, we put these lobsters into a set, which is just a \emph{complement caterpillar-graphic lattice} of the caterpillar-graphic lattice $\textbf{\textrm{L}}(\textbf{\textrm{G}}(\bullet) \textbf{\textrm{T}}_{\textrm{cater}})$.\qqed
\end{example}

\begin{example}\label{exa:8888888888}
(\textbf{Topological authentication}) The ODD-GRACEFUL subdivision-algorithm shows that each caterpillar $T$ admits a set-ordered odd-graceful total labeling $f_1$. By Theorem \ref{thmc:5-equivalent-therems}, the caterpillar $T$ admits a set-ordered odd-edge edge-magic total labeling $f_2$, a set-ordered odd-edge edge-difference total labeling $f_3$, a set-ordered odd-edge felicitous-difference total labeling $f_4$, as well as a set-ordered odd-edge graceful-difference total labeling $f_5$. We define a set-coloring $F$ of a caterpillar $T$ as:

$F(u)=\{f_1(u),f_2(u),f_3(u),f_4(u),f_5(u)\}$ for each vertex $u\in V(T)$,

$F(xy)=\{f_1(xy),f_2(xy),f_3(xy),f_4(xy),f_5(xy)\}$ for each edge $xy\in E(T)$,\\
such that each edge $xy$ holds $f_i(xy)=\theta_i(f_i(x),f_i(y))$ for $i\in [1,5]$. Write the caterpillar admitting the set-coloring $F$ by $T^*$, we get a more complex Topcode-matrix $T_{code}(T^*)$. Furthermore, it makes people to get more complicated number-based string public-keys and number-based string private-keys, and brings more convenience and options for users.\qqed
\end{example}

\begin{thm}\label{thm:666666}
(\textbf{Coloring closure}) Since each tree admits a $W$-magic total coloring with $W$-magic $\in \{$graceful-difference, edge-magic, edge-difference, felicitous-difference$\}$, so a uniformly-$n^*$ $W$-magic graphic lattice is closure to the $W$-magic total coloring.
\end{thm}

\section{Graphic lattices based on $W$-magic total colorings}

We call some graphic lattices as \emph{linear-graphic lattices} if they were made by tree-bases and each element in them is a tree, otherwise \emph{non-linear graphic lattices}.

\subsection{Linear-graphic lattices}

\textbf{CONSTRUCTION algorithm-I.}

\textbf{Step-11.} Let $J_{i_1},J_{i_2},\dots ,J_{i_A}$ be a permutation of trees $a_1T_1^c,a_2T_2^c,\dots ,a_mT_m^c$ based on a tree-base $\textbf{\textrm{T}}^c$, where $A=\sum ^m_{k=1}a_k\geq 1$, and each $T_i^c$ admits an odd-edge $W$-magic total labeling (or coloring) $f_i$.

\textbf{Step-12.} Adding $n_{i_j}$ leaves to each tree $J_{i_j}$ for $j\in [1,A]$ produces a tree $H_{i_j}$, denoted as $H_{i_j}=\langle n_{i_j}[\ominus_e] J_{i_j}\rangle $, admitting an odd-edge $W$-magic total labeling (or coloring) $f_{i_j}$ induced from the odd-edge $W$-magic total labeling (or coloring) of the tree $J_{i_j}$, such that $f_{i_s}(V(H_{i_s}))\cap f_{i_{s+1}}(V(H_{i_{s+1}}))\neq \emptyset $ for $s\in [1,A-1]$.

\textbf{Step-13.} Vertex-coincide a vertex $x\in V(H_{i_s})$ with a vertex $y\in V(H_{i_{s+1}})$ into one vertex $x\odot y$, where $f_{i_s}(x)=f_{i_{s+1}}(y)$ for $s\in [1,A-1]$, the resultant tree is denoted as
\begin{equation}\label{eqa:w-magic-labeling-coloring}
\odot ^A_{j=1}H_{i_j}=\odot ^A_{j=1}\langle n_{i_j}[\ominus_e] J_{i_j}\rangle =\big [\odot ^m_{k=1}\ominus ^A_{j=1}\big ]a_kT_k^c
\end{equation} since each tree $H_{i_j}=\langle n_{i_j}[\ominus_e] J_{i_j}\rangle $ is the result of adding $n_{i_j}$ leaves to each tree $J_{i_j}$ for $j\in [1,A]$.

\vskip 0.4cm

By the CONSTRUCTION algorithm-I above, we get a \emph{linear-graphic lattice}
\begin{equation}\label{eqa:w-magic-linear-graph-lattices}
\textbf{\textrm{L}}_{W}(Z^0[\odot_v\ominus_e]\textbf{\textrm{T}}^c)=\left \{\big [\odot ^m_{k=1}\ominus ^A_{j=1}\big ]a_kT_k^c:~a_k\in Z^0,~T_k^c\in \textbf{\textrm{T}}^c\right \}
\end{equation} with $A=\sum ^m_{k=1}a_k\geq 1$.

There are the following facts:

(I-1) Each tree $T\in \textbf{\textrm{L}}_{W}(Z^0[\odot_v\ominus_e]\textbf{\textrm{T}}^c)$ admits a coloring $\lambda$ with each edge $uv\in E(T)$ holding $\lambda(uv)=f_{i_j}(uv)$ if this edge $uv$ is an edge of some tree $H_{i_j}$ defined in Step-12 of the CONSTRUCTION algorithm-I, that is, $\lambda(uv)$ is an odd integer. Because of $f_{i_j}(uv)$ satisfies one of the following forms:

(I-1-i) $f_{i_j}(uv)+|f_{i_j}(u)-f_{i_j}(v)|=a_{i_j}$;

(I-1-ii) $\big ||f_{i_j}(u)-f_{i_j}(v)|-f_{i_j}(uv)\big |=b_{i_j}$;

(I-1-iii) $\big |f_{i_j}(u)+f_{i_j}(v)-f_{i_j}(uv)\big |=c_{i_j}$; and

(I-1-vi) $f_{i_j}(u)+f_{i_j}(uv)+f_{i_j}(v)=d_{i_j}$, where $a_{i_j}$, $b_{i_j}$, $c_{i_j}$ and $d_{i_j}$ are non-negative integers.

Then the coloring $\lambda$ is called an \emph{odd-edge compound multiple-magic total coloring} of the tree $T$.

(I-2) Each tree $T\in \textbf{\textrm{L}}_{W}(Z^0[\odot_v\ominus_e]\textbf{\textrm{T}}^c)$ defined in Eq.(\ref{eqa:w-magic-linear-graph-lattices}) corresponds another graph $T^*$ admitting a \emph{twin odd-edge compound multiple-magic total coloring} $\mu$ such that $\mu(V(T^*))\cup \lambda(V(G))\subseteq [0,M]$, and we call the graph $T^*$ a \emph{twin odd-edge graph} of the tree $T$. All twin odd-edge graphs of trees of $\textbf{\textrm{L}}_{W}(Z^0[\odot_v\ominus_e]\textbf{\textrm{T}}^c)$ form a set $\textbf{\textrm{L}}^{tree}_{twin}$, and we call it \emph{twin odd-edge linear-graphic lattice} of the linear-graphic lattice $\textbf{\textrm{L}}_{W}(Z^0[\odot_v\ominus_e]\textbf{\textrm{T}}^c)$.

(13) Each tree $H\in \textbf{\textrm{L}}_{W}(Z^0[\odot_v\ominus_e]\textbf{\textrm{T}}^c)$ defined in Eq.(\ref{eqa:w-magic-linear-graph-lattices}) corresponds another graph $H^*$ forming a uniformly $M$-complement complex graphic matching $M_{\textrm{comp}}\langle H,H^*\rangle $ (refer to Problem \ref{problem:complement-complex-graphic-matching}). All such graphs $H^*$ form a \emph{$M$-uniform matching odd-edge graphic lattice}
\begin{equation}\label{eqa:uniform-matching-graph-lattices}
\textbf{\textrm{L}}(M_{\textrm{comp}})=\left \{M_{\textrm{comp}}\langle H,H^*\rangle:~H\in \textbf{\textrm{L}}_{W}(Z^0[\odot_v\ominus_e]\textbf{\textrm{T}}^c)\right \}
\end{equation} with $A=\sum ^m_{k=1}a_k\geq 1$.

\subsection{Complexity of linear-graphic lattices}

The Topcode-matrix $T_{code}(G)$ can distribute us $3q(G)!$ number-based strings, where $q(G)=|E(G)|$. Clearly, rewriting the Topcode-matrix $T_{code}(G)$ from one of these $3q(G)!$ number-based strings is quite difficult, and reconstructing the tree $G$ from $T_{code}(G)$ is impossible since
\begin{equation}\label{eqa:555555}
G=[\odot ^m_{k=1}\ominus ^A_{j=1}]a_kT_k^c=\odot ^A_{j=1}\langle n_{i_j}[\ominus_e] J_{i_j}\rangle =\odot ^A_{j=1}H_{i_j}
\end{equation} according to Eq.(\ref{eqa:w-magic-labeling-coloring}).

The above works are related with two unsolved problems: One is the \textbf{Subgraph Isomorphic Problem}, and another one is the \textbf{Integer Partition Problem}: $n_{i_j}=m_{i_j,1}+m_{i_j,2}+\cdots +m_{i_j,p_{i_j}}$ with with integers $m_{i_j,k}\geq 1$ and $p_{i_j}\geq 2$, and there are $(p_{i_j}!)$ methods to adding these $n_{i_j}$ leaves to a group of $p_{i_j}$ vertices selected from ${p(J_{i_j}) \choose p_{i_j}}$ groups of vertices of tree $J_{i_j}$ in Step-12 of the CONSTRUCTION algorithm-I, where $p(J_{i_j})$ is the number of vertices of tree $J_{i_j}$. Suppose that there are $S(n_{i_j})$ ways of partitioning integer $n_{i_j}$, then we have $\prod ^{A}_{i_j=1}S(n_{i_j})(p_{i_j}!){p(J_{i_j}) \choose p_{i_j}}$ methods to add leaves.

There are $A!$ permutations $J_{i_1},J_{i_2},\dots ,J_{i_A}$ of trees $a_1T_1^c,a_2T_2^c,\dots ,a_mT_m^c$ based on a tree-base $\textbf{\textrm{T}}^c$, such that each permutation produces at least one tree $G\in \textbf{\textrm{L}}_{W}(Z^0[\odot_v\ominus_e]\textbf{\textrm{T}}^c)$, since $\odot ^A_{j=1}H_{i_j}$ defined in Eq.(\ref{eqa:w-magic-labeling-coloring}) may make two or more trees of $\textbf{\textrm{L}}_{W}(Z^0[\odot_v\ominus_e]\textbf{\textrm{T}}^c)$. The notation $n(\odot ^A_{j=1}H_{i_j})$ is the number of different trees made by $\odot ^A_{j=1}H_{i_j}$. So, we claim that there are at least
\begin{equation}\label{eqa:555555}
N(A,\textbf{\textrm{T}}^c)=A!\cdot n(\odot ^A_{j=1}H_{i_j})\cdot \prod ^{A}_{i_j=1}S(n_{i_j})(p_{i_j}!){p(J_{i_j}) \choose p_{i_j}}
\end{equation}
trees of the linear-graphic lattice $\textbf{\textrm{L}}_{W}(Z^0[\odot_v\ominus_e]\textbf{\textrm{T}}^c)$ for each $A=\sum ^m_{k=1}a_k\geq 1$.

\subsection{Non-linear graphic lattices}

Let $\textbf{\textrm{H}}^c=\{H_1^c,H_2^c,\dots ,H_m^c\}$ be a \emph{non-tree base}, where there is at least one colored graph $H_i^c$ to be a non-tree graph, and each $H_i^c$ is connected and admits an odd-edge $W$-magic total labeling.

\vskip 0.4cm

\noindent \textbf{CONSTRUCTION algorithm-II.}

\textbf{Step-21.} Let $T_{i_1},T_{i_2},\dots ,T_{i_B}$ be a permutation of graphs $b_1H_1^c,b_2H_2^c,\dots ,b_mH_m^c$ based on a non-tree base $\textbf{\textrm{H}}^c$, where $B=\sum ^m_{k=1}b_k\geq 1$, and each graph $H_i^c$ is connected and admits an odd-edge $W$-magic total labeling (or coloring) $g_i$. Here, there is at least one connected graph $H_t^c$ to be not a tree with $b_t\neq 0$.

\textbf{Step-22.} Adding $m_{i_j}~(\geq 1)$ leaves to each connected graph $T_{i_j}$ for $j\in [1,B]$ produces a connected graph $G_{i_j}$, denoted as $G_{i_j}=\langle m_{i_j}[\ominus_e] T_{i_j}\rangle $, admitting an odd-edge $W$-magic total labeling (or coloring) $g_{i_j}$ induced from the odd-edge $W$-magic total labeling (or coloring) of the connected graph $T_{i_j}$, such that $g_{i_s}(V(G_{i_s}))\cap g_{i_{s+1}}(V(G_{i_{s+1}}))\neq \emptyset $ for $s\in [1,B-1]$.

\textbf{Step-23.} Vertex-coincide a vertex $w\in V(G_{i_s})$ with a vertex $z\in V(G_{i_{s+1}})$ into one vertex $w\odot z$, where $g_{i_s}(w)=g_{i_{s+1}}(z)$ for $s\in [1,B-1]$, the resultant graph is denoted as
\begin{equation}\label{eqa:w-magic-labeling-coloring-algorithm}
\odot ^B_{j=1}G_{i_j}=\odot ^B_{j=1}\langle m_{i_j}[\ominus_e] T_{i_j}\rangle =\big [\odot ^m_{k=1}\ominus ^B_{j=1}\big ]b_kH_k^c
\end{equation}

By the CONSTRUCTION algorithm-II above, we get a \emph{graphic lattice}
\begin{equation}\label{eqa:w-magic-graphic-lattices}
\textbf{\textrm{L}}_{W}(Z^0[\odot_v\ominus_e]\textbf{\textrm{H}}^c)=\left \{\big [\odot ^m_{k=1}\ominus ^B_{j=1}\big ]b_kH_k^c:~b_k\in Z^0,~H_k^c\in \textbf{\textrm{H}}^c\right \}
\end{equation} with $B=\sum ^m_{k=1}b_k\geq 1$, and there is at least one connected graph $H_t^c$ to be not a tree with $b_t\neq 0$.

Thereby, we have the following facts:

(II-1) Each connected graph $G\in \textbf{\textrm{L}}_{W}(Z^0[\odot_v\ominus_e]\textbf{\textrm{H}}^c)$ admits a coloring $\theta$ with each edge $uv\in E(G)$ holding $\theta(uv)=g_{i_j}(uv)$ if the edge $uv$ is an edge of some connected graph $T_{i_j}$ defined in Step-22 of the CONSTRUCTION algorithm-II, also, $\theta(uv)$ is an odd integer. Notice that the edge color $g_{i_j}(uv)$ satisfies one of the following forms:

(II-1-i) $g_{i_j}(uv)+|g_{i_j}(u)-g_{i_j}(v)|=\alpha_{i_j}$;

(II-1-ii) $\big ||g_{i_j}(u)-g_{i_j}(v)|-g_{i_j}(uv)\big |=\beta_{i_j}$;

(II-1-iii) $\big |g_{i_j}(u)+g_{i_j}(v)-g_{i_j}(uv)\big |=\gamma_{i_j}$; and

(II-1-vi) $g_{i_j}(u)+g_{i_j}(uv)+g_{i_j}(v)=\delta_{i_j}$, where $\alpha_{i_j}$, $\beta_{i_j}$, $\gamma_{i_j}$ and $\delta_{i_j}$ are non-negative integers.

So we call the coloring $\theta$ an \emph{odd-edge compound multiple-magic total coloring} of $G$.

(II-2) Each connected graph $G\in \textbf{\textrm{L}}_{W}(Z^0[\odot_v\ominus_e]\textbf{\textrm{H}}^c)$ corresponds a graph $G^*$ admitting a \emph{twin odd-edge compound multiple-magic total coloring} $\varphi$ such that $\varphi(V(G^*))\cup \theta(V(G))\subseteq [0,M]$, and we call the graph $G^*$ a \emph{twin graph} of the connected graph $G$. Such twin graphs $G^*$ form a set $\textbf{\textrm{L}}_{twin}$, called the \emph{twin-graphic lattice} of the graphic lattice $\textbf{\textrm{L}}_{W}(Z^0[\odot_v\ominus_e]\textbf{\textrm{H}}^c)$.

\section{Concluding remarks}

Four new colorings, called odd-edge graceful-difference total coloring, odd-edge edge-difference total coloring, odd-edge edge-magic total coloring, and odd-edge felicitous-difference total coloring, are introduced for producing twin-type $W$-magic graphic lattices. The randomly growing graphs admitting four new colorings can be constructed by the so-called RANDOMLY-LEAF-ADDING algorithms and techniques of adding leaves to continuously.

The uniformly $W$-magic total colorings help us to build up caterpillar-graphic lattices and complementary graphic lattices. One important work is to show a connection between complex graphs and integer lattices, which establishes indirectly some connections between our graphic lattices and some integer lattices of traditional lattices.

Some research topics in this article can be further developed. For example, suppose that a tree $H$ has its own vertex set $V(H)=\{x_1,x_2,\dots ,x_n\}$, we add a leaf set $L_{\textrm{eaf}}(x_i)=\{y_{i, j}:j\in [1, M_i]\}$ to each vertex $x_i\in V(H)$ for $i\in [1,n]$, the resultant tree is denoted as $H^*=H\ominus L_{\textrm{eaf}}$, where $L_{\textrm{eaf}}=\bigcup ^n_{i=1}L_{\textrm{eaf}}(x_i)$. Notice that the tree $H^*$ has its own vertex set $V(H^*)=V(H)\cup L_{\textrm{eaf}}$ and its own edge set $E(H^*)=E(H)\cup E(L_{\textrm{eaf}})$, where $E(L_{\textrm{eaf}})=\bigcup^n_{i=1} E(L_{\textrm{eaf}}(x_i))$ with $E(L_{\textrm{eaf}}(x_i))=\{x_iy_{i,j}:j\in [1, M_i]\})$.

There is a group of trees $T_1,T_2,\dots ,T_A$, in which each tree $T_k$ with $k\in [1,A]$ has its own leaf sets $L_{\textrm{eaf}}(x_{k,i})=\{u_{k,i, j}:j\in [1, m_{k,i}]\}$ for $x_{k,i}\in V(T_k)$ with $i\in [1,n]$, such that removing all leaves from the tree $T_k$ produces a tree $T^*_k=T_k-\bigcup^n_{i=1} L_{\textrm{eaf}}(x_{k,i})=H$, so
\begin{equation}\label{eqa:555555}
V(H)=\{x_1,x_2,\dots ,x_n\}=V(T^*_k)=\{x_{k,1},x_{k,2},\dots ,x_{k,n}\}
\end{equation} We call $H$ the \emph{leaf-core} of each tree $T_k$ for $k\in [1,A]$.

Notice that it is allowed for $L_{\textrm{eaf}}(x_i)=\emptyset$ for some vertex $x_i\in V(H)$, or $L_{\textrm{eaf}}(x_{k,i})=\emptyset$ for some vertex $x_{k,i}\in V(T_k)$.

\begin{asparaenum}[$\bullet$ ]
\item For two trees $T_k$ and $T_r$, if there exists a positive integer $B_{k,r}$, such that $|L_{\textrm{eaf}}(x_{k,i})|+|L_{\textrm{eaf}}(x_{r,i})|=B_{k,r}$ for $i\in [1,n]$, we call two trees $T_k$ and $T_r$ to be \emph{uniform $B_{k,r}$-leaf complementary trees}.
\item For two trees $T_s$ and $T_j$, if there exists a positive integer $C_{s,j}$, such that $|L_{\textrm{eaf}}(x_{s,i})|+|L_{\textrm{eaf}}(x\,'_{j,i})|=C_{s,j}$ for $i\in [1,n]$, we call two trees $T_k$ and $T_r$ to be \emph{$C_{s,j}$-leaf complementary trees}, where $x\,'_{j,1},x\,'_{j,2},\dots ,x\,'_{j,n}$ is a permutation of vertices of the tree $T^*_j=T_j-\bigcup^n_{i=1} L_{\textrm{eaf}}(x_{j,i})$.
\item Since
\begin{equation}\label{eqa:555555}
|L_{\textrm{eaf}}(x_i)|=M_i=\sum^A_{k=1} m_{k,i}=\sum^A_{k=1}|L_{\textrm{eaf}}(x_{k,i})|
\end{equation} we call the tree $H^*$ \emph{universal tree} of the group of trees $T_1,T_2,\dots ,T_A$.
\end{asparaenum}
From graph operation of view, the universal tree $H^*$ is the result of coinciding the leaf-cores of trees $T_1,T_2,\dots ,T_A$ into one.

Since each tree admits an odd-edge $W$-magic total coloring for each $W$-magic $\in \{$edge-magic, edge-difference, felicitous-difference, graceful-difference$\}$, let $f$ be an odd-edge $W$-magic total coloring of the universal tree $H^*$, then each tree $T_k$ admits a total coloring $f_k$ induced by $f$ such that $f_k(w)\in f(V(H^*)\cup E(H^*))$ for $w\in V(T_k)\cup E(T_k)$, that is, $f_k(V(T_k)\cup E(T_k))\subset f(V(H^*)\cup E(H^*))$.

Thereby, the above tree $H^*$ is the \emph{topological authentication} of the trees $T_1,T_2,\dots ,T_A$ if some tree $T_k$ is considered as a public-key, and others are private-keys in applications of blockchain, financial networks, digital currency.

\section{Acknowledgment}

This work was supported by the National Natural Science Foundation of China (No. 61163054, No. 61363060, No. 61662066), Northwest China Financial Research Center project of Lanzhou University of Finance and Economics (JYYZ201905), and the Research project of higher education teaching reform in Lanzhou University of Finance and Economics (LJZ202012).



{\footnotesize
\begin{thebibliography}{9}

\setlength{\parskip}{0pt}
\bibitem{Bondy-2008} J. A. Bondy, U. S. R. Murty. Graph Theory. Springer London, 2008. DOI: 10.1007/978-1-84628-970-5. J. Adrian Bondy and U. S. R. Murty, Graph Theory with Application. The MaCmillan Press Ltd., London, 1976.
\bibitem{Gallian2021} Joseph A. Gallian. A Dynamic Survey of Graph Labeling. The electronic journal of combinatorics, \# DS6, Twenty-fourth edition, December 9, 2021. (576 pages, 3028 reference papers, over 200 graph labelings)
\bibitem{Tian-Li-Peng-Yang-2021-102212}Yanzhao Tian, Lixiang Li, Haipeng Peng and Yixian Yang. Achieving flatness: Graph labeling can generate graphical honeywords. Computers and Security, \textbf{104} (2021) 102212. DOI: 10.1016/j.cose.2021.102212.


\bibitem{Peikert-Lattice-Cryptography-Internet2014}Peikert, C. Lattice Cryptography for the Internet. In: Mosca, M. (eds) Post-Quantum Cryptography. PQCrypto 2014. Lecture Notes in Computer Science, vol 8772. Springer, Cham. DOI:10.1007/978-3-319-11659-4-12.
\bibitem{Chris-Peikert-decade} Chris Peikert. A Decade of Lattice Cryptography. Found. Trends Theor. Comput. Sci. 10(4) (2016): 283-424. (90 pages)
\bibitem{Tim-Vadim-Thomas-2012-978-3-642-33027} Tim G\"{u}neysu, Vadim Lyubashevsky and Thomas P\"{o}ppelmann. Practical Lattice-Based Cryptography: A Signature Scheme for Embedded Systems. International Workshop on Cryptographic Hardware and Embedded Systems, CHES 2012: Cryptographic Hardware and Embedded Systems, CHES 2012, pp 530-547. DOI: 10.1007/978-3-642-33027-8-31.


\bibitem{Yao-Su-Ma-Wang-Yang-arXiv-2202-03993v1}Bing Yao, Jing Su, Fei Ma, Hongyu Wang, Chao Yang. Topological Authentication Technique In Topologically Asymmetric Cryptosystem. arXiv:2202.03993v1 [cs.CR] 8 Feb 2022.
\bibitem{Wang-Yao-Su-Wanjia-Zhang-2021-IMCEC} Hongyu Wang, Bing Yao, Jing Su, Wanjia Zhang. Number-Based Strings/Passwords From Imaginary Graphs Of Topological Coding For Encryption. submitted to IMCEC 2021.

\bibitem{Bing-Yao-Hongyu-Wang-arXiv-2020-homomorphisms}Bing Yao, Hongyu Wang. Graph Homomorphisms Based On Particular Total Colorings of Graphs and Graphic Lattices. arXiv: 2005.02279v1 [math.CO] 5 May 2020.
\bibitem{Bing-Yao-2020arXiv}Bing Yao. Graphic Lattices and Matrix Lattices Of Topological Coding. arXiv: 2005.03937v1 [cs.IT] 8 May 2020.
\bibitem{Yao-Wang-Liu-ice-flower-2020arXiv}Bing Yao, Hongyu Wang, Xia Liu, Xiaomin Wang, Fei Ma, Jing Su, Hui Sun. Ice-Flower Systems And Star-graphic Lattices. arXiv:2005.02823v1 [Math.CO] 6 May 2020.
\bibitem{Yao-Wang-Su-Sun-ITOEC2020}Bing Yao, Hongyu Wang, Jing Su, Hui Sun. Graphic Lattices For Constructing High-Quality Networks. 2020 IEEE 5th Information Technology and Mechatronics Engineering Conference (ITOEC 2020) 1726-1730.
\bibitem{yao-sun-su-wang-matching-groups-zhao-2020} Bing Yao, Hui Sun, Jing Su, Hongyu Wang, Meimei Zhao. Various Matchings of Graphic Groups For Graphic Lattices In Topological Coding. 2020 IEEE International Conference On Information Technology, Big Data And Artificial Intelligence (ICIBA 2020): 332-337.
\bibitem{Yao-Zhao-Zhang-Mu-Sun-Zhang-Yang-Ma-Su-Wang-Wang-Sun-arXiv2019}Bing Yao, Meimei Zhao, Xiaohui Zhang, Yarong Mu, Yirong Sun, Mingjun Zhang, Sihua Yang, Fei Ma, Jing Su, Xiaomin Wang, Hongyu Wang, Hui Sun. Topological Coding and Topological Matrices Toward Network Overall Security. arXiv:1909.01587v2 [cs.IT] 15 Sep 2019.
\bibitem{Yao-Zhang-Sun-Mu-Sun-Wang-Wang-Ma-Su-Yang-Yang-Zhang-2018arXiv}Bing Yao, Xiaohui Zhang, Hui Sun, Yarong Mu, Yirong Sun, Xiaomin Wang, Hongyu Wang, Fei Ma, Jing Su, Chao Yang, Sihua Yang, Mingjun Zhang. Text-based Passwords Generated From Topological Graphic Passwords. arXiv: 1809. 04727v1 [cs.IT] 13 Sep 2018.
\bibitem{Yao-Sun-Zhang-Mu-Sun-Wang-Su-Zhang-Yang-Yang-2018arXiv}Bing Yao, Hui Sun, Xiaohui Zhang, Yarong Mu, Yirong Sun, Hongyu Wang, Jing Su, Mingjun Zhang, Sihua Yang, Chao Yang. Topological Graphic Passwords And Their Matchings Towards Cryptography. arXiv: 1808. 03324v1 [cs.CR] 26 Jul 2018.

\bibitem{Yao-Sun-Zhao-Li-Yan-ITNEC-2017} Bing Yao, Hui Sun, Meimei Zhao, Jingwen Li, Guanghui Yan. On Coloring/Labelling Graphical Groups For Creating New Graphical Passwords. (ITNEC 2017) 2017 IEEE 2nd Information Technology, Networking, Electronic and Automation Control Conference. 2017: 1371-1375.

\bibitem{Bing-Yao-Cheng-Yao-Zhao2009}Bing Yao, Hui Cheng, Ming Yao and Meimei Zhao. A Note on Strongly Graceful Trees. Ars Combinatoria \textbf{92} (2009), 155-169.


\bibitem{yirong-sun-bing-yao-2022-nwnu}Yirong Sun, Bing Yao. Exploring the practice and application of topology coding. Journal of Northwest Normal University (Chinese), 2022.



\bibitem{Wang-Xu-Yao-2016} Hongyu Wang, Jin Xu, Bing Yao. Exploring New Cryptographical Construction Of Complex Network Data. IEEE First International Conference on Data Science in Cyberspace. IEEE Computer Society, (2016):155-160.
\bibitem{Wang-Xu-Yao-2017-Twin}Hongyu Wang, Jin Xu, Bing Yao. Twin Odd-Graceful Trees Towards Information Security. Procedia Computer Science \textbf{107} (2017) 15-20. DOI: 10.1016/j.procs.2017.03.050

\bibitem{Wang-Xu-Yao-2017-Matching}Hongyu Wang, Jin Xu, Bing Yao. On Magic Matching Sets Having Many Magic Parameters In Information Networks. 2017 IEEE 3rd Information Technology and Mechatronics Engineering Conference (IEEE ITOEC 2017): 77-80.


\bibitem{Shuhong-Wu-Accurate-2007}Shuhong Wu. The Accurate Formulas of $A(n,k)$ and $P(n,k)$. Journal Of Mathematical Research And Exposition, \textbf{27}, NO.2 (2007) 437-444.

\bibitem{WU-Qi-qi-2001}Qiqi Wu. On the law of left shoulder and law of oblique line for constructing a large table of $P(n,k)$ quickly (in Chinese). Sinica (Chin. Ser.), 2001, \textbf{44} (5): 891-897.


\bibitem{Zhang-Yang-Yao-Frontiers-Computer-2021}Mingjun Zhang, Sihua Yang, Bing Yao. Exploring Relationship Between Traditional Lattices and Graph Lattices of Topological Coding. Journal of Frontiers of Computer Science and Technology (Chinese). 2021, 15(11) 1-13. doi: 10.3778/j.issn.1673-9418.2010072.

\bibitem{Zhang-Zhang-Yao-integer-lattices-2022} Mingjun Zhang, Xiaohui Zhang and Bing Yao. Exploring The Topological Coding Realization of Integer Lattices (in Chinese). submitted 2022.



\bibitem{Zhou-Yao-Chen2013}Xiangqian Zhou, Bing Yao, Xiang'en Chen. Every Lobster Is Odd-elegant. Information Processing Letters \textbf{113} (2013): 30-33.
\bibitem{Zhou-Yao-Chen-Tao2012}Xiangqian Zhou, Bing Yao, Xiang'en Chen and Haixia Tao. A proof to the odd-gracefulness of all lobsters. Ars Combinatoria \textbf{103} (2012): 13-18.




\end{thebibliography}


}
\end{document}